\documentclass[12pt]{article}
\usepackage[affil-it]{authblk}
\usepackage[utf8]{inputenc}
\usepackage[english]{babel}
\usepackage{titlesec}
\usepackage[margin=0.5in]{geometry}
\usepackage{mathtools}
\usepackage{enumerate}
\usepackage{adjustbox}
\usepackage{amssymb}
\usepackage{titlesec}
\usepackage[utf8]{inputenc}
\usepackage{mathrsfs}
\usepackage{cite}
\def\dss{\displaystyle}
\usepackage{enumitem}
\usepackage[square,numbers]{natbib}

\usepackage{floatrow}
\newfloatcommand{capbtabbox}{table}[][\FBwidth]

\usepackage{blindtext}
\usepackage[hyphens]{url}
\usepackage{hyperref}
\usepackage{float}
\usepackage[section]{placeins}
\usepackage[dvipsnames]{xcolor}
\usepackage{soul}
\usepackage{amsmath}
\usepackage[labelfont=bf,font=small]{caption}
\usepackage{bm}
\usepackage[title]{appendix}
\usepackage{graphicx}
\usepackage{tabularx}
\usepackage{subcaption}
\usepackage{caption}
\usepackage{float}
\usepackage{mathtools}
\usepackage{amsmath,amssymb,amsthm}
\usepackage{oubraces}
\numberwithin{equation}{section}
\usepackage{chngcntr}
\usepackage{mathrsfs}
\usepackage{multirow}

\newtheorem{theorem}{Theorem}[section]

\newtheorem{definition}[theorem]{Definition}
\newtheorem{proposition}[theorem]{Proposition}
\newcommand{\addappendix}{%
  \section*{\appendixname}
  \addcontentsline{toc}{section}{\appendixname}
  \counterwithin*{figure}{section}
  \stepcounter{section}
  \renewcommand{\thesection}{A}
  \renewcommand{\thefigure}{\thesection.\arabic{figure}}
}

\titleformat{\noindentagraph}
{\normalfont\normalsize\bfseries}{\theparagraph}{1em}{}
\titlespacing*{\noindentagraph}
{0pt}{1ex plus 4pt minus 2pt}{0pt}
\titlespacing\section{0pt}{0ex}{0pt}
\titlespacing\subsection{0pt}{1ex plus 4pt minus 2pt}{0pt}
\titlespacing\subsubsection{0pt}{1ex plus 4pt minus 2pt}{0pt}

\graphicspath{{Figures/}}

\newcommand{\RR}{\mathbb{R}}

\begin{document}
\title{Modeling the impact of hospitalization-induced behavioral changes on SARS-COV-2 spread in New York City}
\author{Alice Oveson$^{a}$, Michelle Girvan$^{b}$, and Abba B. Gumel$^{a,c,}$\footnote{Corresponding author: Email: {\it agumel@umd.edu}}
\begin{flushleft}
{\it {\small $^{a}$  Department of Mathematics, University of Maryland, College Park, MD, 20742, USA}}\newline
{\it {\small $^{b}$  Department of Physics, University of Maryland, College Park, MD, 20742, USA}}\newline
{\it{\small $^{c}$ Department of Mathematics and Applied Mathematics, University of Pretoria, Pretoria, 0002, South Africa}}
\end{flushleft}
\vspace{-1.2cm}
    }
\date{\today}
\maketitle
\vspace{-1.2cm}
\noindent
\begin{abstract}
\noindent The COVID-19 pandemic, caused by SARS-CoV-2, highlighted heterogeneities in human behavior and attitudes of individuals with respect to adherence or lack thereof to public health-mandated intervention and mitigation measures. This study is based on using mathematical modeling approaches, backed by data analytics and computation, to theoretically assess the impact of human behavioral changes on the trajectory, burden, and control of the SARS-CoV-2 pandemic during the first two waves in New York City. A novel behavior-epidemiology model, which considers $n$ heterogeneous behavioral groups based on level of risk tolerance and distinguishes behavioral changes by social and disease-related motivations (such as peer-influence and fear of disease-related hospitalizations), is developed. In addition to rigorously analyzing the basic qualitative features of this model, a special case is considered where the total population is stratified into two groups: risk-averse (Group 1) and risk-tolerant (Group 2). The two-group behavior model has three equilibria in the absense of disease, and their stability is analyzed using standard linearization and the properties of Metzler-stable matrices.  Furthermore, the two-group model was calibrated and validated using daily hospitalization data for New York City during the first wave, and the calibrated model was used to predict the data for the second wave. The two-group model predicts the daily hospitalizations during the second wave almost perfectly, compared to the version without behavioral considerations, which fails to accurately predict the second wave. This suggests that epidemic models of the SARS-CoV-2 pandemic that do not explicitly account for heterogeneities in human behavior may fail to accurately predict the trajectory and burden of the pandemic in a population. Numerical simulations of the calibrated two-group behavior model showed that while the dynamics of the SARS-CoV-2 pandemic during the first wave was largely influenced by the behavior of the risk-tolerant (Group 2) individuals, the dynamics during the second wave was influenced by the behavior of individuals in both groups. It was also shown that disease-motivated behavioral changes (i.e., behavior changes due to the level of SARS-CoV-2 hospitalizations in the community) had greater influence in significantly reducing SARS-CoV-2 morbidity and mortality than behavior changes due to the level of peer or social influence or pressure. Finally, it is shown that the initial proportion of members in the community that are risk-averse (i.e., the proportion of individuals in Group 1 at the beginning of the pandemic) and the early and effective implementation of non-pharmaceutical interventions have major impacts in reducing the size and burden of the pandemic (particularly the total SARS-CoV-2 mortality in New York City during the second wave). 

\end{abstract}

{\it Keywords:} behavioral-epidemiology model; SARS-CoV-2; equilibria; influence dynamics.
\section{Introduction}
\noindent The novel pneumonia-like illness that emerged in Wuhan, China in December 2019, coined the 2019 coronavirus (COVID-19) and caused by SARS-CoV-2, spread rapidly within China and to other parts of the world, prompting the World Health Organization to declare a global pandemic of COVID-19 in March 2020. With over 775 million cases and over 7 million deaths globally \cite{WHOdashboard}, the toll of COVID-19 is long-lasting. It caused many healthcare workers to leave the field, leading to long-term staffing shortages \cite{USBLS, ASPE} and downturns in the global economy \cite{econ1, econ2, econ3}, education quality \cite{education1, education2}, and mental health across the world \cite{mh1,mh2}. Additionally, it resulted in dramatic increases in the level of political polarization, inequity and disparity in access to healthcare \cite{HC}, the spread of mis(dis)information \cite{info1, info2}, anti-vaccination movements \cite{antivax1, antivax2}, distrust of governments and public health agencies \cite{trust1, trust2, trust3}, and incidence of domestic violence \cite{dv1, dv2, dv3}. Four years into the pandemic (November 2024), nearly one thousand people die each week globally of SARS-CoV-2, despite the widespread availability of life-saving vaccines \cite{WHOdashboard}. \\
\indent The behavioral response to the SARS-CoV-2 pandemic in the US and globally was highly heterogeneous \cite{het1, het2}, largely due to individual-level differences (such as heterogeneities with respect to the quality of the pandemic-related information received\cite{gg_antecedents_uscinski}, economic status and amount of savings \cite{gg_sociopolitcal_stoler}, trust of government and health agencies \cite{gg_sociopolitcal_stoler}, and social pressures \cite{gg_sociocultural_yang}). This behavioral heterogeneity played a major role in shaping the trajectory and burden of the pandemic \cite{gg_adaptivecontact_arthur, gg_plateaus_berestycki, gg_simplebehavior_lejune}. Numerous modeling types,(including compartmental \cite{gg_diffusion_berestycki, gg_asymptomatic_espinoza, gg_heterogeneous_espinoza, gg_interplay_nguyen, gg_optimal_pisaneschi, gg_perioddoubling_sharbayta, gg_stochastic_ochab, gg_simplebehavior_lejune, ge_1918flu_he, ge_behavior_binod, gg_plateaus_weitz, gg_adaptivecontact_arthur, gg_saadroy, gg_flu_valle, gg_biases_weitz}, agent-based \cite{gg_abm_lapoirie, gg_sociocultural_yang, gg_abm_kitson, gg_abm_chen, gg_flu_valle}, network \cite{gg_networks_qiu, gg_abm_kitson, gg_abm_chen, ge_jewish_pearson, gg_aware_funk, gg_empathy_weitz}, etc.,) have been used to gain insight and understanding on the role of human behavior changes (or heterogeneities) on the spread and control of the SARS-CoV-2 pandemic. These models typically assume that individual disease-related behavior change (such as adherence or non-adherence to public health interventions, reducing contacts, etc.) is motivated by a combination of disease-related metrics (such as the level of disease prevalence and intervention fatigue in their local community or population) \cite{gg_interplay_nguyen, gg_perioddoubling_sharbayta, ge_behavior_binod, gg_saadroy}, peer pressure or influence \cite{gg_abm_lapoirie, gg_sociocultural_yang, gg_networks_qiu, ge_behavior_binod}, and other costs associated with adoption of disease-related precautionary  behaviors). For example, the compartmental model developed by Pisaneschi et al. \cite{gg_optimal_pisaneschi} accounts for average daily wages, assuming that a lower wage will force individuals to abandon disease-related precautionary behaviors. These behavioral heterogeneities and changes are typically modeled by incorporating an adaptive transmission rate that is dependent on disease prevalence \cite{gg_heterogeneous_espinoza, gg_simplebehavior_lejune, gg_biases_weitz} or by stratifying the members of the population into various mutually-exclusive compartments based on their disease-related behavioral choices or attitudes \cite{gg_asymptomatic_espinoza, gg_heterogeneous_espinoza, gg_sociocultural_yang, ge_behavior_binod, gg_saadroy} (i.e., a higher transmission rate is assumed for the individuals in the risk-taking compartment, in comparison to those in the compartment for risk-averse individuals). \\
\indent Prior modeling studies showed that behavioral heterogeneities and changes occur at a group level (i.e., the whole population becomes more or less cautious, in comparison to its behavioral choices before the onset of the disease or during an earlier phase of its progression) \cite{gg_simplebehavior_lejune, gg_heterogeneous_espinoza} and at an individual level (i.e., certain individuals will alter their past behavior or attitude with respect to the disease, despite the trends or choices made by the general population) \cite{ge_behavior_binod, gg_networks_qiu, gg_sociocultural_yang}, and that both types of behavior changes hold strong influence on the trajectory and burden of the pandemic\cite{gg_heterogeneous_bansal}. For example, Espinoza et al. \cite{gg_heterogeneous_espinoza} used a game-theoretic compartmental model (with a utility function that influences the contact rate within the population), which stratifies the susceptible, exposed, and asymptomatic portions of the population based on compliance or lack thereof to interventions, and showed that heterogeneous responses based on perceived risk (calculated using only the number of symptomatically-infectious individuals) can increase the final size of the SARS-CoV-2 pandemic. Furthermore, several other modeling studies confirmed that asymptomatic transmission was a major driver of the COVID-19 pandemic \cite{ge_behavior_binod, gg_asymptomatic_espinoza, lockdownpaper, asymp1, asymp2}, resulting in large disparity between true and perceived disease prevalence (which changed people's behavior or attitude towards the pandemic \cite{gg_asymptomatic_espinoza}). Specifically, Pant et al. \cite{ge_behavior_binod} showed a strong correlation between the proportion of newly-infected individuals who become asymptomatically infectious and decreases in positive behavior changes with respect to (adherence to) public health control and mitigation measures. Thus, these studies clearly show that the effect of human behavior changes during the SARS-CoV-2 pandemic can be incorporated into epidemiological models using measurable and publicly-available health metrics (such as the number of new confirmed/symptomatic cases, hospitalizations, and disease-induced mortality) and sociological metrics (e.g., intervention fatigue) \cite{ge_behavior_binod}.\vspace*{2mm} \ \\ 
\noindent The objective of the current study is to assess the population-level impact of behavioral changes, with respect to levels of hospitalization, on the transmission dynamics and control of the SARS-CoV-2 pandemic in New York City. To achieve this objective, a multi-group compartmental model is developed, which subdivides the total population into $n$ heterogeneous groups that differ based on their behavior and/or attitude towards the disease. Specifically, the model takes into account how adverse group members are to acquiring (if they are not infected or unaware of their infection) or transmitting (if they are infected and aware of their infection) the disease. The model to be developed, which consists of a deterministic system of nonlinear differential equations, will be parameterized using observed SARS-CoV-2 hospitalization data in New York City during the first wave of the pandemic. Furthermore, group-level behavioral change will be accounted for in the model by incorporating a time-varying parameter that accounts for the relative reduction in individual contact rates. This parameter will be formulated to have a nonlinear dependence on perceived risk of acquisition and transmission of infection, measured by the proportion of hospitalized individuals in the community (New York City). The paper is organized as follows. The model (with $n$ heterogeneous behavioral groups), tagged {\it the behavior model}, is formulated in Section \ref{formulation}. For the special case with $n=2$ behavioral groups, the stability of the disease-free equilibria is analyzed in Section \ref{sec:stab_analysis}. The model with $n=2$ behavioral groups is fitted and cross-validated with observed data for the first wave of the SARS-CoV-2 pandemic in New York City in Section \ref{sec:fitting}. To assess the impact of behavior on accurately capturing the correct trajectory and burden of the pandemic, the behavior-free equivalent of the model is also fitted and cross-validated for comparison purposes (specifically, the objective is to determine which of the two models (behavior vs. behavior-free) fits and cross-validates the observed data more accurately). Global parameter sensitivity analysis and numerical simulations of the behavior model are reported in Sections \ref{sec:sensitivity} and \ref{sec:numsim}, respectively. Discussion and concluding remarks are presented in Section \ref{discussion}.

\section{Formulation of behavior-epidemiology model}\label{formulation}
\noindent The focus is to develop a mechanistic model for assessing the population-level impact of hospitalization-induced and peer pressure-induced behavior changes on the transmission dynamics of the SARS-CoV-2 pandemic in New York City during the first wave (which runs from late February 2020 to late July 2020 \cite{CDC_NYCstats}. The model is formulated by subdividing the total population at time $t$, denoted by $N(t)$, into $n$ heterogeneous behavioral groups based on their perceived risk (of acquisition/transmission of infection) and their resulting behavior choices with respect to the risk. The first group (referred to as Group 1) represents individuals in the community who consistently adopt the most risk-averse behaviors against the acquisition or transmission of the disease by consistently and strictly adhering to public health intervention and mitigation measures. Similarly, Groups 2 through $n$ represent sub-groups within the community who adopt more risk-taking behaviors, with Group $n$ representing the individuals in the community who adopt the most risk tolerant behaviors \cite{gg_interplay_nguyen}. Furthermore, each of the $n$ heterogeneous behavioral groups is further subdivided into the mutually-exclusive compartments of susceptible ($S_i(t)$), exposed ($E_i(t))$, asymptomatically-infectious ($I_{a,i}(t)$), symptomatically-infectious ($I_{s,i}(t)$), hospitalized ($I_{h,i}(t)$) and recovered ($R_i(t)$) individuals (with $i=1,2,\cdots, n$), so that:
$$N_i(t)=S_i(t)+E_i(t)+I_{a,i}(t)+I_{s,i}(t)+I_{h,i}(t)+R_i(t).$$  
\noindent 
The equations for the rate of change of each of the aforementioned epidemiological compartments (for each of the $n$ behavioral sub-groups) are given by the following deterministic system of nonlinear differential equations (where a dot represents differentiation with respect to time $t$, and $i=1,2,\,\cdots,n$):

\begin{align}
    \begin{split}
\Dot{S_i}(t) &= \xi R_i(t)-\left(\theta_l\right)\left[\dfrac{c_i^A(t)S_i(t)}{N(t)}\right]\sum_{j=1}^nc_j^A(t) \left[\beta_aI_{a,j}(t)+\beta_iI_{s,j}(t)+\beta_hI_{h,j}(t)\right] \\ 
        &+\frac{S_{i-1}(t)}{N(t)}\sum_{j=i}^n c_{j(i-1)}^BN_j(t) + \frac{S_{i+1}(t)}{N(t)}\sum_{j=1}^i c_{j(i+1)}^BN_j(t)-\frac{S_i(t)}{N(t)}\sum_{\substack{j=1\\
                  j\neq i}}^n c_{ji}^BN_j(t) \\
\Dot{E_i(t)} &=\left(\theta_l\right)\left[\dfrac{c_i^A(t)S_i(t)}{N(t)}\right]\sum_{j=1}^nc_j^A(t)\left[\beta_aI_{a,j}(t)+\beta_iI_{s,j}(t)+\beta_hI_{h,j}(t)\right] -\sigma_eE_i(t) \\
        & + \dfrac{E_{i-1}(t)}{N(t)}\sum_{j=i}^n c_{j(i-1)}^BN_j(t) + \dfrac{E_{i+1}(t)}{N(t)}\sum_{j=1}^i c_{j(i+1)}^BN_j(t)-\dfrac{E_i(t)}{N(t)}\sum_{\substack{j=1\\
                  j\neq i}}^n c_{ji}^BN_j(t)  \\
\Dot{I}_{a,i}(t) &= (1-r)\sigma_eE_i(t)-\gamma_aI_{a,i}(t)  \\
        & + \dfrac{I_{a,i-1}(t)}{N(t)}\sum_{j=i}^n c_{j(i-1)}^BN_j(t) + \dfrac{I_{a,i+1}(t)}{N(t)}\sum_{j=1}^i c_{j(i+1)}^BN_j(t)-\dfrac{I_{a,i}(t)}{N(t)}\sum_{\substack{j=1\\
                  j\neq i}}^n c_{ji}^BN_j(t) \\
\Dot{I}_{s,i}(t) &= r\sigma_eE_i(t)-\sigma_iI_{s,i}(t)   \\
        & + \dfrac{I_{s,i-1}(t)}{N(t)}\sum_{j=i}^n c_{j(i-1)}^BN_j(t) + \dfrac{I_{s,i+1}(t)}{N(t)}\sum_{j=1}^i c_{j(i+1)}^BN_j(t)-\dfrac{I_{s,i}(t)}{N(t)}\sum_{\substack{j=1\\
                  j\neq i}}^n c_{ji}^BN_j(t)  \\
\Dot{I}_{h,i}(t) &= q\sigma_iI_{s,i}(t)-(\gamma_h + \delta_h)I_{h,i}(t)  \\
        & + \dfrac{I_{h,i-1}(t)}{N(t)}\sum_{j=i}^n c_{j(i-1)}^BN_j(t) + \dfrac{I_{h,i+1}(t)}{N(t)}\sum_{j=1}^i c_{j(i+1)}^BN_j(t)-\dfrac{I_{h,i}(t)}{N(t)}\sum_{\substack{j=1\\
                  j\neq i}}^n c_{ji}^BN_j(t)  \\
\Dot{R_i}(t) &= \gamma_aI_{a,i}(t)+(1-q)\sigma_iI_{s,i}(t)+\gamma_hI_{h,i}(t) -\xi R_i(t) \\
        & + \dfrac{R_{i-1}(t)}{N(t)}\sum_{j=i}^n c_{j(i-1)}^BN_j(t) + \dfrac{R_{i+1}(t)}{N(t)}\sum_{j=1}^i c_{j(i+1)}^BN_j(t)-\dfrac{R_i(t)}{N(t)}\sum_{\substack{j=1\\
                  j\neq i}}^n c_{ji}^BN_j(t)
    \end{split}\label{eq:fullmodel}
\end{align}
\noindent A flow diagram of the heterogeneous $n-$group behavior-epidemiology model \eqref{eq:fullmodel} is depicted in Figure \ref{fig:dynamics}, and the state variables and parameters of the model are described in Tables \ref{table:state_var} and \ref{table:parameters}, respectively.
\begin{table}[ht!] 
  \begin{center}
     \begin{tabular}{|p{2cm}|p{14cm}|}
     \hline
      \textbf{State variable} & \textbf{Description}  \\
          \hline \hline
      $S_i$ & Population of susceptible individuals in  behavioral group $i$ \\ \hline
      $E_i$ & Population of exposed individuals in behavioral group $i$  \\ \hline
      $I_{a,i}$ & Population of asymptomatically-infectious individuals in behavioral group $i$ \\ \hline
      $I_{s,i}$ & Population of symptomatically-infectious individuals in behavioral group $i$ \\ \hline
      $I_{h,i}$ & Population of hospitalized individuals in behavioral group $i$  \\ \hline
      $R_i$ & Population of recovered individuals in behavioral group $i$  \\ \hline
      \end{tabular}
    \caption{Description of the state variables of the heterogeneous $n-$group behavior model \eqref{eq:fullmodel}}
    \label{table:state_var}
\end{center}
\end{table}
\begin{figure}[H]
    \centering
    \includegraphics[width=145mm,scale=1]{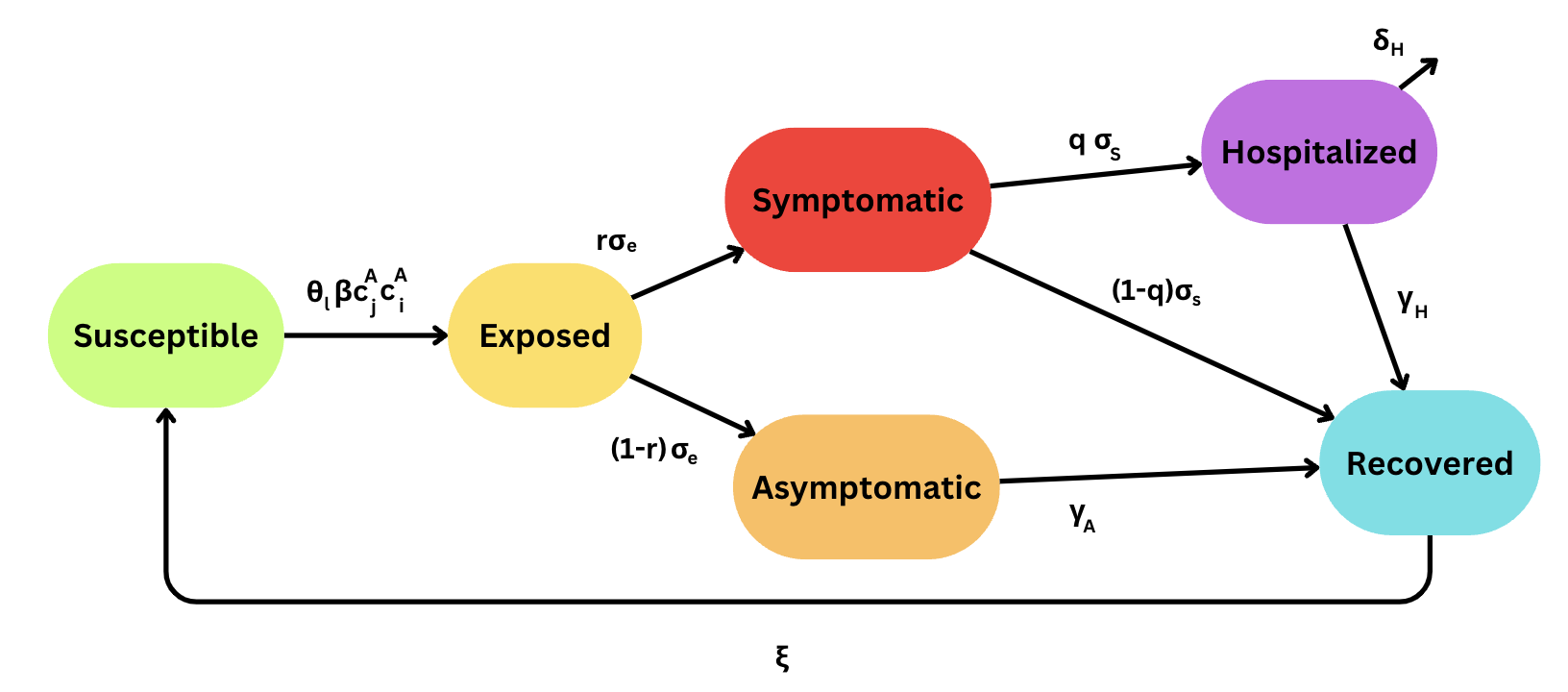}\\\includegraphics[width=145mm,scale=1]{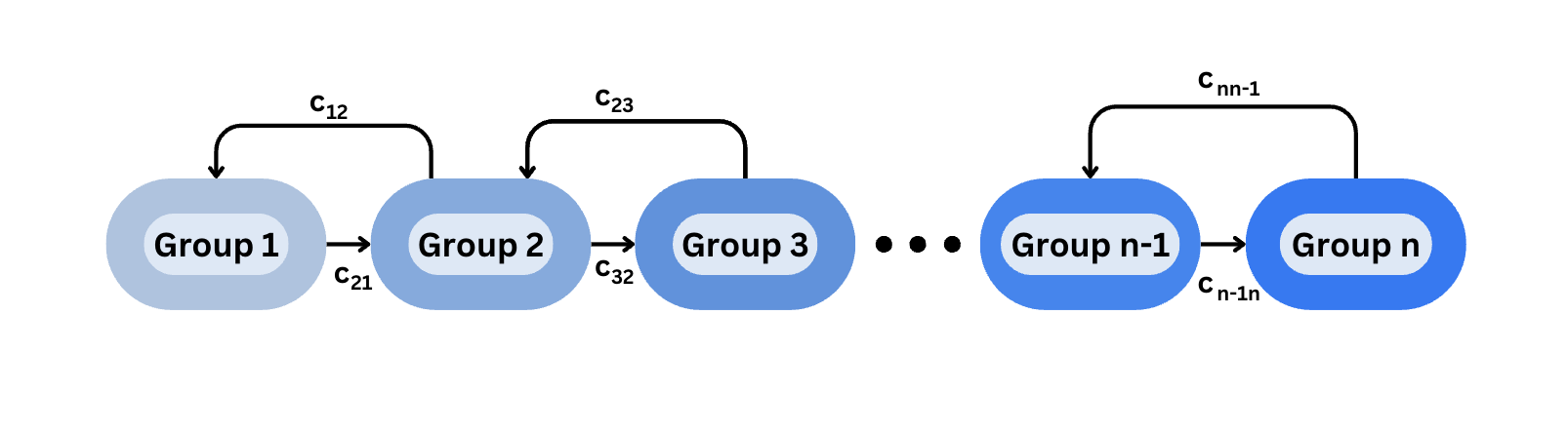}
    \caption{Flow diagram of the heterogeneous $n-$group behavior model \eqref{eq:fullmodel}, illustrating the disease dynamics (top panel) and the influence dynamics with $n$ behavioral groups (bottom panel).
    \label{fig:dynamics}}
\end{figure}
\noindent Specifically, in the model \eqref{eq:fullmodel}, susceptible individuals in behavioral group $i$ (i.e., $S_i$) acquire SARS-CoV-2 infection at a rate, $\lambda$ (known as the {\it force of infection}), given by (with $i=1,2,\,\cdots,n$):
$$ \lambda= \left(\theta_l\right)\left[\dfrac{c_i^A(t)S_i(t)}{N(t)}\right]\sum_{j=1}^nc_j^A(t)\left[\beta_aI_{a,j}(t)+\beta_iI_{s,j}(t)+\beta_hI_{h,j}(t)\right],$$ 
\noindent where $0< c^A_i(t)\le 1$ is a modification parameter accounting for the reduction in risky behavior by individuals in group $i$ due to the size of the proportion of individuals hospitalized due to SARS-CoV-2 in the community (the corresponding modification parameter for individuals in group $j$, with $j=1,2, \cdots,n$ and $j\ne i$, is $c^A_j(t)$.
Specifically, the modification parameter $c_i^A(t)$ is defined as:
\begin{align}
    \begin{split}
        c^A_i(t) &= e^{-a_iI_{h}(t)/N(t)},\;  i=1,2,\cdots, n,
    \end{split}\label{eq:cia}
\end{align}
\noindent where $I_h(t)=\dss\sum_{i=1}^nI_{h,i}(t)$ is the number of individuals in the community who are hospitalized with SARS-CoV-2 at time $t$, and the exponent $a_i>>1$ is a measure of hospitalization-induced positive behavioral change by individuals in group $i$. It is assumed that individuals in the more cautious behavioral groups (i.e., individuals in group $i$, where $i$ is relatively low (in comparison to $n$)) will positively change their behaviors more readily than those in the less cautious behavioral groups (i.e., those in group $i$ with $i$ close enough to $n$). Figure \ref{fig:cia} shows that a higher value of $a_i$ results in a larger reduction in contact rate by individuals in behavioral group $i$ (and, as a result, the force of infection) as a function of the proportion of individuals hospitalized in the community $I_h(t)/N(t)$. 
\begin{figure}[H]
    \centering
    \includegraphics[width=130mm,scale=1]{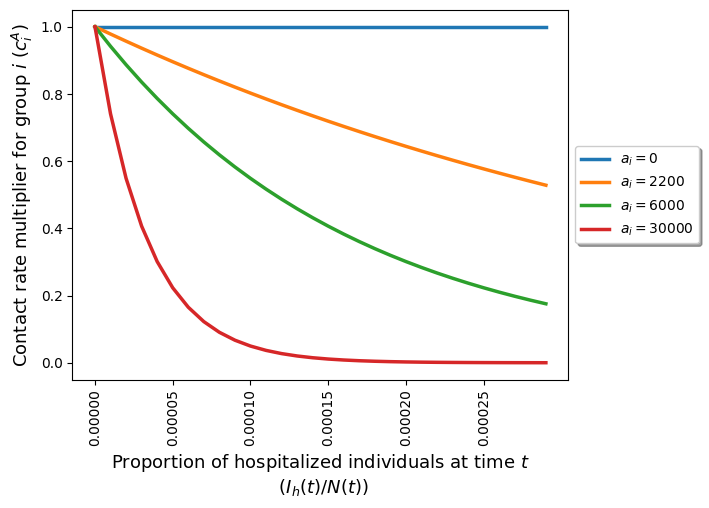}
    \caption{Profile of the effective contact rate modification parameter ($c_i^A(t)$), as a function of the proportion of the population that is hospitalized ($I_h(t)/N(t)$), for various values of $a_i$ (the modification parameter for the hospitalization-induced behavior change by individuals in group $i$). Blue, orange, green, and red curves represent, respectively, the case with $a_i=$ 0, 2,200, 6,000, and 30,000.}\label{fig:cia}
\end{figure} \vspace{-5mm}
\noindent The parameter $0\le \theta_l<1$ accounts for the effectiveness of government-mandated community lockdown (and other non-pharmaceutical intervention measures) implemented to prevent the acquisition and transmission of infection during the first wave of the pandemic in New York City \cite{ge_npi_gumel}. When there are no government-mandated lockdowns in place, $\theta_l$ is set equal to one, as this does not reduce the force of infection. As government-mandated lockdowns are put in place and become more effective (strict), the value of $\theta_l$ decreases. The parameters $\beta_a$, $\beta_i$ and $\beta_h$ represent, respectively, the maximum effective contact rate (i.e., maximum infection rate) for infectious individuals in the asymptomatic, symptomatic and hospitalized compartments.
\noindent Recovered individuals lose their infection-acquired (natural) immunity, and revert to the susceptible class, at a rate $\xi$. Exposed individuals develop clinical symptoms of SARS-CoV-2 at a rate $r \sigma_e$ (where $0<r<1$ is the proportion of exposed individuals who develop clinical symptoms of the disease at the end of the exposed period) or become asymptomatically-infectious at the rate $(1-r)\sigma_e$ (where $1-r$ is the remaining proportion of exposed individuals who become asymptomatic at the end of the exposed period). Symptomatically infectious individuals are hospitalized at a rate $q\sigma_i$ (where $0<q<1$ is the proportion of symptomatically infectious individuals who are hospitalized) or recover at a rate $(1-q)\sigma_i$. 
\noindent Asymptomatically infectious and hospitalized individuals recover at rates  $\gamma_a$ and $\gamma_h$, respectively. Hospitalized individuals suffer disease-induced mortality at a rate $\delta_h$ (it is assumed that only hospitalized infectious individuals suffer disease-related mortality, which is not a perfect assumption but is predominantly the case \cite{outofhospdeath}). Influence dynamics (where individuals in behavioral group $i$ exert some degree of influence over individuals in another behavioral group $j$) is modeled using the {\it influence-motivated behavioral change} parameter, $c^B_{ij}$. Specifically, the parameter $c^B_{ij}$ accounts for individual behavioral change in group $j$ with respect to public health interventions and risk taking behavior, meaning individuals in group $j$ who switch behavioral groups (increase or decrease the riskiness of their behavior) as they are influenced by individuals in behavioral group $i$. This parameter is multiplied by the total proportion of individuals in behavioral group $i$, ($N_i(t)/N(t)$) meaning the power a behavioral group $i$ has to change the behavior of others is proportional to its size. Exerted influence $c^B_{ij}$ results in a proportion of individuals in group $j$ changing their behavioral group and moving one group closer to group $i$, going to either group $j+1$ if $i>j$, or group $j-1$ if $i<j$. This influence dynamics process is illustrated in the bottom panel of Figure \ref{fig:dynamics}. 
\begin{table}[H] 
  \begin{center}
     \begin{tabular}{|c|p{10cm}|}
      \hline
      \textbf{Parameter ($i=1,2,\cdots,n$)} & \textbf{Description ($i=1,2,\cdots,n$)}  \\
          \hline \hline
      $\xi$ & Rate of loss of natural immunity \\ \hline
      $a_i$ & Modification parameter for hospitalization-induced behavioral change for individuals in group $i$ \\ \hline
      $0<(1-c^A_i(t))\le1$ &  A measure for the relative reduction in contact rate by individuals in behavioral group $i$ due to the level of disease-related hospitalization in the community \\ \hline
      $c^B_{ij}$ & Rate at which individuals in group $i$ influence those in group $j$ to transition to a group closer to $i$ (i.e., $j-1$ if $i<j$ or $j+1$ if $i>j$) \\ \hline
      $0<\theta_l\le 1$ & A measure for the efficacy of the government-mandated community lockdown (and other non-pharmaceutical intervention measures) \\ \hline
      $\beta_{a}$& Maximum effective contact rate for asymptomatic infectious individuals\\ \hline
       $\beta_{s}$& Maximum effective contact rate for symptomatic infectious individuals\\ \hline
    $\beta_{h}$& Maximum effective contact rate for hospitalized individuals\\ \hline
      $\sigma_{e}$& Rate of progression of exposed individuals to asymptomatic or symptomatic infectious compartment\\ \hline
      $r$& Proportion of exposed individuals that become symptomatic at the end of the exposed period\\ \hline
      $1-r$& Proportion of exposed individuals that become asymptomatic at the end of the exposed period\\ \hline
      $\sigma_i$ & Rate at which symptomatic individuals leave the symptomatic class\ \\ \hline 
      $q$& Proportion of symptomatic individuals that are hospitalized\\ \hline
      $1-q$& Proportion of symptomatic individuals that recover without hospitalization \\ \hline
      $\gamma_{a}$& Recovery rate for asymptomatic infectious individuals \\ \hline
      $\gamma_{h}$& Recovery rate for hospitalized individuals \\ \hline
      $\delta_{h}$& Disease-induced mortality rate for hospitalized individuals\\ \hline
\end{tabular}
    \caption{Description of the parameters of the $n-$group behavior model \eqref{eq:fullmodel}.}
    \label{table:parameters}
\end{center}
\end{table}
\noindent The main assumptions made in the formulation of the $n-$group behavior model \eqref{eq:fullmodel} include:
\begin{enumerate}
\item[(a)] Closed, large, and well-mixed population (i.e., homogeneous mixing between the various behavioral groups).
\item[(b)] Since the focus of this study is to model the dynamics of the SARS-CoV-2 pandemic during the first wave, no demographic (birth and natural death) processes are assumed. In other words, it is assumed that the timescale of the disease is shorter than the demographic timescale \cite{ITME_martcheva, gg_primer_gumel, hethcote_siam, km_1927}). The consequence of this assumption is that an epidemic, rather than an endemic, model will be used \cite{hethcote_siam}.
\item[(c)] Changes in risk-taking behaviors are induced either by (1) the level of disease burden in the community (measured by the proportion of individuals in the community hospitalized due to SARS-CoV-2 infection) or by (2) peer (influence) pressure (defined in terms of a linear function of behavioral group size (i.e., larger size groups will exert more pressure than groups of smaller size \cite{gg_sociocultural_yang}.
\item[(e)] Individuals currently in behavioral group $i$ can only transition (change their behavior) to the nearest behavioral groups ($i-1$ or $i+1$). This is to account for the assumption that individuals do not typically drastically change their risk-tolerance behavior over short time periods \cite{behaviorchange}.
\end{enumerate} 
\noindent Since the behavior model \eqref{eq:fullmodel} monitors the temporal dynamics of human populations, all its parameters and state variables are non-negative for all time $t$. Furthermore, the model is an extension of numerous other compartmental models for the SARS-CoV-2 pandemic that incorporate human behavior changes, such as the models in \cite{gg_simplebehavior_lejune, gg_heterogeneous_espinoza, ge_behavior_binod}, by, {\it inter alia}:
\begin{enumerate}
\item[(i)] Incorporating greater heterogeneity in behavior by having $n$ behavioral compartments ($S_i,E_i,A_i,\cdots R_i$ for $i=1,2,\,\cdots,n$). 
This extends the studies by Lejune et al. \cite{gg_simplebehavior_lejune} (which assumes homogeneous behavior within the population), Espinoza et al. \cite{gg_heterogeneous_espinoza} and Pant et al. \cite{ge_behavior_binod} (which considered a maximum of two behavioral groups for all disease compartments).
\item[(ii)] Allowing for back-and-forth transition between the $n$ heterogeneous behavioral sub-groups ($c_{ij}$, where $j,i=1,2,\cdots,n$, and $j\neq i$).  This extends the Espinoza et al. study  \cite{gg_heterogeneous_espinoza} (where no such transitions are allowed between the two behavioral groups considered therein). 
\item[(iii)] Incorporating the dynamics of hospitalized individuals (i.e., adding equations for the dynamics of the compartments $H_i$, for $i=1,2,\,\cdots,\,n$). The dynamics of hospitalized individuals are not accounted for in the models presented in \cite{ge_behavior_binod,gg_simplebehavior_lejune,gg_heterogeneous_espinoza}.
\item[(iv)] Using a time-dependent transmission parameter ($c_i^A(t)$, with $i=1,\cdots,n$) to account for group-level change of behavior in response to the level of hospitalization in the community. This extends the studies in \cite{ge_behavior_binod,gg_heterogeneous_espinoza} (which used fixed parameters for behavior change in response to the disease).
\item[(v)] Assuming that recovery from infection does not induce permanent immunity against future infection (i.e., $\xi\neq 0$). This extends the studies in \cite{ge_behavior_binod,gg_heterogeneous_espinoza} (which assume recovery from natural infection induces permanent immunity against future infection).
\end{enumerate}
The basic qualitative properties of the behavior model \eqref{eq:fullmodel} will now be explored below.
\subsection{Basic qualitative properties of the multigroup model}
\noindent Consider the following biologically-feasible region for the model \eqref{eq:fullmodel}:
\[\Omega=\biggl\{\left(S_1,\cdots,S_n,E_1,\cdots,E_n,\cdots,R_n\right)\in {\mathbb R}_+^{6n}:0\leq \sum_{i=1}^n(S_i+E_i+I_{a,i}+I_{s,i}+I_{h,i}+R_i)\leq N(0)\biggr\},\]
where $N(0)=\dss\sum_{i=1}^n(S_i(0)+E_i(0)+I_{a,i}(0)+I_{s,i}(0)+I_{h,i}(0)+R_i(0))$ is the total initial size of the population. We claim the following result:
\begin{theorem}\label{invariance}
    Suppose that the initial values $S_i(0),E_i(0),I_{a,i}(0),I_{i,i}(0),I_{h,i}(0),R_i(0)$ for $i=1,\cdots, n$ of the model \eqref{eq:fullmodel} are non-negative, with $S_i(0)>0$ for at least some $i$. Then, the solutions 
    \[S_i(t),E_i(t),I_{a,i}(t),I_{i,i}(t),I_{h,i}(t),R_i(t) \text{ for } i=1,\cdots, n \]
    \noindent of the model are bounded for all time $t\ge0$. Furthermore, the region $\Omega$ is positively-invariant with respect to the flow generated by the behavior model \eqref{eq:fullmodel}.
\end{theorem}
\begin{proof}    
\indent To prove the boundedness of the model, it is convenient to obtain the equation for the rate of change of the total population (obtained by adding all the equations of the model), given by: 
    \[\frac{dN}{dt}=-\delta_h\sum_{i=1}^nI_{h,i}(t).\]
    Since all the parameters and state variables of the model are non-negative (consequently, $\delta_h\ge0$ and  $I_{h,i}(t)\ge0$ for all $t$, with $i=1,\cdots,n$),
    it follows from the above equation that: 
    \begin{align}\label{eq:IVP}
        \frac{dN}{dt}\leq 0, \,\,{\rm for\,\,all\,\,} t\ge0.
    \end{align}
    To determine an upper bound for $N(t)$, the upper solution of the differential inequality \eqref{eq:IVP}, given by the solution of the initial-value problem, $\frac{dN}{dt}= 0$ subject to $N(0)=N_0$, is obtained.  Applying the initial condition into the solution $N(t)=c$ (where $c$ is a constant of integration) of $dN/dt=0$ gives $c=N_0$.  Thus, $N(t)$ is bounded above by $N_0$. 
    Since $S_i(t)>0$ for some $i$ and all $t\ge0$, the total population, $N(t)$, is bounded below by zero. Hence, $0<N(t)\leq N_0$, meaning $N(t)$ is bounded. Since $N(t)$ is bounded, it follows that all the state variables of the model are bounded. Thus, all solutions of the model \eqref{eq:fullmodel} are  bounded in $\Omega$ for all time $t\ge 0$. Since all solutions of the model are bounded in $\Omega$ (i.e., all initial solutions in $\Omega$ remain in $\Omega$ for all time $t$), the region $\Omega$ is positively-invariant with respect to the flow generated by the model \eqref{eq:fullmodel}.
\end{proof}
\noindent The implication of Theorem \ref{invariance} is that the behavior model \eqref{eq:fullmodel} is well-posed mathematically and epidemiologically in the bounded and positively-invariant region $\Omega$ \cite{hethcote_siam}. Hence, it is sufficient to study its qualitative dynamics in the region $\Omega$.
\section{Analysis of a special case with two behavioral groups}
\noindent For practical application and illustrative purposes, a special case of the $n$-group behavior model with two behavioral groups (i.e., Equation \eqref{eq:fullmodel} with $n=2$) will be considered and rigorously analyzed. The equations for this special case, tagged the {\it 2-group behavior model},  are given by Equation \eqref{eq:n2model} of Appendix B.
\subsection{Existence and stability of disease-free equilibria of the 2-group model}\label{sec:stab_analysis}
\subsubsection{Existence}
\noindent 
The 2-group behavior model \eqref{eq:n2model} has the following disease-free equilibria (which always exist):
\begin{enumerate}
    \item[(i)]{\bf non-trivial disease-free equilibria (NTDFE):}
\begin{equation}\label{eq:NTDFE} (S_1^*, E_1^*, I_1^*, A_1^*, H_1^*, R_1^*, S_2^*, E_2^*, I_2^*, A_2^*, H_2^*, R_2^*)
= (kp,(1-k)p,0,0,0,0,0,0,0,0,0,0),
\end{equation}
where $0\leq k\leq1$ and $p$ is the total population size at equilibrium. These equilibria are classified as follows:
    \begin{enumerate}
        \item[\bf(a)]\label{DFE1} ($k=1$): the resulting disease-free equilibrium, termed the \textit{Group 1-only} disease-free equilibrium, is denoted by \textit{G1DFE}.
        \item[\bf(b)]\label{DFE2} ($k=0$): the resulting disease-free equilibrium, called the \textit{Group 2-only} disease-free equilibrium, is denoted by \textit{G2DFE}.
        \item[\bf(c)]\label{DFE3} ($0<k<1$): the resulting disease-free equilibrium, termed the \textit{coexistence} disease-free equilibrium, is denoted by \textit{G3DFE}.
    \end{enumerate}
\item[(ii)] {\bf trivial disease-free equilibrium (TDFE):}
\begin{equation}\label{TDFE} (S_1^*, E_1^*, I_1^*, A_1^*, H_1^*, R_1^*, S_2^*, E_2^*, I_2^*, A_2^*, H_2^*, R_2^*)
= (0,0,0,0,0,0,0,0,0,0,0,0).  
\end{equation}
\noindent Since this {\it extinction} equilibrium is not epidemiologically realistic (i.e., no humans at steady-state), it will not be analyzed for its stability. 
\end{enumerate}
The stability of the three non-trivial disease-free equilibria of the 2-group model will now be analyzed.
\subsubsection{Stability of the non-trivial disease-free equilibria of the 2-group model}
\noindent The stability of the three non-trivial DFE of the 2-group behavior model (G1DFE, G2DFE, and G3DFE) will be analyzed using standard linearization. Owing to the large size of the 2-group behavior model (making the computation of the eigenvalues of its Jacobian evaluated at the respective DFE less tractable mathematically), the linearization approach will be combined with some results related to the properties of Metzler matrices \cite{metzler1,metzler2}. In particular, the notations in \cite{godsend} will be used. It is convenient to recall the following definitions (from \cite{godsend}), 
\begin{definition}\label{def:metzlermats}
    A matrix $M\in \RR^{n\times n}$ is called a \textit{Metzler matrix} if all of its off-diagonal entries are nonnegative. 
\end{definition}
\begin{definition}
    A matrix $M\in \RR^{n\times n}$ is called \textit{Hurwitz-Stable} if the real part of all eigenvalues of $M$ are strictly negative.
\end{definition}
\begin{definition}\label{def:metzlerstab}
    A matrix $M\in \RR^{n\times n}$ is called \textit{Metzler-stable} if it is a Metzler matrix with Hurwitz-Stability (i.e. the real part of all eigenvalues of $M$ are strictly negative).
\end{definition}
\noindent Define the following quantity:
\begin{equation}\label{eq:RC}
    \RR_c=\theta_l\biggl(\frac{r\beta_i}{\sigma_i} + \dss\frac{(1-r)\beta_a}{\gamma_a} + \dss\frac{rq\beta_h}{\gamma_h+\delta_h}\biggr)
\end{equation}
\noindent Lastly, it is convenient to define the additional parameter 
\[\Gamma=c^B_{12}/c^B_{21},\]
which will be called the \textit{relative influence ratio} of behavioral Group 1 to behavioral Group 2. We claim the following result:
\begin{theorem}\label{thm:LASG1DFE}
The Group 1-only disease-free equilibrium (G1DFE) of the 2-group behavior model \eqref{eq:n2model} is stable whenever $\RR_c<1$ \textit{and} $\Gamma>1$. The G1DFE is unstable whenever $\RR_c>1$ or $\Gamma\leq 1$.
\end{theorem}
\noindent The proof of this theorem (\ref{thm:LASG1DFE}) is given in Appendix B.
\noindent Similarly, the following results can be established for the stability of the G2DFE and the G3DFE (the proofs are not given here to save space):
\begin{theorem}\label{thm:LASG2DFE}
The Group 2-only disease-free equilibrium (G2DFE) of the 2-group behavior model \eqref{eq:n2model} is stable whenever $\RR_c<1$ \textit{and} $\Gamma<1$, and unstable whenever $R_c>1$ or $\Gamma\geq 1$.
\end{theorem}
\begin{theorem}\label{thm:LASG3DFE}
The coexistence disease-free equilibrium (G3DFE) of the 2-group behavior model \eqref{eq:n2model} is stable whenever $\RR_c<1$ \textit{and} $\Gamma=1$, and is unstable whenever $R_c>1$ or $\Gamma\neq 1$.
\end{theorem}

\noindent The epidemiological implication of Theorems \ref{thm:LASG1DFE}-\ref{thm:LASG3DFE} is that a small influx of SARS-CoV-2 infected individuals will not generate a large outbreak of SARS-CoV-2 in the community if $\RR_c$ is less than unity. The quantity $\RR_c$ is the \textit{control reproduction number} of the 2-group model \eqref{eq:n2model}, which represents the average number of new SARS-CoV-2 infections generated by an infectious individual introduced into a population where non-pharmaceutical public health interventions (such as lockdown measures) are implemented \cite{hethcote_siam, andersonandmay}. In the absence of non-pharmaceutical interventions (i.e., $\theta_l=1$), the control reproduction number $\RR_c$ reduces to  
\begin{align}
   \RR_0=\RR_c\lvert_{\theta_l=1}= \frac{r\beta_i}{\sigma_i} + \dss\frac{(1-r)\beta_a}{\gamma_a} + \dss\frac{rq\beta_h}{\gamma_h+\delta_h}, 
\end{align}\label{eq:R0}
\noindent where $\RR_0$ is the {\it basic reproduction number} of the model \eqref{eq:n2model} (which measures the average number of new infections generated by an infectious individual introduced into a completely susceptible population). The results of Theorems \ref{thm:LASG1DFE}- \ref{thm:LASG3DFE} are numerically illustrated in Figure \ref{fig:rc_less} where numerous initial conditions of the model are shown to converge to one of the three non-trivial disease-free equilibria when $\RR_c<1$.\\
\indent However, when $\RR_c>1$, simulations suggest that all initial solutions of the 2-group model \eqref{eq:n2model} converge to the TDFE (see Figure \ref{fig:rc_greater}). Hence, the 2-group behavior model \eqref{eq:n2model} undergoes a transcritical bifurcation at $\RR_c=1$. This result is a consequence of the fact that the 2-group behavior model \eqref{eq:n2model} is an epidemic model, which assumes a closed population (no immigration or emigration), no demographics (no birth rate or natural death rate) and recovery from SARS-CoV-2 infection does not induce permanent immunity against future infections (i.e., $\xi\neq0$). Under these assumptions, every member of the community will acquire infection and succumb to the disease, and (since there is no replenishment or influx of new susceptible individuals into the community, by birth or immigration) the population will eventually die out. This does not happen when $\xi=0$, meaning when infection \textit{does} induce permanent immunity the disease dies out while the population does not (see Figure \ref{fig:xi0} in Appendix C).\\\\
\noindent The results of Theorems \ref{thm:LASG1DFE}-\ref{thm:LASG3DFE}, along with the simulations in Figure \ref{fig:rc_greater}, are numerically illustrated in Figure \ref{fig:stabilitycountour}, showing the stability regions of the disease-free equilibria (G1DFE, G2DFE, G3DFE, TDFE) for various values of the control reproduction number, $\RR_c$, and the relative influence ratio, $\Gamma$.
\begin{figure}[H]
    \centering \textbf{(a) $\Gamma>1$}\quad\quad\quad\quad\quad\quad\quad\textbf{(b) $\Gamma<1$}\quad\quad\quad\quad\quad\quad\quad\textbf{(c) $\Gamma=1$}\\
    \includegraphics[width=0.25\linewidth]{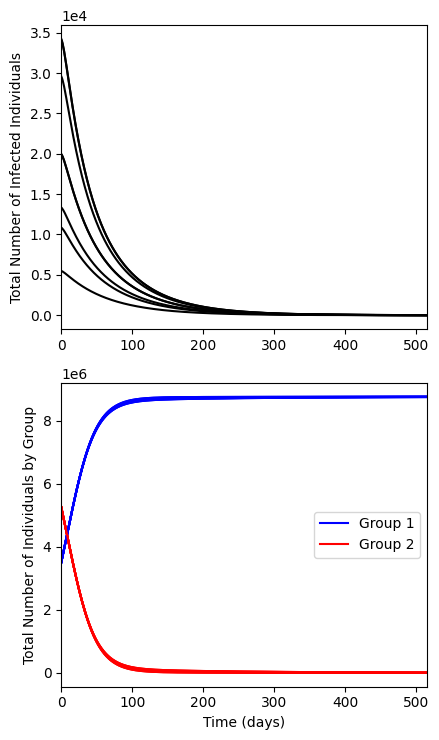}\includegraphics[width=0.25\linewidth]{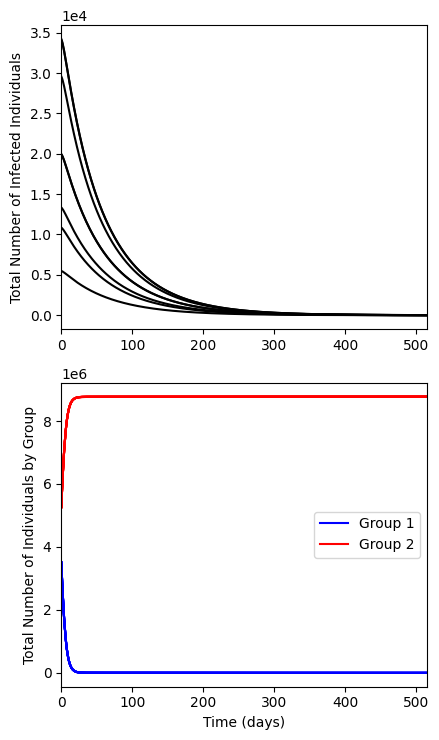}\includegraphics[width=0.25\linewidth]{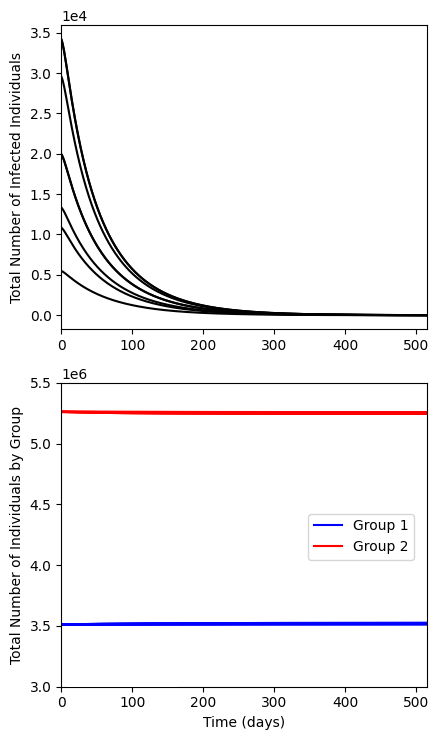}
    \caption{Simulations of the 2-group behavior model \eqref{eq:n2model}, showing the profile of the total number of infected individuals (top) and the total number of individuals in the community, stratified by behavioral group ($N_1(t)$ and $N_2(t)$)(bottom), as a function of time for various initial conditions. Column (a): $c^B_{12}=0.5,c^B_{21}=0.25$ ($\Gamma=2>1$). Column (b): $c^B_{12}=0.25, c^B_{21}=0.5$ ($\Gamma=0.5<1$). Column (c): $c^B_{12}=c^B_{21}=0.5$ ($\Gamma=1$). In all panels, all other parameter values used in these simulations are as given in Tables \ref{tab:fixedparams} and \ref{tab:fittedparams_full}, but with $\theta_l=1$, $\beta_a=0.1$, and $\beta_i=0.05$ (so that, $\RR_c=0.78<1$). These simulations show that when $\RR_c<1$, solutions of model \eqref{eq:n2model} converge to G1DFE if $\Gamma>1$ (see Figure(a)), G2DFE if $\Gamma<1$ (see Figure(b)), and G3DFE if $\Gamma=1$ (see Figure(c)), in line with Theorems \ref{thm:LASG1DFE}-\ref{thm:LASG3DFE}.}
    \label{fig:rc_less}
\end{figure}
\begin{figure}[H]
    \centering 
    \textbf{(a)} \quad\quad\quad\quad\quad\quad\quad\quad \textbf{(b)} \quad\quad\quad\quad\quad\quad\quad\quad\textbf{(c)}\\
    \includegraphics[width=0.7\linewidth]{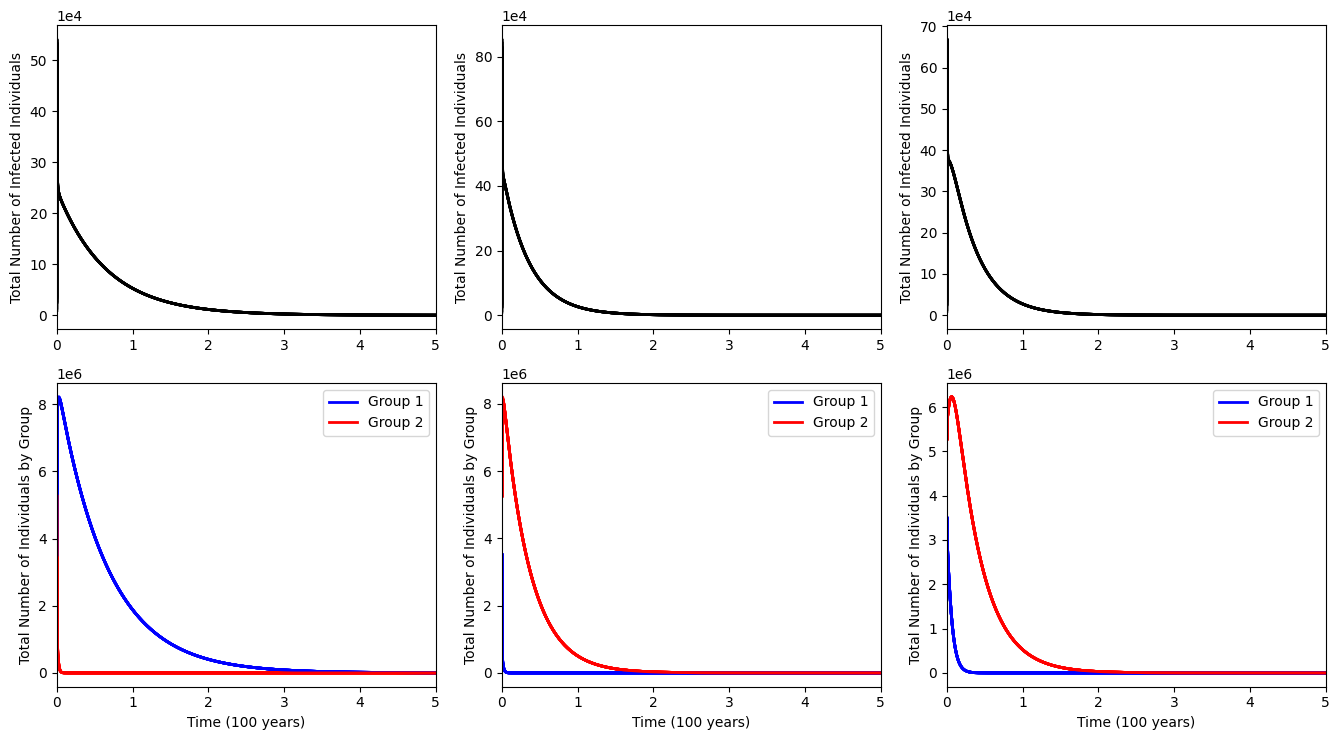}
    \caption{Simulations of the 2-group behavior model \eqref{eq:n2model}, showing the profile of the total number of infected individuals (top) and the total number of individuals in the community, stratified by behavioral group ($N_1(t)$ and $N_2(t)$)(bottom), as a function of time for various initial conditions. Column (a): $c^B_{12}=0.5,c^B_{21}=0.25$ ($\Gamma=2>1$). Column (b): $c^B_{12}=0.25, c^B_{21}=0.5$ ($\Gamma=0.5<1$). Column (c): $c^B_{12}=c^B_{21}=0.5$ ($\Gamma=1$). In all panels, all other parameter values used in these simulations are as given in Tables \ref{tab:fixedparams} and \ref{tab:fittedparams_full}, but with $\theta_l=1$, $\beta_a=1$, and $\beta_i=0.5$ (so that, $\RR_c=7.8>1$). These simulations show that when $\RR_c>1$, solutions of the 2-group model \eqref{eq:n2model} converge to the extinction equilibrium (TDFE), regardless of the value of $\Gamma$.}
    \label{fig:rc_greater}
\end{figure} \vspace{-5mm}
\begin{figure}[H]
    \centering \includegraphics[width=0.4\linewidth]{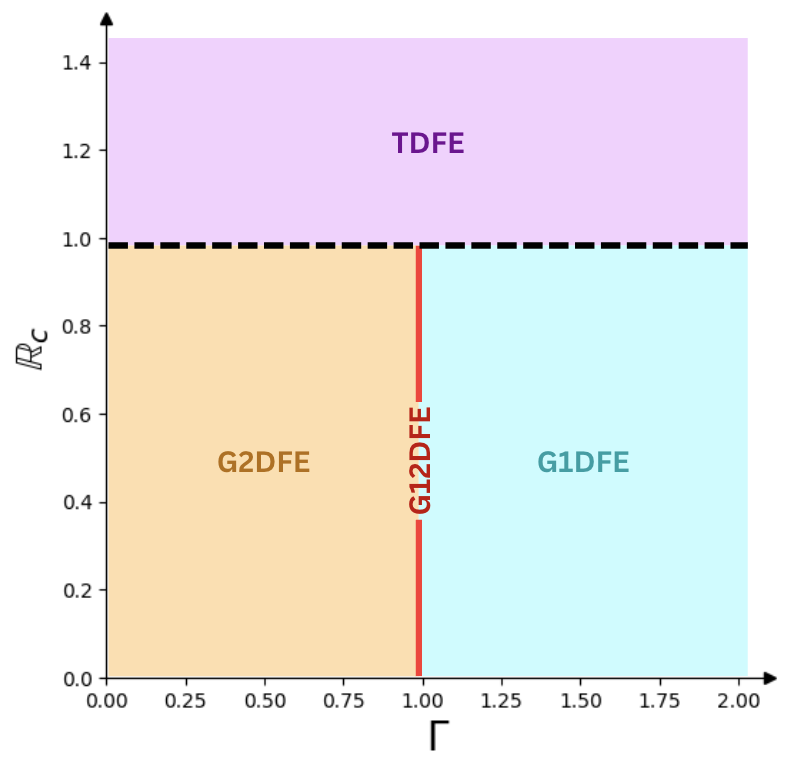}
    \caption{Stability regions of different disease-free equilibria of the 2-group behavior model \eqref{eq:n2model}, as determined by the values of the control reproduction number $\RR_c$ and the relative influence ratio $\Gamma$. This figure shows that the 2-group behavior model \eqref{eq:n2model} has a bifurcation at $\RR_c=1$ and a separatrix at $\Gamma=1$.}
    \label{fig:stabilitycountour}
\end{figure}
\subsection{Data fitting and parameter estimation for the 2-group model \eqref{eq:n2model}}\label{sec:fitting}
\noindent The 2-group behavior model \eqref{eq:n2model} has 15 parameters. The baseline values of ten of these parameters are known from the literature, and the baseline values of the remaining five are unknown. The 2-group behavior model will now be fitted to the observed daily hospitalization data (which was smoothed into 7-day running average data) for New York City \cite{nycgov_data} during the first wave of the SARS-CoV-2 pandemic to estimate the five unknown parameters. For notational convenience, the known (unknown) parameters of the 2-group behavior model are categorized as {\it fixed} ({\it fitted}) parameters.  Furthermore, for model fitting purposes, the following five time intervals or phases for the implementation of the SARS-CoV-2 lockdown measures in New York City are considered  \cite{314_1,44_1,44_2,528_1,528_2,825_1,825_2,825_3}:
\begin{enumerate}
\item[(a)] Phase 0 ({\it pre-lockdown phase}): time period from January 20, 2020 (day of the index case in the US) to the beginning of initial lockdown (March 13, 2020, \cite{314_1}). For this phase, the lockdown parameter, $\theta_l$, is set to 1 (i.e., no lockdown measures are implemented in New York City yet).
\item[(b)] Phase 1 (\textit{initial lockdown phase}): this accounts for the time period from March 14, 2020 to April 4, 2020, \cite{44_1,44_2} when initial lockdown measures were implemented in New York City. Here, the lockdown parameter is represented by $\theta_{l,1}$ which will be fitted for.
\item[(c)] Phase 2 (\textit{hard lockdown phase}): this accounts for the time period from April 5, 2020 to May 27, 2020 \cite{528_1,528_2}, when stringent lockdown measures were implemented in New York City. Here, the lockdown parameter is represented by $\theta_{l,2}$ which will be fitted for.
\item[(d)] Phase 3 (\textit{first reopening phase}): this accounts for the time period from May 28, 2020 to August 25, 2020 \cite{825_1,825_2,825_3}, when some of the lockdown measures were lifted in New York City. Here, the lockdown parameter is represented by $\theta_{l,3}$ which will be fitted for.
\item[(e)] Phase 4 (\textit{second/final reopening phase}): this accounts for the time period from August 25, 2020 onward when more lockdown measures were lifted. Here, the lockdown parameter is represented by $\theta_{l,4}$, whose value will be obtained from fitting the model with data.
\end{enumerate}
Thus, based on these lockdown implementation phases, the NPI parameter, $\theta_l$, in the model \eqref{eq:n2model} is now explicitly defined in terms of the following piecewise-defined function:
\begin{align}
    \theta_{l} = \begin{cases} 
      1 & \text{during Phase 0} \\
      \theta_{l,1} & \text{during Phase 1} \\
      \theta_{l,2} & \text{during Phase 2} \\
      \theta_{l,3} & \text{during Phase 3} \\
      \theta_{l,4} &  \text{during Phase 4}
   \end{cases}\label{eq:theta}
\end{align} 
This will result in three additional (eight total) unknown parameters to be fitted. The values of the fixed parameters of the 2-group behavior model \eqref{eq:n2model} used in fitting it to the observed data \cite{nycgov_data} are briefly described below.
\subsubsection{Baseline values of the fixed parameters of the 2-group model \eqref{eq:n2model}}\label{sec:paramsfixed}
\noindent The baseline values of all the disease-related parameters of the 2-group behavior model \eqref{eq:n2model} have been estimated in the literature. For instance, the effective contact rate for disease transmission by asymptomatically-infectious individuals ($\beta_a$) during the first wave of the SARS-CoV-2 pandemic in the United States is estimated to be between 0.34 and 0.72 {\it per} day \cite{Gumel4,Gumel5}. Similarly, the transmission rates for symptomatic infectious individuals ($\beta_i$) during the same period is estimated to be between 0.36 and 0.45 {\it per} day \cite{Gumel4,Gumel1}. The rate of transition out of the exposed class ($\sigma_e$) to the asymptomatic-infectious class or symptomatically-infectious class has been estimated to be 1/4 {\it per} day \cite{sigmae1, sigmae2}.  The proportion ($r$) of these exposed individuals who show clinical symptoms of the disease at the end of the exposed period is estimated to be 0.6 \cite{sigmap_brozak}. The rate of transition out of the symptomatic-infectious class ($\sigma_i$) is estimated to be 1/14 {\it per} day\cite{gammaa1}, and the proportion, $q$, of symptomatic infectious individuals who are hospitalized is estimated to be 0.05 \cite{cdc1}. Dan \textit{et al}. \cite{xi_dan} estimated the rate of loss of infection-acquired immunity ($\xi$) during the first wave of the SARS-CoV-2 pandemic to be 1/180 {\it per} day. Furthermore, asymptomatic and hospitalized individuals recover within 9 and 10 days, respectively (so that, $\gamma_a=1/9$ {\it per} day and $\gamma_h=1/10$ {\it per} day) \cite{gammaa1,gammaa_kissler}. Finally, individuals hospitalized with SARS-CoV-2 during the first wave of the pandemic suffer disease-induced mortality ($\delta_h$) at a rate 0.41 {\it per} day \cite{deltah}. The baseline values of these fixed parameters of the 2-group model \eqref{eq:n2model} are tabulated in Table \ref{tab:fixedparams}. 
\begin{table}
\begin{center}
\begin{tabular}{|c|c|c|}
\hline
\textbf{Parameter} & \textbf{Baseline value}  & \textbf{Source}  \\
 \hline \hline
  $\beta_{a}$  & 0.625 day$^{-1}$  &  \cite{Gumel4,Gumel5}\\ \hline
  $\beta_{s}$ & 0.375 day$^{-1}$&  \cite{Gumel4,Gumel1}\\ \hline
  $\xi$ & 1/180 day$^{-1}$& \cite{xi_dan} \\ \hline
  $\sigma_{e}$ & $1/4$ day$^{-1}$& \cite{sigmae1, sigmae2} \\ \hline
  $\sigma_{s}$ & $1/14$ day$^{-1}$& \cite{gammaa1} \\ \hline
  $\gamma_{a}$ & $1/9$ day$^{-1}$ & \cite{gammaa1} \\ \hline
  $\gamma_{h}$ & $1/10$ day$^{-1}$ & \cite{gammaa_kissler} \\ \hline
  $r$ & $0.6$ (dimensionless)& \cite{sigmap_brozak} \\ \hline
  $q$ & $0.05$ (dimensionless)  & \cite{cdc1} \\ \hline
  $\delta_h$ &  $0.41$ day$^{-1}$  &  \cite{deltah} \\ \hline
\end{tabular}
\caption{Baseline values of the fixed parameters of the 2-group model \eqref{eq:n2model} during the first wave of the SARS-CoV-2 pandemic in New York City.} \label{tab:fixedparams}
\end{center}
\end{table}
\subsubsection{Values of estimated parameters of the 2-group model obtained from data fitting}\label{sec:paramsfitted}
\noindent The 2-group behavior model \eqref{eq:n2model} is fitted to the 7-day moving average of the observed daily SARS-CoV-2 hospitalization data for the City of New York during the first wave of the pandemic \cite{nycgov_data} by using a suitable optimization method to estimate the eight unknown parameters (namely, the four behavior-related parameters: $a_1$, $a_2$, $c^B_{12}$, $c^B_{21}$, and the four phased lockdown implementation parameters: $\theta_{l,1}$, $\theta_{l,2}$, $\theta_{l,3}$, and $\theta_{l,4}$, defined in equation \eqref{eq:theta}). In particular, the BFGS algorithm (embedded in the \texttt{LMFit} library in python, and chosen due to its fast convergence rate and robustness in handling noisy data and local minima \cite{BFGS}) is used due to fit the 2-group behavior model \eqref{eq:n2model}). The results obtained from fitting the model to the 7-day running average hospitalization data \cite{nycgov_data} are depicted in Figure \ref{fig:predict_full}. This figure shows a very good fit of the 2-group model (blue curve) to the observed 7-day moving average of the daily hospitalization data for New York City (red dots) for the fitting period (i.e., the time period between February 29th, 2020 till July 28th, 2020). 
The estimated values of the {\it fitted parameters} (obtained from fitting the 2-group behavior model \eqref{eq:n2model} with the data), together with their associated 95\% confidence intervals, are tabulated in Table \ref{tab:fittedparams_full}. The accuracy of the fitted 2-group model is assessed by using it to predict the second wave of the SARS-CoV-2 pandemic in New York City.  To do this, the 2-group behavior model \eqref{eq:n2model} is simulated using the fixed and estimated parameters in Tables \ref{tab:fixedparams} and \ref{tab:fittedparams_full}, respectively, and compared with the available (withheld) 7-day running average hospitalization data for the second wave. The results obtained, depicted to the right of the dashed vertical black line of this figure, showed that the fitted model (green curve) almost perfectly predicts the second wave (red dots to the right of the dashed vertical black line).
\begin{figure}[H]
    \centering
\includegraphics[width=150mm,scale=1]{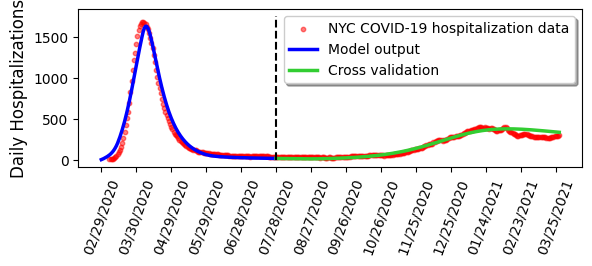}
    \caption{Data fitting and cross validation of the 2-group behavior model \eqref{eq:n2model}, using the 7-day running average daily hospitalization data for New York City \cite{nycgov_data} (red dots), carried out using the BFGS algorithm. The values of the fixed parameters of the 2-group model, used in fitting it with the data, are given in Table \ref{tab:fixedparams}, and the values of the estimated parameters of the 2-group model (obtained from the data fitting), together with their associated 95\% confidence intervals, are given in Table \ref{tab:fittedparams_full}. The data fitting (blue curve) is carried out using the data from the first 150 days after the index case in New York City (i.e., from February 29, 2020 until July 28, 2020), while the cross-validation (green curve) is carried out using the following eight months worth of data (i.e., from July 29, 2020 until March 25, 2021). The eight unknown parameters of the 2-group model were estimated from the data fitting to the model, and their estimated values are tabulated in Table \ref{tab:fittedparams_full}.}
    \label{fig:predict_full}
\end{figure}
\begin{center}
\begin{table}[H]
    \centering
\begin{tabular}{|p{1.8cm}||p{4.5cm}|p{3.3cm}|}
     \hline
     Parameter & Baseline value  & 95\% CI \\
     \hline
     \hline
     $a_1$ & 8,000 (dimensionless) & (7,923.4,$\infty$) \\ \hline
     $a_2$ & 2,800 (dimensionless)  &  (2,790.7,7,923.4)\\ \hline
      $c^B_{12}$ & 1/30 day$^{-1}$ & (0.031,0.036) \\ \hline
      $c^B_{21}$ & 1/90 day$^{-1}$ & (0.0087,0.014) \\ \hline
       $\theta_{l,1}$ & 0.74664 (dimensionless) & (0.65,0.75) \\ \hline
       $\theta_{l,2}$ & 0.00133 (dimensionless) & (0,0.0061) \\ \hline
       $\theta_{l,3}$ & 0.15427 (dimensionless) & (0.15,0.16) \\ \hline
       $\theta_{l,4}$ & 0.30970 (dimensionless) & (0.29,0.34) \\ \hline
    \end{tabular}\caption{Baseline values of the estimated parameters for the 2-group behavior model \eqref{eq:n2model}, obtained by fitting it to the 7-day running average data of daily hospitalizations in New York City \cite{nycgov_data} from February 29, 2020 till July 28, 2020 using the BFGS algorithm. The values of the fixed parameters of the 2-group model, used in the fitting of the model, are as given in Table \ref{tab:fixedparams}.}
       \label{tab:fittedparams_full}
\end{table}
\end{center} \vspace{-20mm}
It can be seen from Table \ref{tab:fittedparams_full} that individuals in Group 1 reduce their contacts ($c_1^A(t))$, defined in \eqref{eq:cia}) at a rate much faster than that of individuals in Group 2 ($c_2^A(t))$, defined in \eqref{eq:cia}). This is because the estimated value of the exponent, $a_1$ (for the hospitalization-induced behavioral change for the individuals in Group 1), tabulated in Table \ref{tab:fittedparams_full}, is much higher (almost 3 times) than that of the exponent, $a_2$ (for the hospitalization-induced behavioral change for Group 2). It should be noted that individuals in Group 2 also reduce their risk-taking behavior in response to the level of hospitalization in the community ({\it albeit} only slightly; recall that, $a_2=2,800\ne0$). The profile of estimated contact rates for individuals in Group 1 ($c_1^A(t)$; blue curve) and Group 2 ($c_2^A(t)$; red curve) is shown in Figure \ref{fig:contactrate_time}, along with simulated daily hospitalizations ($I_h(t)$) in the population (computed using the 2-group behavior model \eqref{eq:n2model} with parameters in Tables \ref{tab:fixedparams} and \ref{tab:fittedparams_full}, shown in green curve). This figure shows that, at the peak of the SARS-CoV-2 hospitalizations in New York City during the first wave (highlighted by the yellow shaded region), individuals in Group 1 dramatically decreased their daily contacts (by about 80\%, compared to their baseline contact rate; see the blue curve at the peak of the green curve). In contrast, individuals in Group 2 only decreased their daily contacts by about 40\% during the same period (see the red curve, at the peak of the green curve). A similar phenomenon occurred at the peak of the second wave (highlighted by the shaded purple region), but with the individuals in Group 1 decreasing contacts only by about to 30\% (in comparison to their baseline), while Group 2 individuals only decreased their contacts by a little more than 10\% (in comparison to the baseline). The lower reduction in contacts during the second wave of the SARS-CoV-2 pandemic in New York City may have been a result of masking fatigue \cite{nycfatigue}, or perhaps the relatively lower hospitalization rate during this time period (as compared to the first wave) resulted in a lower perception of risk.\\\\
\begin{figure}
    \centering
\includegraphics[width=115mm,scale=1]{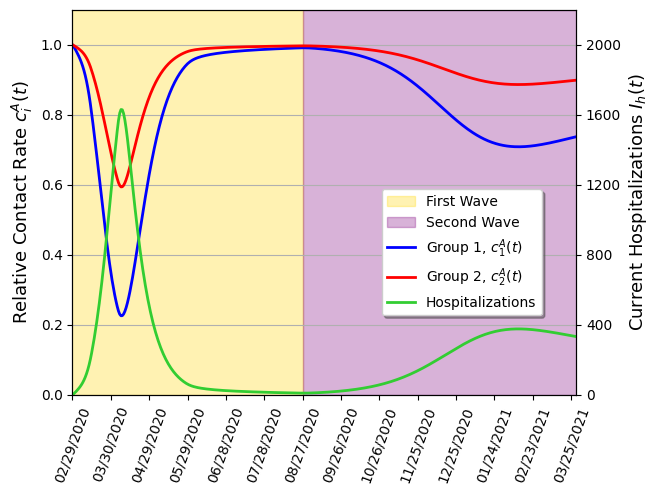}
    \caption{Profile of the relative contact rates ($c_i^A(t)$ for $i=1,2$) as compared with the number of hospitalized individuals over time ($I_h(t)$). This figure is generated by solving the 2-group behavior model \eqref{eq:n2model}, using the fixed and estimated parameters given in Tables \ref{tab:fixedparams} and \ref{tab:fittedparams_full}, respectively (to obtain the value of $I_h(t)$ for each $t$), together with the estimated values of the parameters $a_1$ and $a_2$ (given in Table \ref{tab:fittedparams_full}) to compute the relative contact rate multipliers for individuals in Group 1 ($c^A_1(t)$, blue curve) and Group 2 ($c^A_2(t)$, red curve) over time. The number of individuals who are currently hospitalized ($I_h(t)$) is depicted by the green curve. The time periods of the first and second waves are represented by the shaded yellow and purple regions, respectively.}
    \label{fig:contactrate_time}
\end{figure}
\noindent Table \ref{tab:fittedparams_full} also shows that the rate at which individuals in Group 2 change their behavior and become {\it more cautious} ($c^B_{12}$) is much higher ($3-$fold) than the rate at which cautious individuals (i.e., those in Group 1) change their behavior and become less cautious ($c^B_{21}$). Although the estimated value of the influence parameter $c^B_{12}$ is larger than that of the parameter $c^B_{21}$, the relative influence of Group 1 (given by $c^B_{12}\frac{N_1(t)}{N(t)}$) is initially small (due to the small initial size of Group 1; note that, in the formulation of the model \eqref{eq:n2model}, it was assumed that the majority of the population is in Group 2 during the early stages of the epidemic) when compared to the relative influence of Group 2 (given by $c^B_{21}\frac{N_2(t)}{N(t)}$) throughout the first wave. However, as Group 1 begins to increase in size  (particularly during the summer of 2020), the relative sizes of the behavioral groups begin to change rapidly (see the shaded white and purple regions of Figure \ref{fig:pop_dynam}), with Group 1 ultimately overtaking Group 2 around February of 2021 (with 70\% of the population ending up in Group 1 by March 2021, while the rest of the population (30\%) remained in Group 2). In other words, Figure \ref{fig:pop_dynam} shows that, despite its initial small size, Group 1 ultimately became dominant (i.e., the vast majority of individuals in the community significantly reduced their risk-taking behavior and adhered to public health interventions) by the second wave of the pandemic. Finally, Table \ref{tab:fittedparams_full} also shows that the efficacy of the lockdown measures implemented in New York City during phase 1 was low (with efficacy, $1-\theta_{l,1}$, estimated to be around $0.25$; that is, 25\%), but increased dramatically to nearly 100\% during the second phase (i.e., $1-\theta_{l,2}\approx 1$). The efficacy decreased during subsequent phases (down to 85\% and 70\% during phases 3 and 4, respectively). Thus, Table \ref{tab:fittedparams_full} confirms that non-pharmaceutical interventions (NPIs) and mitigation measures were not effectively implemented in New York City during the first phase (due, perhaps, in large part, to the lack of clarity and consistency in messaging, and pandemic preparedness \cite{nycfail}). However, NPIs were effectively implemented by phase 2, although this was not generally maintained during phases 3 and 4 (potentially due to increasing levels of interventions fatigue and spread of misinformation and disinformation about the pandemic \cite{nycfatigue}).\\
\begin{figure}
    \centering
\includegraphics[width=115mm,scale=1]{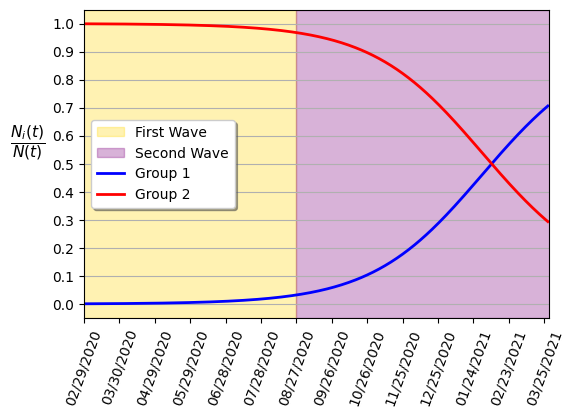}
    \caption{Profile of the relative group sizes ($N_i(t)/N(t)$ for $i=1,2$) over time. This figure is generated by solving the 2-group behavior model \eqref{eq:n2model}, using the fixed and estimated parameters given in Tables \ref{tab:fixedparams} and \ref{tab:fittedparams_full}, respectively (to obtain the values of $N_1(t)$ and $N_2(t)$ for each $t$) to compute  the population dynamics for Group 1 ($\frac{N_1(t)}{N(t)}$, blue curve) and Group 2 ($\frac{N_2(t)}{N(t)}$, red curve) over time. The time periods of the first and second waves are represented by the shaded yellow and purple regions, respectively.}
    \label{fig:pop_dynam}
\end{figure} 
\noindent It is worth stating that the estimated values of the four behavior-related parameters ($a_1$, $a_2$, $c^B_{12}$, and $c^B_{21}$) were obtained by fitting the 2-group behavior model \eqref{eq:n2model} to the 7-day moving hospitalization average data \cite{nycgov_data}.
Ideally, these parameters should be estimated by fitting this model with behavior data, which was not not widely collated and made publicly available until several months into the pandemic. As a result, 7-day running average hospitalization data was used to estimate the four behavioral parameters of the 2-group behavior model \eqref{eq:n2model}. However, as behavioral data (in the form of mask compliance on the subway in New York City \cite{behaviordatanyc}) became available starting in June 2020, the accuracy of the data fitting is further corroborated by comparing the results obtained (by simulating the model \eqref{eq:n2model} with parameters given in Tables \ref{tab:fixedparams} and \ref{tab:fittedparams_full}) for the relative contact rates with this behavioral data. The results obtained (depicted in Figure \ref{fig:noncompliance}) validate the fitted values of $a_1$ and $a_2$ by showing qualitatively similar curves between population contact rates ($c^A_1(t)$ for Group 1, $c^A_2(t)$ for Group 2, and the weighted average contact rate $N(t)(\frac{c^A_1(t)}{N_1(t)}+\frac{c^A_2(t)}{N_2(t)})$ for the full population) and the percentage of noncompliant (with respect to masking) individuals in the population. The  trends observed in each of these curves suggest that as attitudes towards the disease shift and individuals decrease their relative contacts (and become more cautious in general), the noncompliant portion of the population decreases.
\begin{figure}[H]
    \centering
    \includegraphics[width=0.9\linewidth]{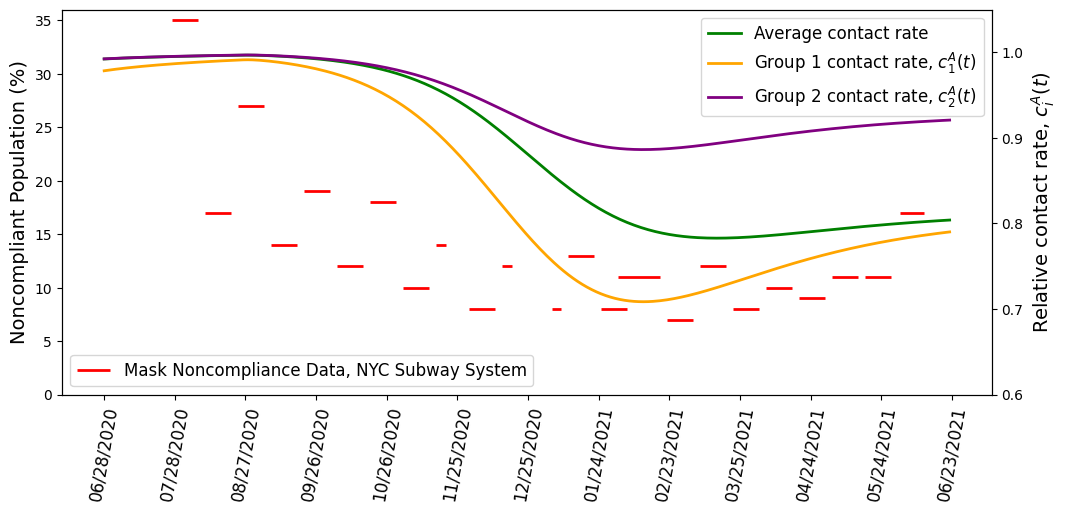}
    \caption{Profile of the contact rates ($c^A_1(t)$ and for $c^A_2(t)$) of individuals in the two behavioral groups, as a function of time, during the first two waves of the SARS-CoV-2 pandemic, superimposed on data for mask noncompliance on the New York City subway system \cite{behaviordatanyc}. Parameter values used in the simulations are as given in Tables \ref{tab:fixedparams} and \ref{tab:fittedparams_full}. Notation: purple curve represents the relative contact rate of individuals in the risk-taking group (Group 2, $c^A_2(t)$); gold curve represents the relative contact rate of individuals in the cautious group (Group 1, $c^A_1(t)$); green curve represents the population average relative contact rate ($\frac{N_1(t)}{N(t)}c^A_1(t)+\frac{N_2(t)}{N(t)}c^A_2(t)$). The profiles are compared with empirical data for mask non-compliance in New York City during the same period (shown in red dashed lines, with each dash corresponding to the data collection period).}
    \label{fig:noncompliance}
\end{figure} \vspace{-8mm}
\begin{center}
\begin{table}[H]
    \centering
\begin{tabular}{|p{1.5cm}||p{2cm}|}
     \hline
     Phase & $\RR_{c}$ \\
     \hline
     \hline
     Phase 1 & 4.045 \\ \hline
     Phase 2 & 0.016 \\ \hline
     Phase 3 & 0.929 \\ \hline
     Phase 4 & 1.539 \\ \hline
    \end{tabular}\vspace{0.5cm} \caption{Values of the control reproduction number $\RR_c$ for the 2-group behavior model \eqref{eq:n2model} (see equation \eqref{eq:RC}) during the four lockdown phases of the pandemic in New York City for model \eqref{eq:n2model}. Parameter values used in computing these values are as given in Tables \ref{tab:fixedparams} and \ref{tab:fittedparams_full}.}
       \label{tab:rcs_full}
\end{table}
\end{center} \vspace{-20mm}
{\noindent \bf Values of ${\mathbb R}_c$ for the 2-group behavior model \eqref{eq:n2model} during the four lockdown phases:} The baseline values of the fixed and fitted parameters in Tables \ref{tab:fixedparams} and \ref{tab:fittedparams_full} are now used to compute the values of the control reproduction number ($\RR_c$, given by Equation \eqref{eq:RC}) for New York City during phases 1 through 4 of the lockdown, and the results obtained are tabulated in Table \ref{tab:rcs_full}. This table shows a more pronounced outbreak during phase 1 (measured in terms of higher values of $\RR_c$; for this phase, $\RR_c\approx 4$). This value of $\RR_c$, which is in line with estimates from other modeling studies for the SARS-CoV-2 outbreak in New York City during the first wave \cite{r01,r02}, is to be expected for a community where the efficacy of the lockdown and other NPIs implemented is low (e.g., $1-\theta_{l,1}=0.25$, as estimated for New York City, during this period, in Table \ref{tab:fittedparams_full}). This resulted in a devastating first wave for New York City, which ran during the period between March 2020 and the end of July of 2020 (i.e., phases 1 and 2), causing $52,899$ hospitalizations and $23,590$ deaths \cite{nycgov_data}. Although the value of $\RR_c$ decreased dramatically (due to the stringent implementation of these control measures during phase 2), it quickly rebounded during phases 3 and 4 when the control measures were progressively relaxed (in particular, relaxation of lockdown measures during phase 4 was what resulted in the second wave, {\it albeit} the second wave was much milder than the first wave observed in New York City during phases 1 and 2). Finally, in order to quantify the importance of explicitly incorporating heterogeneous human behavior in accurately capturing the trajectory of the SARS-CoV-2 pandemic and making realistic predictions, a \textit{behavior-free} equivalent of the 2-group behavior model \eqref{eq:n2model}, obtained by setting all the behavior-related parameters of the model to zero, is now explored (by fitting the resulting behavior-free model with the same 7-day moving average of the daily hospitalization data for New York City during the first wave \cite{nycgov_data}, as discussed below). The objective is to compare which of the two models (behavior or behavior-free) best captures current, and predicts future, trajectory of the disease.

\subsection{Data fitting and parameter estimation for the behavior-free model \eqref{eq:bfmodel}}\label{sec:fitting,bf}
\noindent Setting the four behavior-related parameters ($a_1,a_2,c^B_{12},$ and $c^B_{21}$) of the 2-group behavior model \eqref{eq:n2model} to zero gives the one-group {\it behavior-free} (homogeneous) model, given by the system of nonlinear differential equations in Equation \eqref{eq:bfmodel} of Appendix A.  It should be noted that, unlike in the case of the 2-group behavior model \eqref{eq:bfmodel}, the 1-group (homogeneous) behavior-free model \eqref{eq:bfmodel} has no reduction in contacts due to the level of SARS-CoV-2 hospitalizations in the community. Figure \ref{fig:predict_bf} depicts the results of fitting the behavior-free model \eqref{eq:bfmodel} with the 7-day running average of daily hospitalization data (the values of the estimated parameters obtained from fitting this model to the data are tabulated in Table \ref{tab:fittedparams_bf}). This figure \ref{fig:predict_bf} shows that, although the behavior-free model \eqref{eq:bfmodel} fits the data for the first wave of the pandemic in New York City reasonably well, it fails to accurately capture the dynamics of the second wave. Furthermore, unlike the 2-group behavior model \eqref{eq:n2model}, which accurately captures the disease dynamics during both waves (see Figure \ref{fig:predict_full}), the behavior-free model \eqref{eq:bfmodel} over-estimated the recorded 7-day running average of daily hospitalizations during the winter of 2020 to the spring 2021 (see green curve in Figure \ref{fig:predict_bf}). Specifically, the behavior-free model over-estimated the reported daily 7-day running average of hospitalizations at the end of the second wave (on March 25, 2021) by 104\%. On the other hand, the 2-group behavior model \eqref{eq:n2model} was within 10\% of the recorded 7-day running average of hospitalizations for this date. Hence, it can be concluded, by comparing Figures \ref{fig:predict_full} and \ref{fig:predict_bf}, that epidemiological models that do not explicitly incorporate the effect of human behavior changes may fail to accurately capture the trajectory (as well as predict the future course) of the SARS-CoV-2 pandemic in a community, such as New York City during the first two waves of the pandemic. This is in line with some recent modeling studies which highlight the importance of incorporating human behavior into models for the spread of respiratory pathogens \cite{ge_behavior_binod, gg_simplebehavior_lejune, gg_heterogeneous_espinoza}.
\vspace*{2mm}\ \\ \noindent It is also worth noting from Table \ref{tab:fittedparams_bf} that the estimated value of the NPI implementation parameter during phase 1 ($\theta_{l,1}$), generated from fitting the behavior-free model \eqref{eq:bfmodel}, is much lower ($\theta_{l,1}=0.45$) than the corresponding value obtained from fitting the 2-group behavior model \eqref{eq:n2model} (where $\theta_{l,1}=0.74$).  In other words, in order for the behavior-free model \eqref{eq:bfmodel} to fit the first wave data accurately, the estimated efficacy of NPIs in curtailing the pandemic in New York City during the first phase had to be quite high (about 55\%). This is significantly higher than the estimated efficacy of this parameter obtained from fitting the behavior model with the first wave data (which was about 26\%). It is reasonable to conclude that the behavior-free model may have over-estimated the efficacy of the NPI implementation during the early stages of the first wave of the SARS-CoV-2 pandemic in New York City (since media reports suggest considerable inconsistency and lack of clarity in messaging in New York City during this period \cite{nycfail}). Thus, in addition to over-estimating the hospitalization burden during the second wave, the behavior-free model may have over-estimated the effectiveness of NPI measures implemented in New York City during the first phase of the SARS-CoV-2 pandemic. However, the estimated values of the NPI parameters during phases 2 through 4 ($\theta_{l,2},\theta_{l,3},\theta_{l,4}$), obtained from fitting the behavior and behavior-free models, were similar (see Tables \ref{tab:fittedparams_full} and \ref{tab:fittedparams_bf}). In other words, both the 2-group behavior model and the 1-group behavior-free model gave similar (and possibly reasonable) estimates for the effectiveness of NPIs during subsequent phases of the SARS-CoV-2 pandemic (from April 4, 2020 to March 25, 2021) in New York City. 
\vspace*{2mm}\ \\ \noindent {\noindent \bf Values of ${\mathbb R}_c$ for the behavior-free model \eqref{eq:bfmodel} during the four lockdown phases:} Here, too, the fixed and fitted parameters for the behavior-free model \eqref{eq:bfmodel}, tabulated in Tables \ref{tab:fixedparams} and \ref{tab:fittedparams_bf}, are used to compute the values of the control reproduction number (${\mathbb R}_c$) associated with the behavior-free model, given by Equation \eqref{eq:RC}.  The results obtained are given in Table \ref{tab:rcs_bf}. It should be noted here that the aforementioned low estimate of the NPI parameter (i.e., high efficacy of NPIs) during phase 1 of the pandemic reduces the value of the control reproduction number (${\mathbb R}_c$) to 2.43, as against the value ${\mathbb R}_c=4.045$ estimated from fitting the 2-group behavior model (compare Tables \ref{tab:rcs_full} and \ref{tab:rcs_bf}).
\begin{center}
    \begin{table}[H]
    \centering
\begin{tabular}{|p{1.8cm}||p{4.5cm}|}
     \hline
     Parameter & Baseline value  \\
     \hline
     \hline
       $\theta_{l,1}$ & 0.45  (dimensionless)  \\ \hline
       $\theta_{l,2}$ & 0.0153 (dimensionless)  \\ \hline
       $\theta_{l,3}$ & 0.178  (dimensionless)  \\ \hline
       $\theta_{l,4}$ & 0.248  (dimensionless)  \\ \hline
    \end{tabular}\caption{ Baseline values of the estimated parameters of the behavior-free model \eqref{eq:bfmodel} (with fixed parameters as given in Table \ref{tab:fixedparams}) obtained by fitting the model to the 7-day running average data of daily hospitalizations in New York City \cite{nycgov_data} for the period from February 29, 2020 until July 28, 2020 using the BFGS algorithm.}
       \label{tab:fittedparams_bf}
\end{table}
\end{center}\vspace{-10mm}
\begin{figure}[H]
    \centering
\includegraphics[width=160mm,scale=1]{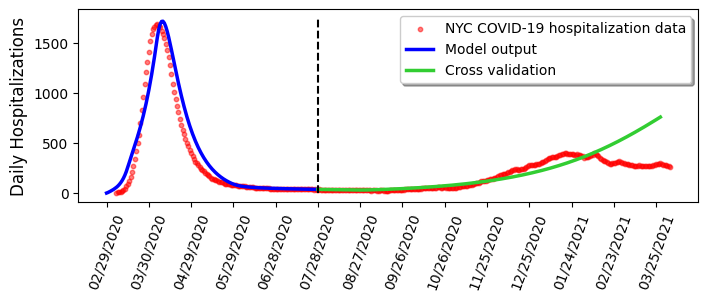}
    \caption{Data fitting and cross validation of the behavior-free model \eqref{eq:bfmodel}, using the seven-day running average daily hospitalization data for New York City (red dots) \cite{nycgov_data}, carried out using the BFGS algorithm. The values of the fixed parameters of the behavior-free model, used in fitting it to the data, are given in Table \ref{tab:fixedparams}, and the values of the estimated parameters of the model (obtained from the fitting) are given in Table \ref{tab:fittedparams_bf}. The fitting (blue curve) is carried out using the data for the first 150 days after the index case in New York City (i.e., from February 29, 2020 until July 28, 2020), while the cross-validation (green curve) is carried out using the following eight months worth of data (i.e., from July 29, 2020 until March 25, 2021). The four unknown parameters of the behavior-free model ($\theta_{l,1},\theta_{l,2},\theta_{l,3},\theta_{l,4}$) were estimated from fitting the model with the data, and their estimated values are tabulated in Table \ref{tab:fittedparams_bf}.}
    \label{fig:predict_bf}
\end{figure}
\begin{center}
\begin{table}[H]
    \centering
\begin{tabular}{|p{1.5cm}||p{1.5cm}|}
     \hline
     Phase & $\RR_{c}$ \\
     \hline
     \hline
     Phase 1 & 2.43\\ \hline
     Phase 2 & 0.083 \\ \hline
     Phase 3 & 0.961 \\ \hline
     Phase 4 & 1.339 \\ \hline
    \end{tabular}\vspace{0.5cm} \caption{Values of the control reproduction number $\RR_c$ for the behavior-free model \eqref{eq:bfmodel} during the four lockdown phases of the SARS-CoV-2 pandemic in New York City. Parameter values used to compute the values of $\RR_c$ are as given in Tables \ref{tab:fixedparams} and \ref{tab:fittedparams_bf}.}
       \label{tab:rcs_bf}
\end{table}
\end{center} \vspace{-20mm}
\noindent Since the 2-group behavior model \eqref{eq:n2model} has been shown to be more realistic in capturing the correct trajectory and burden of the pandemic (in comparison to the more parsimonious behavior-free equivalent \eqref{eq:bfmodel}), the remainder of this study will focus on the 2-group behavior model \eqref{eq:n2model}. To determine what parameters have the greatest impact on the 2-group behavior model \eqref{eq:n2model} (measured in terms of some chosen public health outcome metrics, such as peak daily hospitalizations and cumulative mortality), detailed global sensitivity analysis of the 2-group behavior model, with respect to the chosen health outcome metrics, will now be carried out.
\subsection{Sensitivity analysis for the 2-group behavior model}\label{sec:sensitivity}
\noindent Having calibrated and validated the 2-group behavior model \eqref{eq:n2model}, the next step is to conduct a global sensitivity analysis to determine which of its 18 parameters have the highest impact on a chosen response function. Since the reproduction thresholds of the 2-group behavior model (denoted by ${\mathbb R}_c$ and ${\mathbb R}_0$, given by Equations \eqref{eq:RC} and \eqref{eq:R0}, respectively) do not include the behavior-related parameters of the model, the sensitivity analysis will be carried out with respect to two distinct disease burden-related metrics, namely (a) peak daily SARS-CoV-2 hospitalizations and (b) cumulative mortality, during the first and second waves of the pandemic in New York City. In particular, the sensitivity analysis will be conducted using partial rank correlation coefficients (PRCCs) \cite{prcc1, prcc2, prcc3}, which entails defining each of the 18 parameters of the 2-group behavior model \eqref{eq:n2model} as a distribution (typically uniform) over a range (or an interval).  Following \cite{prcc_intervals}, each of these parameter ranges is determined by taking 40\% to the left and to the right of the baseline value of the corresponding parameter (tabulated in Tables \ref{tab:fixedparams} and \ref{tab:fittedparams_full}). The ranges are divided into 1,000 sub-intervals of equal length and the parameter set is drawn from this set without replacement. This leads to an $18\times 1,000$ parameter matrix (or a {\it hypercube}). PRCC values range from $-1$ to $1$ (and parameters with positive PRCC values are said to be positively correlated with the  chosen response function, while parameters with negative PRCC values are said to be negatively correlated with the response function). A higher magnitude PRCC value (say, a value greater than $0.5$ or less than $-0.5$) signifies a stronger correlation \cite{prcc3}.
\vspace*{2mm}\ \\ \noindent The results for the sensitivity analysis, obtained using the baseline values and ranges of the parameters of the 2-group behavior model given in Table \ref{tab:cube}, are tabulated, for two snapshots of time during the first two waves of the pandemic in New York City (namely, towards the end of the first wave, April 20, 2020; and towards the end of the second wave, February 21, 2021), in Table \ref{tab:prcc}.
\begin{center}
    \begin{table}[H]
    \centering
\begin{tabular}{|p{2cm}||p{4.5cm}|p{3cm}|}
     \hline
     Parameter & Baseline value & Range \\
     \hline
     \hline
     $\theta_{l,1}$ & 0.74664 (dimensionless) & [0.448,1.045]   \\ \hline
     $\theta_{l,2}$ & 0.00133 (dimensionless) & [0.0008,0.0019]   \\ \hline
     $\theta_{l,3}$ & 0.15427 (dimensionless) & [0.0926,0.216]  \\ \hline
     $\theta_{l,4}$ & 0.30970 (dimensionless) & [0.186,0.434]  \\ \hline
     $c^B_{12}$ & 1/30 day$^{-1}$ & [0.02,0.0467] \\ \hline
     $c^B_{21}$ & 1/90 day$^{-1}$& [0.00667,0.0156]   \\ \hline
     $a_1$ & 8,000 (dimensionless)& [4,800,11,200]  \\ \hline
     $a_2$ & 2,800 (dimensionless)& [1,680,3920]  \\ \hline
     $\xi$ & 1/180 day$^{-1}$& [0.00333,0.00778]   \\ \hline     
     $\sigma_i$ & 1/14 day$^{-1}$& [0.0429,0.1]  \\ \hline
     $\sigma_e$ & 0.25 day$^{-1}$& [0.15,0.35]   \\ \hline
     $\gamma_a$ & 1/9 day$^{-1}$& [0.0667,0.1556]  \\ \hline
     $\gamma_h$ & 0.1 day$^{-1}$& [0.06,0.14]   \\ \hline
     $\delta_h$ & 0.41 day$^{-1}$& [0.246,0.574]  \\ \hline
     $\beta_a$ & 0.625 day$^{-1}$& [0.375,0.875]  \\ \hline
     $\beta_i$ & 0.375 day$^{-1}$& [0.225,0.525]   \\ \hline
     $r$ & 0.6 (dimensionless)& [0.36,0.84]   \\ \hline
     $q$ & 0.05 (dimensionless)& [0.03,0.07]  \\ \hline
    \end{tabular}\caption{Baseline values and ranges for the 18 parameters of the  2-group behavior model \eqref{eq:n2model}.  The baseline values are as given in  Tables \ref{tab:fixedparams} and \ref{tab:fittedparams_full}, and the range of each parameter is determined by taking 40\% to the left and right of its baseline value \cite{prcc1, prcc2, prcc3}.}\label{tab:cube}
    \end{table}
\end{center}\vspace{-20mm}
\begin{center}
    \begin{table}[H]
    \centering
\begin{tabular}{|p{2cm}||p{1.7cm}|p{1.6cm}|}
\multicolumn{3}{c}{\textbf{(a)}}\\
\multicolumn{3}{c}{ }\\
     \hline
     Parameter & Apr 2020 & Feb 2021  \\
     \hline
     \hline
     $\theta_{l,1}$ & 0.806 & 0.262  \\ \hline
     $\theta_{l,2}$ & 0.165 & 0.0277  \\ \hline
     $\theta_{l,3}$ & -0.0524 & 0.241  \\ \hline
     $\theta_{l,4}$ & 0.0670 & 0.744  \\ \hline
     $c^B_{12}$ & -0.261 & -0.228 \\ \hline
     $c^B_{21}$ & -0.00475 & -0.00195  \\ \hline
     $a_1$ & 0.0285 & 0.0147  \\ \hline
     $a_2$ & -0.185 & -0.152  \\ \hline
     $\xi$ & 0.188 & 0.0897  \\ \hline     
     $\sigma_i$ & -0.125 & -0.864  \\ \hline
     $\sigma_e$ & 0.786 & -0.0545  \\ \hline
     $\gamma_a$ & -0.448 & -0.711 \\ \hline
     $\gamma_h$ & -0.104 & -0.187  \\ \hline
     $\delta_h$ & -0.612 & -0.0580  \\ \hline
     $\beta_a$ & 0.849 & 0.561  \\ \hline
     $\beta_i$ & 0.842 & 0.535  \\ \hline
     $r$ & -0.130 & -0.290  \\ \hline
     $q$ & 0.576 & -0.041  \\ \hline
    \end{tabular} \hspace{6mm}
    \begin{tabular}{|p{2cm}||p{1.7cm}|p{1.6cm}|}
    \multicolumn{3}{c}{\textbf{(b)}}\\
    \multicolumn{3}{c}{ }\\
     \hline
     Parameter & Apr 2020 & Feb 2021  \\
     \hline
     \hline
     $\theta_{l,1}$ & 0.665 & 0.661  \\ \hline
     $\theta_{l,2}$ & -0.102 & 0.0156  \\ \hline
     $\theta_{l,3}$ & -0.0616 & 0.305  \\ \hline
     $\theta_{l,4}$ & 0.138 & 0.572  \\ \hline
     $c^B_{12}$ & 0.0266 & -0.152 \\ \hline
     $c^B_{21}$ & -0.0396 & -0.125  \\ \hline
     $a_1$ & -0.169 & -0.127  \\ \hline
     $a_2$ & -0.126 & -0.141  \\ \hline
     $\xi$ & -0.0535 & 0.0738  \\ \hline     
     $\sigma_i$ & 0.160 & -0.778  \\ \hline
     $\sigma_e$ & 0.796 & 0.470  \\ \hline
     $\gamma_a$ & -0.268 & -0.539 \\ \hline
     $\gamma_h$ & -0.167 & 0.127  \\ \hline
     $\delta_h$ & 0.518 & 0.436  \\ \hline
     $\beta_a$ & 0.772 & 0.606  \\ \hline
     $\beta_i$ & 0.785 & 0.691  \\ \hline
     $r$ & 0.343 & 0.0223  \\ \hline
     $q$ & 0.690 & 0.219  \\ \hline
    \end{tabular}\caption{Partial rank correlation coefficients (PRCCs) of the 18 parameters of the 2-group behavior model \eqref{eq:n2model} at two time periods (April 2020 and February 2021) during the SARS-CoV-2 pandemic in New York City, with respect to: (a) peak daily SARS-CoV-2 hospitalizations and (b) cumulative SARS-CoV-2 mortality. The baseline values and ranges of the parameters used in generating the PRCC values are given in Table \ref{tab:cube}.}\label{tab:prcc}
    \end{table}
\end{center} \vspace{-20mm}
\subsubsection{Sensitivity analysis with respect to the peak daily hospitalization}\label{sec:SA_peakhosp}
\noindent Table \ref{tab:prcc}(a) shows that the top-five parameters of the 2-group behavior model \eqref{eq:n2model} that have the most influence on peak daily hospitalizations during the first wave of the SARS-CoV-2 pandemic in New York City (April 20, 2020) are:
\begin{itemize} 
\item[(i)] The maximum effective contact rate for asymptomatic infectious individuals ($\beta_a$; PRCC $=$ +0.849). 
\item[(ii)] The maximum effective contact rate for symptomatic infectious individuals ($\beta_i$; PRCC = +0.842).
\item[(iii)] The parameter for the measure of effectiveness of government-mandated community lockdowns (and other NPIs) during Phase 1 of the pandemic in New York City ($\theta_{l,1}$; PRCC = +0.806). 
\item[(iv)] The rate of progression of individuals leaving the exposed class ($\sigma_e$; PRCC = +0.786).
\item[(v)] The proportion of symptomatic individuals that are hospitalized ($q$; PRCC = +0.576).
\end{itemize}
\noindent Thus, peak daily hospitalizations in New York City during the first wave of the pandemic can be significantly reduced by implementing intervention and mitigation strategies that decrease maximum effective contact rates ($\beta_a$ and $\beta_i$) and the proportion of symptomatic individuals that are hospitalized ($q$), and increase the incubation period ($1/\sigma_e$) and the efficacy of Phase 1 government-mandated community lockdowns and NPIs ($1-\theta_{l,1}$). The effective contact rate parameters ($\beta_a$ and $\beta_i$) can be decreased by, for instance, reducing contacts, large scale random testing, quarantine/isolation, and the use of face masks in public. Similarly, $\theta_{l,1}$ can be increased by increasing the coverage and effectiveness of the NPIs implemented. The parameter $\sigma_e$ can be reduced by any biological mechanism that extends the incubation or latency period of the disease.  Finally, the parameter $q$ can be decreased by early detection and effective treatment of symptomatic individuals.

\noindent Table \ref{tab:prcc}(a) also shows that the top-five parameters that have the most influence on this response metric during the second wave (February 21, 2021) are: 
\begin{itemize} 
\item[(i)] The rate of progression of individuals leaving the symptomatic infectious class ($\sigma_i$; PRCC = -0.864).
\item[(ii)] The parameter for the measure of effectiveness of the government-mandated community lockdowns (and
other NPIs) during Phase 4 of the pandemic in New York City
($\theta_{l,4}$; PRCC = +0.744).
\item[(iii)] The rate at which asymptomatically infectious individuals recover ($\gamma_a$; PRCC = -0.711).
\item[(iv)] The maximum effective contact rate for asymptomatic infectious individuals ($\beta_a$; PRCC = +0.561).
\item[(v)] The maximum effective contact rate for symptomatic infectious individuals ($\beta_i$; PRCC = +0.535).
\end{itemize}
\noindent Thus, two of the top-five parameters that affect peak daily hospitalizations during the first wave ($\beta_a$ and $\beta_i$) also remain relevant during the second wave, while the parameters $\theta_{l,1}$ (for the measure of effectiveness of the government-mandated community lockdowns in Phase 1), $\sigma_e$ (for the rate at which individuals leave the exposed class), and $q$ (for the proportion of symptomatic individuals that are hospitalized) diminish in relevance with respect to peak daily hospitalizations during the second wave. Unlike during the first wave, three additional parameters (namely, $\sigma_i$, for the rate of progression of individuals leaving the symptomatic infectious class; $\theta_{l,4}$, for the measure of effectiveness of the government-mandated community lockdowns in Phase 4; and $\gamma_a$, for the rate at which asymptomatically infectious individuals recover), become very relevant during the second wave (where their PRCC values increased in comparison to their values during the first wave by about 7-fold, 11- fold, and 1.6-fold, respectively). In other words, this study shows that, during the second wave of the SARS-CoV-2 pandemic in New York City, the peak daily hospitalizations are significantly decreased with increasing values of $\sigma_i$ (the rate at which symptomatically infectious individuals leave the symptomatic class), $1-\theta_{l,4}$ (the efficacy of Phase 4 government-mandated community lockdowns and NPIs) and $\gamma_a$ (the rate at which asymptomatically infectious individuals recover). Thus, peak daily hospitalizations in New York City during the second wave of the pandemic can be significantly reduced using the approaches discussed above, as well as by implementing intervention and mitigation strategies that minimize the amount of time the average individual spends in the symptomatic class ($1/\sigma_i$). This could be achieved by effective treatment of symptomatic individuals.
\subsubsection{Sensitivity analysis with respect to cumulative mortality}
\noindent It can be seen from Table \ref{tab:prcc}(b) that the top-five parameters of the behavior model \eqref{eq:n2model} that have the most influence on cumulative mortality during the first wave of the SARS-CoV-2 pandemic in New York City  (April 20, 2020) are:
\begin{itemize} 
\item[(i)] The rate of progression of individuals leaving the exposed class ($\sigma_e$; PRCC = +0.796).
\item[(ii)] The maximum effective contact rate for symptomatic infectious individuals ($\beta_i$; PRCC = +0.785).
\item[(iii)] The maximum effective contact rate for asymptomatic infectious individuals ($\beta_a$; PRCC = +0.772).
\item[(iv)] The proportion of symptomatic individuals that are hospitalized ($q$; PRCC = +0.690).
\item[(v)] 
The parameter for the measure of effectiveness of government-mandated community lockdowns (and
other NPIs) during Phase 1 of the pandemic in New York City
($\theta_{l,1}$; PRCC = +0.665).
\end{itemize}
As these are the same parameters as listed in Section \ref{sec:SA_peakhosp}, for influence on peak daily hospitalizations at the end of the first wave of the SARS-CoV-2 pandemic in New York City, similar approaches as listed there would lead to a decrease in first wave mortality. \\\vspace{-4mm}

\noindent Table \ref{tab:prcc}(b) also shows that the top-five parameters that have the most influence on this response metric during the second wave (February 21, 2021) are:  
\begin{itemize} 
\item[(i)] The rate of progression of individuals leaving the symptomatic infectious class ($\sigma_i$; PRCC = -0.778).
\item[(ii)] The maximum effective contact rate for symptomatic infectious individuals ($\beta_i$; PRCC = +0.691).
\item[(iii)] 
The parameter for the measure of effectiveness of government-mandated community lockdowns (and
other NPIs) during Phase 1 of the pandemic in New York City
($\theta_{l,1}$; PRCC = +0.661).
\item[(iv)] The maximum effective contact rate for asymptomatic infectious individuals ($\beta_a$; PRCC = +0.606).
\item[(v)] 
The parameter for the measure of effectiveness of government-mandated community lockdowns (and
other NPIs) during Phase 4 of the pandemic in New York City
($\theta_{l,4}$; PRCC = +0.572).
\end{itemize}
Thus, cumulative mortality in New York City during the second wave of the SARS-CoV-2 pandemic can be significantly reduced using the strategies mentioned in Section \ref{sec:SA_peakhosp}, as well as by implementing strategies to increase the efficacy of Phase 4 government-mandated community lockdown (and other non-pharmaceutical intervention measures), given by ($1-\theta_{l,4}$). This could be achieved through a slower reopening process, in tandem with public health campaigning to combat mask fatigue.\\\\
\noindent Bar graphs of the PRCC values of the 2-group behavior model \eqref{eq:n2model}, generated for the two  disease metrics discussed above, are depicted in Figures \ref{fig:prcc_hosp} and \ref{fig:prcc_deaths} of Appendix C.

\subsection{Numerical simulations of the 2-group behavior model \eqref{eq:n2model}}\label{sec:numsim}
\noindent In this section, the population-level impact of the four behavioral parameters ($a_1$ and $a_2$, which provide a measure of the sensitivity individuals in groups 1 and 2, respectively, show in response to hospitalization levels in the community; and $c^B_{12}$ and $c^B_{21}$, which provide a measure of the influence Group 1 holds over Group 2, and Group 2 holds over Group 1, respectively) on the trajectory and burden of the SARS-CoV-2 pandemic in New York City during the first two waves will be assessed. Such simulations will provide a scientific basis for determining which of the factors that induce behavioral change (namely, disease-motivated, as measured by $a_1$ and $a_2$; and peer influence-motivated, as measured by $c^B_{12}$ and $c^B_{21}$) is more influential in affecting the trajectory and burden of the disease during the two waves. Determining the behavioral change factors that have the most influence on disease trajectory and burden will be critical in shaping public health policy and messaging, as identified factors may be targeted for public health intervention (for instance, if the hospitalization-induced behavioral change factor is determined to be more influential than the peer influence-induced behavioral change factor, then public health messaging resources should be prioritized for making hospitalization-related information widely available to the public). The objective of this section is achieved by simulating the 2-group behavior model \eqref{eq:n2model} with the fixed and estimated parameters in Tables \ref{tab:fixedparams} and \ref{tab:fittedparams_full}, respectively, for various values of the four behavioral change parameters as described below.


\subsubsection{Impact of hospitalization-induced behavior change parameters ($a_1$ and $a_2$)}\label{sec:ai_impact}
\noindent The simulation results obtained from running the 2-group behavior model \eqref{eq:n2model} with various values of the hospitalization-induced behavior change parameters $a_1$ and $a_2$ (using the parameters in Tables \ref{tab:fixedparams} and \ref{tab:fittedparams_full}), depicted in Figure \ref{fig:ai_impact}, show that the peak daily hospitalizations during the first wave of the SARS-CoV-2 pandemic in New York City dramatically decrease with increasing values of $a_2$ (see Figure \ref{fig:ai_impact}(a)). This figure also shows that changes in $a_1$ have little to no effect on the peak daily hospitalizations during the first wave. The dominance of the contact rate modifier for Group 2 ($a_2$) over that of Group 1 ($a_1)$, on decreasing the peak daily hospitalizations during the first wave, is likely due to the fact that the size of Group 2 was larger than that of Group 1 during the first wave (i.e., $N_2(t)>N_1(t)$ for all $t\ge0$ during the first wave). For example, this figure (\ref{fig:ai_impact}(a)) shows that a 25\% reduction in the peak daily hospitalizations during the first wave can be attained if the baseline value of $a_2$ is increased by 15\% (i.e., $a_2$ is increased from the baseline value of 2,800 to 3,200). A similar result was obtained for the peak daily hospitalizations during the second wave (Figure \ref{fig:ai_impact}(b)). Furthermore, this figure shows that, despite the decrease in the relative size of Group 2 ($\frac{N_2(t)}{N(t)}$) between the first two waves of the pandemic in New York City (see Figure \ref{fig:pop_dynam}), the contact rate modifier for Group 2 ($a_2$) still has a significant effect on the peak daily hospitalizations during the second wave. This result is in line with the sensitivity analysis conducted in Section \ref{sec:SA_peakhosp}, where the PRCC value of $a_2$ with respect to the peak daily hospitalizations increased (in magnitude) from -0.185 in April 2020 (i.e., during the first wave) to -0.443 in February 2021 (i.e., during the second wave). However, Figure \ref{fig:ai_impact}(b) also shows that behavior change by individuals in Group 1 has a more pronounced impact on reducing the peak daily hospitalizations during the second wave, in comparison to its impact during the first wave (compare Figures \ref{fig:ai_impact}(a) and (b)). \\
\indent Simulations for the effect of the hospitalization-induced behavior change parameters ($a_1$ and $a_2$) on the cumulative SARS-CoV-2 mortality during the first wave (depicted in Figure \ref{fig:ai_impact}c) show a similar pattern to their effect on the peak daily hospitalizations during the first wave (depicted in Figure \ref{fig:ai_impact}a). For the cumulative mortality metric, Figure \ref{fig:ai_impact}c further shows that achieving a 25\% reduction in cumulative SARS-CoV-2 mortality during the first wave would require at least a 25\% increase (from baseline value of 2,800 to 3,500) in the value of the contact rate modifier for Group 2, $a_2$ (as against the 15\% increase in the baseline value of the same parameter to achieve the 25\% reduction in peak daily hospitalizations for the scenario depicted in Figure \ref{fig:ai_impact}(a)). Thus, these simulations show that, during the first wave of the pandemic, a more pronounced positive behavior change in Group 2 (i.e., increase in the value of $a_2$) will be needed to significantly minimize the cumulative SARS-CoV-2 mortality in New York City during this wave. However, during the second wave, a significantly lower threshold value of $a_2$ exist ($a_2\approx 2,200$), above which the aforementioned significant reduction in the cumulative mortality can be achieved (Figure \ref{fig:ai_impact}d).  Furthermore, there exists a region in the $a_1-a_2$ plane ($4,100\leq a_1\leq 8,000, 2,200\leq a_2\leq 4,000$, see red box in Figure \ref{fig:ai_impact}d) within which an increase in $a_1$ (i.e., positive behavior within Group 1 induced by the level of SARS-CoV-2 hospitalizations) leads to an increase in the cumulative SARS-CoV-2 mortality. This  may be due to the fact that, in this region in the $a_1-a_2$ plane (unlike in all other regions in Figure \ref{fig:ai_impact}d), an increase in $a_1$ prolongs the duration of the second wave (see Figure \ref{fig:ai_w2len} in Appendix C). 
\begin{figure}
    \centering
    \includegraphics[width=175mm]{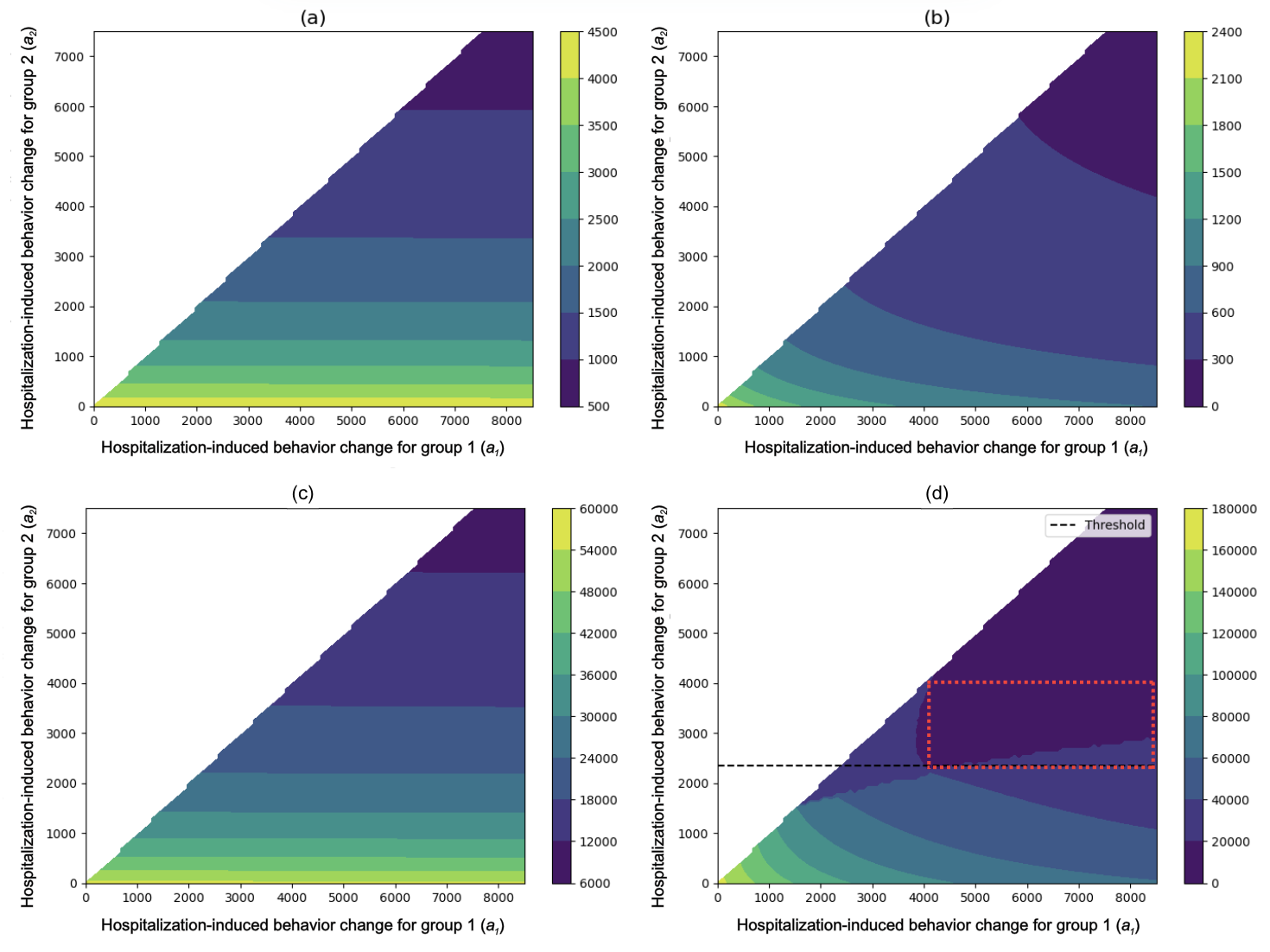}
    \caption{Effect of the hospitalizations-induced behavior change parameters, $a_1$ and $a_2$, on the peak daily hospitalizations and cumulative mortality in New York City during the first two waves of the SARS-CoV-2 pandemic.  Simulations of the behavior model \eqref{eq:n2model} with the baseline parameter values in Table \ref{tab:fixedparams} and \ref{tab:fittedparams_full}, and various values of $a_1$ and $a_2$.
    (a) Effect of $a_1$ and $a_2$ on the peak daily hospitalizations during the first wave.  (b) Effect of $a_1$ and $a_2$ on the peak daily hospitalizations during the second wave. (c) Effect of $a_1$ and $a_2$ on the cumulative mortality during the first wave. (d) Effect of $a_1$ and $a_2$ on the cumulative mortality during the second wave.}
    \label{fig:ai_impact}
\end{figure} 
\subsubsection{Impact of peer influence-induced behavioral change parameters ($c^B_{12}$ and $c^B_{21}$)}\label{sec:cij_impact} 
The simulation results obtained from running the 2-group behavior model \eqref{eq:n2model} with various values of the influence-related behavior change parameters, $c^B_{12}$ and $c^B_{21}$ (using the parameters in Tables \ref{tab:fixedparams} and \ref{tab:fittedparams_full}), are depicted in Figure \ref{fig:cij_impact}. In particular, Figure \ref{fig:cij_impact}(a) shows that the peak daily hospitalizations during the first wave of the SARS-CoV-2 pandemic in New York City are far more impacted by the behavior of Group 2 (measured in terms of the average time it takes an individual in Group 2 to change their behavior and move to Group 1, as a result of interaction with individuals in Group 1, given by $1/c^B_{12}$) than it is by the behavior of Group 1 (measured in terms of the average time it takes an individual in Group 1 to change their behavior and move to Group 2, as a result of interaction with individuals in Group 2, given by $1/c^B_{21}$).  This is in line with what was observed in the corresponding simulations for hospitalization-induced behavioral changes (as shown in Figure \ref{fig:ai_impact}(a)). Furthermore, this figure shows that, if the average duration for individuals in Group 2 to move to Group 1 is not quite short (i.e., $1/c^B_{12}<5$ days), then the value of peak daily hospitalizations is essentially fixed at 1,800 hospitalized individuals at the height of the first wave of the pandemic. It is worth noting from Figure \ref{fig:cij_impact}(a) that, even in the case where the change of behavior from Group 2 to Group 1 is almost instantaneous ($1/c^B_{12}\leq1 $ day), peak daily hospitalizations cannot be reduced above 50\% of the aforementioned baseline value of 1,800. Thus, by comparing Figures \ref{fig:ai_impact}(a) and \ref{fig:cij_impact}(a), it is clear that a greater reduction in the number of peak daily hospitalizations is achieved by increases in the values of the hospitalization-induced behavioral change parameters ($a_1$, $a_2$) than by increases in the values of the influence-motivated behavioral change parameters ($c^B_{12}$ and $c^B_{21}$). In other words, control strategies that focus on increasing positive behavior change with respect to hospitalization-motivated behavioral changes is more effective (in reducing the peak daily hospitalizations) than strategies aimed at increasing positive change with respect to influence-motivated behavioral changes during the first wave of the pandemic in New York City. During the second wave of the pandemic, the simulation results show that greater reduction in the peak daily hospitalizations (compared to during the first wave) can be achieved even if it takes, on average, up to 20 days for individuals in Group 2 to change their behavior and move to Group 1 (see Figure \ref{fig:cij_impact}(b)). In other words, peer influence has far more effect on the peak daily hospitalizations during the second wave than during the first wave (compare Figures \ref{fig:cij_impact}(a) and (b)).\\
\noindent The simulations further show that influence rates $c^B_{12}$ and $c^B_{21}$ have only marginal effect on the cumulative mortality during the first wave of the pandemic in New York City, with only the most optimistic scenario ($1/c^B_{12}\approx0$) resulting in a maximum of 50\% reduction in the cumulative mortality (see Figure \ref{fig:cij_impact}(c)). While this reduction is sizable, it is unrealistic to assume that most individuals would be capable of changing their attitude (and positively altering their health behaviors) so quickly. Similarly, as seen in the case of the peak daily hospitalizations metric, the size of the SARS-CoV-2 cumulative mortality in New York City is more affected by changes in the peer influence parameters, $c^B_{12}$ and $c^B_{21}$, during the second wave (compare Figures \ref{fig:cij_impact}(c) and (d)). Lastly, similar to what was observed in the Section \ref{sec:ai_impact}, Figure \ref{fig:cij_impact}(d) reveals a region in the $c^B_{12}-c^B_{21}$ plane wherein an increase in $c^B_{12}$ leads to an increase in cumulative mortality. This, once again, is likely due to an increase in the length (in days) of the second wave under these circumstances.
\begin{figure}
    \centering
    \includegraphics[width=1\linewidth]{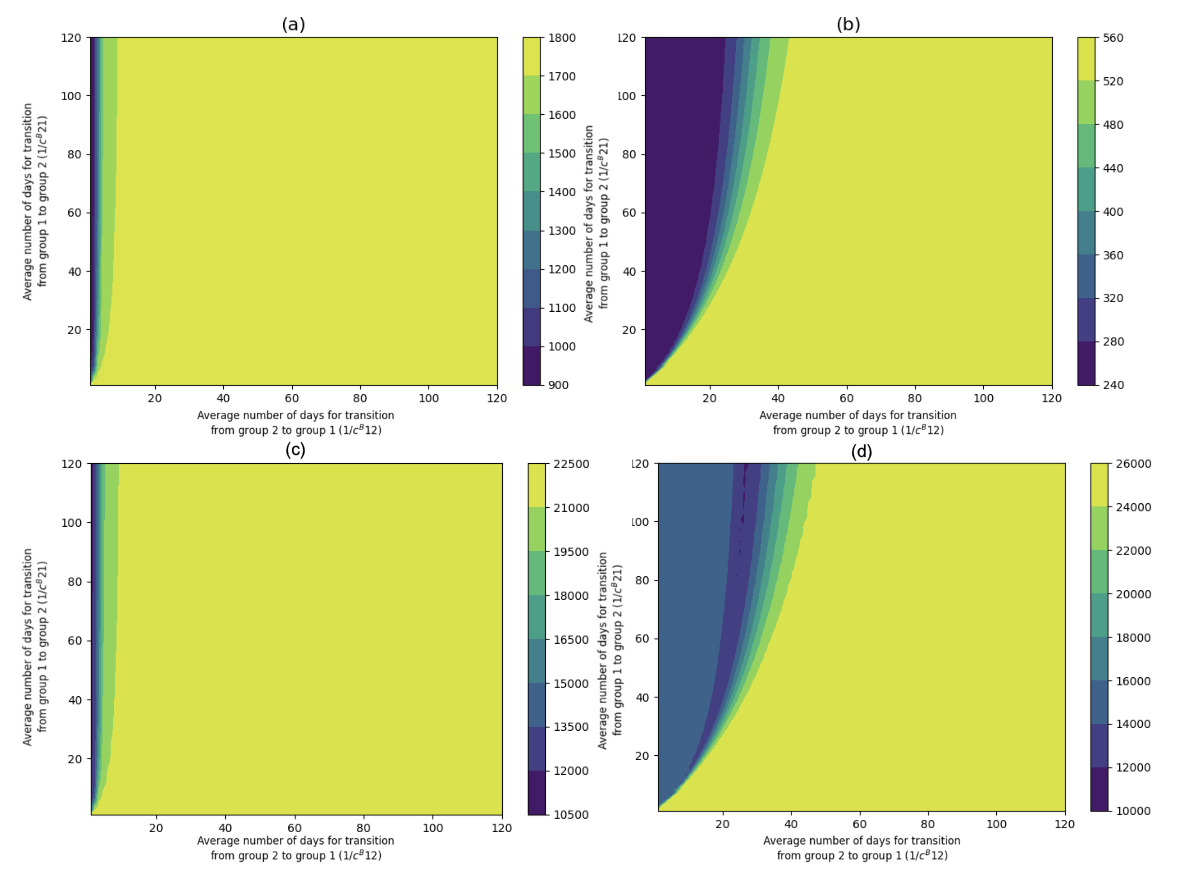}
    \caption{Effect of the peer influence-induced behavior change parameters, $c^B_{12}$ and $c^B_{21}$, on the peak daily hospitalizations and cumulative mortality in New York City during the first two waves of the SARS-CoV-2 pandemic.  Simulations of the 2-group behavior model \eqref{eq:n2model} with the baseline parameter values in Table \ref{tab:fixedparams} and \ref{tab:fittedparams_full}, and various values of $c^B_{12}$ and $c^B_{21}$.
    (a) Effect of $c^B_{12}$ and $c^B_{21}$ on the peak daily hospitalizations during the first wave.  (b) Effect of $c^B_{12}$ and $c^B_{21}$ on the peak daily hospitalizations during the second wave. (c) Effect of $c^B_{12}$ and $c^B_{21}$ on the cumulative mortality during the first wave. (d) Effect of $c^B_{12}$ and $c^B_{21}$ on the cumulative mortality during the second wave.}
    \label{fig:cij_impact}
\end{figure}
\subsubsection{Impact of relative behavioral group size at onset of the pandemic }\label{sec:impact_of_ics}
\noindent  The calibration and simulations of the 2-group behavior model \eqref{eq:n2model} (carried out in Sections \ref{sec:fitting}-\ref{sec:cij_impact}) were for the case where all individuals in the community started out in the risk-tolerant Group 2 before the onset of the pandemic. In this section, the model \eqref{eq:n2model} will be simulated for scenarios where this assumption is relaxed. This (relaxation of the assumption) allows for the assessment of the (potentially more realistic) situation where certain members of the community are predisposed to be cautious during the very beginning of public health emergencies (this seemed to have been the case in New York City during the SARS-CoV-2 pandemic \cite{maskhoarding}). Specifically, the case where a proportion, $k\in[0,1]$, of susceptible members of the community start out in Group 1 (i.e., $S_1(0)\approx kN(0)$, where $N(0)$ is the total initial population size) and the remaining susceptible proportion start out in Group 2 (i.e., $S_2(0)\approx (1-k)N(0)$) is now simulated for various values of $0\leq k\leq1$ (with the same parameter values in Tables \ref{tab:fixedparams} and \ref{tab:fittedparams_full}). The simulation results obtained show a marked decrease in daily hospitalizations ($I_h(t)$) with increasing proportion of members of the community ($k$) who started out in Group 1 (this reduction is particularly pronounced during the respective daily hospitalization peaks of the first and second waves). For example, a 23\% reduction of peak daily hospitalizations in the first wave can be achieved if 20\% of the susceptible members of the community started out in Group 1, in relation to the baseline scenario where every susceptible member of the community started out in Group 2 (compare red and orange curves in Figure \ref{fig:ic_impact}(a), and see also Table \ref{tab:ic_impact_bywave}). Up to 55\% reduction in the peak daily hospitalizations in the first wave can be achieved if everyone started out in Group 1 (see blue curve in Figure \ref{fig:ic_impact} and Table \ref{tab:ic_impact_bywave}). Although an increase in $k$ reduces the peak daily hospitalizations in the second wave (with an increase in $k$ from 0 to 0.2 leading to a 33\% reduction in the peak daily hospitalizations), this reduction diminishes for $k\geq 0.2$, with $k=0.2$ and $k=1$ producing nearly identical values for the peak daily hospitalizations during the second wave (see purple shaded region in Figure \ref{fig:ic_impact}(a)). This is likely the result of similar population makeups, meaning that as long as enough individuals started out cautious at the beginning of the first wave (i.e., $k>=0.2$), then by the start of the second wave the population would be mostly made up of cautious individuals ($\frac{N_1(t)}{N(t)}\approx 1$ for $t>180$, see Figure \ref{fig:n1_k} in Appendix C). Similar results were obtained with respect to the cumulative SARS-CoV-2 mortality metric, as depicted in Figure \ref{fig:ic_impact}(b). This suggests that initial attitudes towards the pandemic (specifically, in the few days after the disease arrives in the community) and the government agencies responsible for public health messaging and NPI implementation have a large impact on disease dynamics. Hence the early implementation of effective public health messaging that encourages a sizable proportion of the susceptible population to adhere to NPIs will significantly reduce the burden of the pandemic. \vspace{-5mm}
\begin{figure}[H]
    \centering
\hspace{3em}\textbf{(a)}\hspace{18em}\textbf{(b)}\\ \includegraphics[width=160mm,scale=1]{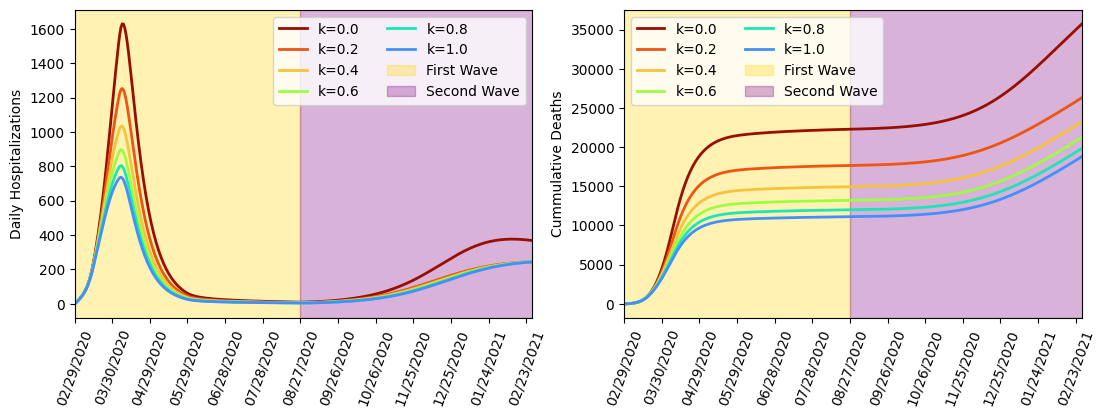}\\
    \caption{Effect of initial sizes of the two behavioral groups on daily hospitalizations ($I_h(t)$) and cumulative mortality during the first two waves of the SARS-CoV-2 pandemic in New York City. Simulations of the 2-group behavior model \eqref{eq:n2model} using the parameter values in Tables \ref{tab:fixedparams} and \ref{tab:fittedparams_full} and various initial sizes of the behavioral groups ($S_1(0)=kN(0)$ and $S_2(0)=(1-k)N(0)$, for $k=0,0.2,0.4,...,1$).  (a) Daily SARS-CoV-2 hospitalizations, as a function of time during the first two waves ($I_h(t)$). (b) Cumulative SARS-CoV-2 mortality, as a function of time during the first two waves in New York City.
    } 
    \label{fig:ic_impact}
\end{figure}\vspace{-5mm}
\begin{center}
\begin{table}[H]
    \centering
\begin{tabular}{|p{2.15cm}||p{2.73cm}|p{2.5cm}||p{2.73cm}|p{2.5cm}|}
     \hline
     Proportion & \multicolumn{2}{c||}{First Wave} & \multicolumn{2}{c|}{Second Wave} \\ \cline{2-5}
     initially in Group 1 ($k$) & Peak Daily Hospitalizations & Cumulative Mortality & Peak Daily Hospitalizations & Cumulative Mortality \\
     \hline
     \hline
     0 & 1,630 & 17,914 & 376 & 25,993 \\ \hline
     0.2 & 1,253 & 13,959 & 250 & 15,720 \\ \hline
     0.4 & 1,034 & 11,750 & 250 & 15,180 \\ \hline
     0.6 & 898 & 10,465 & 251 & 14,670 \\ \hline
     0.8 & 804 & 9,565 & 251 & 14,265 \\ \hline
     1 & 736 & 8,756 & 251 & 14,135 \\ \hline
    \end{tabular}
    \vspace{0.5cm} \caption{SARS-CoV-2 related peak daily  hospitalizations and cumulative mortality in New York City at the end of the first wave of the pandemic and after the second wave of the pandemic, generated by simulating the 2-group behavior model \eqref{eq:n2model} using the values of the fixed and estimated parameters given in Tables \ref{tab:fixedparams} and \ref{tab:fittedparams_full}, respectively. Initial conditions used in the simulations are given by $S_1(0)\approx kN(0)$ and $S_2(0) \approx (1-k)N(0)$ for various values of $k \in [0,1]$.}
       \label{tab:ic_impact_bywave}
\end{table}
\end{center}
\subsubsection{Impact of early implementation and efficacy of lockdown and other NPIs }\label{sec:impact_of_lockdown}
\begin{figure}[H]
    \centering
    \includegraphics[width=0.55\linewidth]{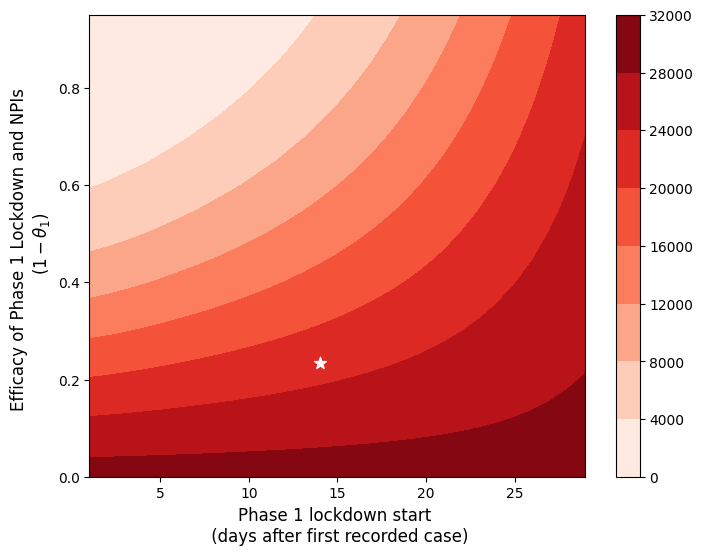}
    \caption{Heat map for assessing the effect of early implementation and efficacy of lockdown and other NPI measures $(1-\theta_{l,1})$ on the cumulative mortality during the first wave of the SARS-CoV-2 pandemic in New York City. Simulations of the 2-group behavior model \eqref{eq:n2model} with the baseline values of the fixed and estimated parameters of the model given in Tables \ref{tab:fixedparams} and \ref{tab:fittedparams_full}, respectively, and various values of $\theta_{l,1}$ and start dates for the implementation of the Phase 1 lockdown and other NPI measures after the index case of the SARS-CoV-2 pandemic in New York City. The white star marks the coordinates of the fitted value of $\theta_{l,1}=0.74664$ (estimated in Section \ref{sec:paramsfitted}) and the reported start date of lockdown in New York City ($=14$ days \cite{314_1}).}
    \label{fig:p1vt1}
\end{figure} \vspace*{-5mm}
\noindent 
In the previous sections discussed thus far, simulations of the model were carried out under the assumptions that Phase 1 of lockdown began in New York City 14 days after the index case of the SARS-CoV-2 pandemic was reported in the city, and that the efficacy of lockdown and other NPIs during Phase 1 was somewhat low (approximately 25\%, as estimated in Section \ref{sec:paramsfitted}). It is instructive, therefore, to assess the population-level impact of varying the initial start date (beginning of Phase 1) and efficacy of the lockdown and other NPI measures (as measured by $1-\theta_{l,1}$) on the cumulative SARS-CoV-2 mortality at the end of the first wave in New York City. To achieve this objective, the 2-group behavior model \eqref{eq:n2model} was simulated using the parameter values in Tables \ref{tab:fixedparams} and \ref{tab:fittedparams_full} and varying values of $1-\theta_{l,1}$ and the Phase 1 lockdown start date. The results obtained, depicted by the heat map in Figure \ref{fig:p1vt1}, show a dramatic reduction in the cumulative SARS-CoV-2 mortality during the first wave with increasing efficacy of the intervention measures ($1-\theta_{l,1}$) even if the start date of the implementation of the intervention measures was up to the 14-day baseline.  For instance, a 36\% reduction (from the baseline of 22,000 to 14,000) in the cumulative SARS-CoV-2 mortality can be recorded during the first wave if lockdown measures were implemented at the 14-day baseline period after the index case and the efficacy of this measure is 50\% (i.e., $1-\theta_{l,1}$=$0.5$).
Furthermore, this heat map shows that if the lockdown and other NPI interventions are implemented earlier (e.g., five days after the index case), the SARS-CoV-2 pandemic would have been greatly suppressed in New York City if the efficacy of these interventions is at least 60\% (see the top- left corner of the heat map). This figure further shows a sizable increase in the first wave cumulative mortality with increasing delay in the implementation of the intervention measures from the index case. If such delay is large enough, even high efficacy of these interventions only lead to a marginal reduction in cumulative mortality. For example, a 3-week delay in the implementation of these measures (after the index case) and 50\% efficacy will result in first wave cumulative mortality of about 20,000 (which represents only about 9\% reduction from the baseline value of 22,000). Furthermore, if the implementation is delayed by a month from the index case, even with the efficacy of these interventions as high as 90\%, up to 24,000 SARS-CoV-2 deaths will be recorded by the end of the first wave in New York City (this represents a 10\% increase in the baseline cumulative mortality recorded). Overall, this figure emphasizes the importance of early implementation of lockdown and other NPI measures, at moderate to high efficacy levels, in successfully curtailing the burden of the SARS-CoV-2 pandemic in New York City during the first wave.  In fact, no significant disease-induced mortality would have been recorded during the first wave of the SARS-CoV-2 pandemic if the lockdown and other NPI measures were implemented early (such as less than a week after the index case) and their efficacy is moderate or high enough (such as 60\%)). 

\section{Discussion and conclusion}\label{discussion}
\noindent One of the key lessons learned during the 2019 SARS-CoV-2 pandemic (also known as COVID-19, which accounted for over 700 million cases and seven million deaths worldwide \cite{WHOdashboard}) was the impact of heterogeneity in changes in human behavior or attitude towards the disease.  Some of the notable heterogeneities in human behavior observed during the course of the pandemic include differences in attitude (i.e., acceptance, hesitancy, or refusal) towards the public health intervention and mitigation measures implemented in their community or jurisdiction (e.g., social distancing, quarantine, isolation, wearing a face mask in public, testing, and vaccination), acceptance and/or refusal of public health messaging and/or mis(dis)information and heterogeneities based on age, gender, race and political affiliation (or polarization) \cite{altruism_bir,nycfatigue,gg_sociopolitcal_stoler,gg_policy_lim}. This study is based on using mathematical modeling approaches, backed by data analytics and computation, to realistically quantify the impact of the aforementioned human behavioral changes and heterogeneities on the spread and control of the SARS-CoV-2 pandemic in a population during the early stages of the pandemic.  In particular, the study assessed the impacts of these heterogeneities during the first two waves of the pandemic in New York City (a one-time global epicenter of the pandemic, and the jurisdiction that suffered the most SARS-CoV-2 burden in the United States).
\vspace*{1mm}\ \\ \indent 
The objective of this study was achieved {\it via} the design, analysis, parameterization/validation, and simulations of a novel mechanistic mathematical model which explicitly accounts for changes in human behavior during the pandemic by stratifying the total population into $n$ \textit{behavioral groups} based on their perception of risk (of acquisition/transmission of SARS-CoV-2 infection) and their resulting behavior choices with respect to the risk. Specifically, the first group (referred to as Group 1) represents individuals in the community who consistently adopt the most risk-averse behaviors against the acquisition or transmission, while Groups 2 through $n$ adopt an increasing risk-taking behaviors as $n$ increases (i.e., Group $n$ is the most risk-tolerant behavior group in the community).  The resulting behavior-epidemiology epidemic model, which takes the form of an $n-$group deterministic system of nonlinear differential equations, was rigorously analyzed for its basic qualitative properties (such as the invariance and non-negativity of its solutions).
\vspace*{1mm}\ \\ \indent 
A special case of the $n-$group model with two behavioral groups (i.e., $n=2$), with  one group for the most risk-averse and the other for the most risk-tolerant members of the community, was considered for mathematical tractability. The resulting 2-group behavior model was shown to have  three non-trivial disease-free equilibria,  namely, a \textit{cautious-only} (risk-averse) disease-free equilibrium (denoted by \textit{G1DFE}), a \textit{cautious-free} (risk-tolerant) disease-free equilibrium (denoted by \textit{G2DFE}), and a \textit{coexistence} (where both risk-averse and risk-tolerant groups co-exist) disease-free equilibrium (denoted by \textit{G3DFE}). These equilibria are shown to be stable whenever a certain epidemiological threshold, known as the \textit{control reproduction number} (denoted by $\RR_c$), is less than unity and the influence ratio (denoted by $\Gamma$) is greater, less than, or equal to 1, respectively.  The epidemiological implication of this result is that, regardless of the value of the influence parameter $\Gamma$, no major outbreaks of the SARS-CoV-2 pandemic can occur in the community if the intervention and mitigation strategies implemented in the community are able to bring, and maintain, the control reproduction number to a value below unity. From a public health standpoint, the cautious-only equilibrium (G1DFE) is the most desirable as it is associated with lower disease burden (as measured in terms of peak daily hospitalizations and cumulative mortality; see Table \ref{tab:ic_impact_bywave} and Figure \ref{fig:ic_impact}). It was further shown that, for the desirable G1DFE to be stable, it is necessary that $\Gamma=c^B_{12}/c^B_{21}$ must exceed unity (i.e., the risk-averse population must exert much higher level of influence on the non-cautious population than the other way round). While other studies have shown complementary results (for instance, Bir et al. \cite{altruism_bir} shows that social influence in the form of the promotion of altruism is correlated with a higher likelihood of cautious behavior adoption, and Espinoza et al. \cite{gg_heterogeneous_espinoza} shows that the behavioral makeup of the population has a large impact on final epidemic size), the current study is among the first to highlight and rigorously show the major role social (peer) influence plays on the dynamics of the SARS-CoV-2 pandemic, as measured in terms of helping to make the epidemiologically-desirable disease-free equilibrium (G1DFE) to be stable. \vspace*{1mm}\\
\indent The 2-group behavior model was calibrated using a 7-day running average of daily hospitalizations data  \cite{nycgov_data} for the first wave of the SARS-CoV-2 pandemic in New York City (this corresponds to the time period from February 29th, 2020 to July 28th, 2020 \cite{CDC_NYCstats}). The same hospitalizations dataset, for the period from July 29th, 2020 to Mar/ch 25th, 2021, was used to cross-validate the calibrated 2-group behavior model, and the resulting fitted and cross-validated 2-group behavior model was used to estimate the eight unknown parameters of this model (and their associated 95\% confidence intervals were also given) and make predictions for the second wave of the SARS-CoV-2 pandemic in New York City. It was shown that, while both the 2-group behavior model and its behavior-free equivalent (obtained by setting the behavior-related parameters of the 2-group behavior model to zero) fitted the first wave of the SARS-CoV-2 pandemic in New York City reasonably well, the behavior-free model failed to accurately predict the second wave, unlike the 2-group behavior model which did so almost perfectly (see Figures \ref{fig:predict_full} and \ref{fig:predict_bf}). This suggests that epidemic models of the SARS-CoV-2 pandemic that do not explicitly account for heterogeneous human behavior may fail to accurately predict the trajectory and burden of the pandemic in a population. This result complements the findings in \cite{ge_behavior_binod, gg_utility_marathe}.\vspace*{1mm} \\
\indent Detailed global sensitivity analysis, in the form of Partial Rank Correlation Coefficients (PRCCs) using Latin Hypercube Sampling \cite{LHS_blower}, was carried out for the 2-group behavior model to determine which of its 18 parameters had the largest impact on the two chosen response functions (namely, the peak daily SARS-CoV-2 hospitalizations and the cumulative mortality, corresponding to both the first and second waves of the SARS-CoV-2 pandemic in New York City). This analysis, carried out for two snapshots of time, namely April 20, 2020 (which corresponds to the middle of the first wave of the SARS-CoV-2 pandemic in New York City) and February 21, 2021 (which corresponds to the middle of the second wave of the pandemic in New York City), identified five main parameters of the 2-group behavior model that have the highest impact on the two response functions during the first two waves of the SARS-CoV-2 pandemic.  The identified parameters were, the maximum effective contact rates for asymptomatically and symptomatically infectious individuals ($\beta_a$ and $\beta_i$, respectively), the rate of progression of individuals leaving the exposed class ($\sigma_e$), and the parameter for the efficacy of the lockdown and other non-pharmaceutical interventions (NPIs) implemented during phase 1 of the implementation period ($\theta_{l,1}$). The sensitivity analysis results suggest that the lockdown and NPI measures implemented would have significantly reduced the burden of the pandemic during the first two waves of the pandemic in New York City (and, consequently, prevent the healthcare system from being overwhelmed during this period) if their efficacy and coverage was high enough (i.e., if the quantity, $1-\theta_{l,1}$, is high enough). This could be achieved by adequate preparation (such as stockpiling NPIs and making them widely available to the public) and early implementation of effective public health messaging campaigns encouraging adherence to intervention and mitigation measures against the pandemic. 
\\
\indent The calibrated and validated 2-group behavior model was then simulated to quantify the impact of human behavioral changes and peer influence on the dynamics of the SARS-CoV-2 pandemic in New York City during the first two waves. In particular, two sets of numerical simulations were carried out to assess the impact of the behavioral parameters (namely, the hospitalization-motivated behavioral modification parameters, $a_1$ and  $a_2$, and the influence-motivated behavioral modification parameters, $c^B_{12}$ and $c^B_{21}$) on the burden of the pandemic in New York City during the first two waves. The first set of simulations was conducted by varying the hospitalization-motivated behavioral change parameters, $a_1$ and $a_2$, while fixing all other parameters of the behavior model with their baseline values (given in Tables \ref{tab:fixedparams} and \ref{tab:fittedparams_full}). The results of these simulations showed that the first wave of the SARS-CoV-2 pandemic in New York City (as measured by the size of the peak daily hospitalizations and the cumulative SARS-CoV-2 mortality) was largely determined by the disease-motivated behavioral changes of individuals in Group 2 (the risk-tolerant group), while the impact of behavioral changes by risk-averse individuals (i.e., member of Group 1) was marginal during this time (i.e., the metrics for disease burden were more strongly correlated with changes in the parameter $a_2$, than with changes in the parameter $a_1$; see Figure \ref{fig:ai_impact}a). Hence, these simulations suggest that the first wave of the SARS-CoV-2 pandemic in New York City was largely driven by the behavior of the risk-tolerant group (as measured by changes in the parameter $a_2$), and much less so by the behavior of the risk-averse group; this may be due to the assumption that every member of the community started out in the risk-tolerant Group 1. Qualitatively similar results were obtained for the second wave, although the role of $a_1$ (i.e., its negative correlation with the disease burden) is more pronounced than it was during the first wave (see Figure \ref{fig:ai_impact}b). The simulations for the second wave also showed a threshold value of the parameter $a_2$, above which the cumulative SARS-CoV-2 mortality during this wave is minimized. In other words, this shows that if risk-tolerant individuals (i.e., members of Group 2) were to significantly decrease their contact rate in response to the level of hospitalizations in the community (i.e., if $a_2>2,200$, which roughly translates to a 40-80\% reduction of contacts at peak daily hospitalization times, see the orange curve in Figure \ref{fig:cia}), the second wave of the SARS-CoV-2 pandemic in New York City would have been largely reduced or averted entirely. \\
\indent Similar simulation results were obtained with respect to the disease burden by now varying the influence parameters, $c^B_{12}$ and $c^B_{21}$, while keeping the hospitalization-induced behavior change parameters $a_1$ and $a_2$ (and all other parameters) fixed at their baseline values (given in Tables \ref{tab:fixedparams} and \ref{tab:fittedparams_full}). In other words, this study showed that the size of the first wave of the SARS-CoV-2 pandemic in New York City (measured in terms of the the peak daily hospitalizations and the cumulative SARS-CoV-2 mortality) was largely determined by the behavior of members of the risk-tolerant group (Group 2).  In particular, the value of peak daily SARS-CoV-2 hospitalization is strongly affected by the average time taken by individuals in Group 2 to positively change their behavior (and move to the risk-averse Group 1) in response to peer pressure from members of Group 1 ($1/c^B_{12}$). For example, a 10-day reduction from the baseline value of this parameter (i.e.,from $1/c^B_{12}=30$ days to $20$ days) resulted in a 27\% reduction in peak daily hospitalization. Based on the data for the SARS-CoV-2 pandemic during the first wave in New York City \cite{nycgov_data}, this will translate to a reduction of over 400 daily SARS-CoV-2-related hospitalizations (thereby significantly lessening the burden on the City's healthcare infrastructure). This suggests that a timely implementation of an effective public health strategy that emphasizes positive behavior changes due to social influences (i.e., decrease $1/c^B_{12}$, which can be achieved through, for example, positive messaging and advocacy from major community influencers, such as celebrities, community and religious leaders, and public health officials) will significantly decrease disease burden, as well as reduce strain on the healthcare system \cite{religious,altruism_bir,japan_parady}. \\
\begin{figure}
    \centering
    \includegraphics[width=0.8\linewidth]{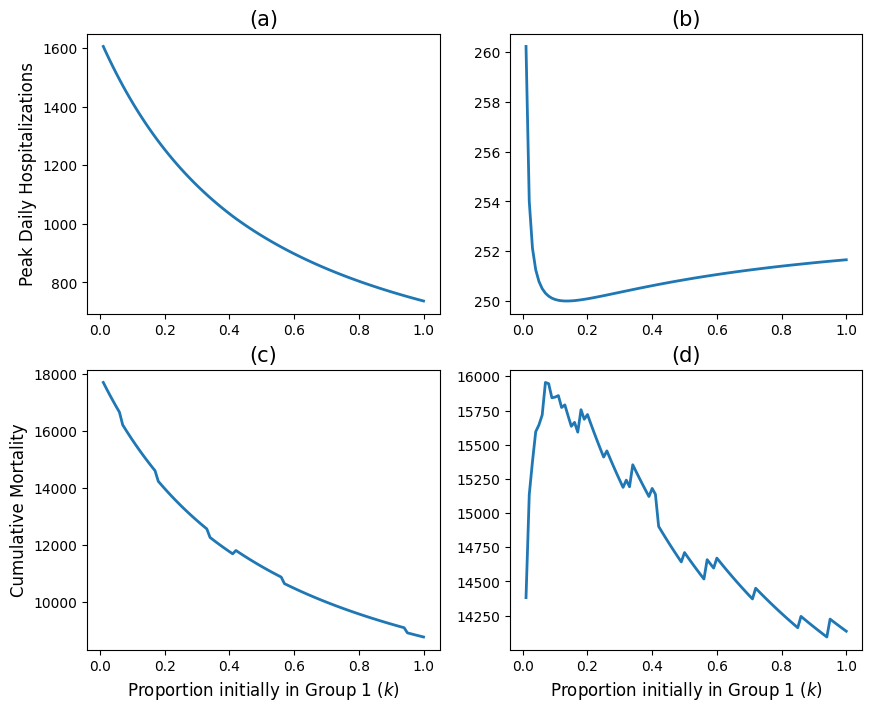}
    \caption{Effect of initial sizes of the two behavioral groups on peak daily hospitalizations and cumulative mortality during the first two waves of the SARS-CoV-2 pandemic in New York City. Simulations of the behavior model \eqref{eq:n2model} using the parameter values in Tables \ref{tab:fixedparams} and \ref{tab:fittedparams_full} and various initial sizes of the behavioral groups ($S_1(0)=kN(0)$ and $S_2(0)=(1-k)N(0)$, for $k=0,0.2,0.4,...,1$).  (a) First wave peak daily hospitalizations, as a function of initial Group 1 size ($k$). (b) Second wave peak daily hospitalizations, as a function of initial Group 1 size ($k$). (c) First wave cumulative mortality, as a function of initial Group 1 size ($k$). (d) Second wave cumulative mortality, as a function of initial Group 1 size ($k$).}
    \label{fig:ic_impact_bywave}
\end{figure}
\indent Finally, simulations to assess the impact of relative behavioral group size at pandemic outset were carried out to determine to what extent the size of the disease burden was determined by the initial attitudes of members of the community (i.e., by members of groups 1 and 2) with respect to the disease risk. This was achieved by assuming a proportion, $k\in[0,1]$, of the initial number of susceptible individuals in the community started out as members of Group 1 (i.e., they initially started in the risk-averse group), while the remaining proportion, $1-k$ , started out as members of Group 2 (i.e., they initially started in the risk-tolerant group). The simulations results obtained for this scenario (depicted in Figure \ref{fig:ic_impact}) showed a decrease in the peak daily hospitalizations and cumulative mortality with increasing values of $k$. Even a moderate increase in the baseline value of $k$ (e.g., an increase from  $k=0$ to $k=0.2$) reduced the baseline value of the peak daily hospitalizations by approximately 23\% during the first wave, and by 33\% during the second wave of the pandemic. This also decreased the cumulative mortality by approximately 22\% and 39\% by the end of the first and second waves, respectively (see Table \ref{tab:ic_impact_bywave}).However, by separating these simulation results by pandemic waves (as shown in Figure \ref{fig:ic_impact_bywave}), it can be seen that this simple negative correlation between the value of $k$ and the two disease metrics holds only for the first wave of the pandemic (see Figures \ref{fig:ic_impact_bywave}(a) and (c)), and not during the second. Specifically, for the second wave, this study showed that peak daily hospitalizations remained relatively constant (between 250 and 260 SARS-CoV-2 hospitalizations at the peak) regardless of the value of $k$ (see Figure \ref{fig:ic_impact_bywave}(b)). For cumulative mortality of the second wave of the SARS-CoV-2 pandemic in New York City, simulations suggest that the impact of increasing $k$ changes depending on the size of $k$. For small values of $k$, (e.g. $0<k<0.1$) an increase in $k$ within this interval is associated with a large \textit{increase} in second wave cumulative mortality (see Figure \ref{fig:ic_impact_bywave}(d)). This is likely due to the corresponding increase in the length of the second wave as $k$ increases in the interval $(0,0.1)$ (see Figure \ref{fig:ic_impact_w2l} in Appendix C). However, for values of $k$ larger than $0.1$, an increase in $k$ leads to a decrease in second wave cumulative mortality (see Figure \ref{fig:ic_impact_bywave}(d)). This suggests that having a significant proportion of members of the community starting out in the risk-averse group (e.g., $k>0.2$) at the beginning of the SARS-CoV-2 pandemic will lead to a much smaller pandemic overall in New York City. These results, combined with the earlier simulation results, support the conclusion that a risk-averse attitude at the early stages of the pandemic, coupled with early implementation of effective public health intervention and mitigation measures, are vital for effectively combating, suppressing or eliminating the burden of respiratory pandemics, such as the SARS-CoV-2 pandemic. \\
\indent Although the $n$-group behavior model considered in this study incorporated many key pertinent features of the SARS-CoV-2 pandemic, it is not without some limitations. For instance, since peer influence changes over time in reality, the model may be improved by relaxing the assumption that the peer influence parameters ($c^B_{ij}$, with $i,j=1,2,...,n$ and $i\neq j$) are constant (i.e., re-defining the rates $c^B_{ij}$ to be functions of time, thereby making the model to be {\it non-autonomous} \cite{nonauto_1,nonauto_2,gg_primer_gumel}). 
The model may further be improved by explicitly incorporating related heterogeneities, such as age, sex, race, political affiliation, and socio-economic status. The authors hope to consider these limitations in a future study \cite{gg_optimal_pisaneschi, gg_saadroy, gg_adaptivecontact_arthur}. Finally, the 2-group behavior model considered in this study was calibrated using disease data (7-day running average SARS-CoV-2 hospitalizations during the first two waves of the pandemic in New York City) to estimate its unknown epidemiological and behavioral parameters (although the values of the behavioral change parameters $a_1$ and $a_2$ were verified using mask compliance data from the New York City subway system (available starting mid June 2020, \cite{behaviordatanyc}), see Figure \ref{fig:noncompliance}). However, it would have been ideal to parameterize the behavior-related parameters of the model with human behavior data, which, unfortunately, was not available (for New York City during the first two waves of the pandemic) until several months after the advent of the pandemic. This study, therefore, strongly highlights the need for the collection of such behavior-related data (for risk-taking and adherence to interventions) during major respiratory pandemics, such as SARS-CoV-2, during the early stages, and making such data publicly-available for modelers to use. Such crucially-important data can be collected using surveys and digital tools (such as computer vision to determine how many people are wearing masks \cite{behaviordatanyc}). Behavioral data can also be collected through mobile phone data \cite{behaviordatacollection1} or well-done contact surveys \cite{behaviordatacollection2}.
Overall, this study shows that the prospect of the effective control of the SARS-CoV-2 pandemic in a large population stratified by human behavior changes using NPIs (if implemented in an effective and timely manner) is highly promising. 

\section*{Acknowledgments}
\noindent ABG acknowledges the support, in part, of the National Science Foundation (Grant Number: DMS-2052363; transferred to DMS-2330801). AO acknowledges the support of the University of Maryland School of Graduate Studies Dean's Fellowship. 
\addappendix
\subsection*{A $\enspace$ Equations for Special Cases of the $n-$group Behavior Model \eqref{eq:fullmodel}}\label{sec:appendixA}
\subsection{Equations for the special case with $n=2$}
 \noindent Setting $n=2$ in the behavior model \eqref{eq:n2model} gives the following 2-group behavior model for SARS-CoV-2 dynamics in New York City:
\begin{align}
    \begin{split}
\Dot{S_1}(t) &= \xi R_1(t)-\left(\theta_l\right)\left[\dfrac{c_1^A(t)S_1(t)}{N(t)}\right]\sum_{j=1}^2c_j^A(t) \left[\beta_aI_{a,j}(t)+\beta_iI_{s,j}(t)+\beta_hI_{h,j}(t)\right] \\
&+ c_{12}^B\dfrac{S_{2}(t)N_1(t)}{N(t)}-c_{21}^B\frac{S_1(t)N_2(t)}{N(t)} \\
\Dot{S_2}(t) &= \xi R_2(t)-\left(\theta_l\right)\left[\dfrac{c_2^A(t)S_2(t)}{N(t)}\right]\sum_{j=1}^2c_j^A(t)  \left[\beta_aI_{a,j}(t)+\beta_iI_{s,j}(t)+\beta_hI_{h,j}(t)\right]  \\
&+c_{21}^B\dfrac{S_1(t)N_2(t)}{N(t)}- c_{12}^B\frac{S_{2}(t)N_1(t)}{N(t)} \\
\Dot{E}_1(t) &=\left(\theta_l\right)\left[\dfrac{c_1^A(t)S_1(t)}{N(t)}\right]\sum_{j=1}^2c_j^A(t)\left[\beta_aI_{a,j}(t)+\beta_iI_{s,j}(t)+\beta_hI_{h,j}(t)\right] -\sigma_eE_1(t)
         \\
&+ c_{12}^B\frac{E_{2}(t)N_1(t)}{N(t)}-c_{21}^B\frac{E_1(t)N_2(t)}{N(t)} \\
\Dot{E}_2(t) &=\left(\theta_l\right)\left[\dfrac{c_2^A(t)S_2(t)}{N(t)}\right]\sum_{j=1}^2c_j^A(t)\left[\beta_aI_{a,j}(t)+\beta_iI_{s,j}(t)+\beta_hI_{h,j}(t)\right] -\sigma_eE_2(t) \\
&        +c_{21}^B\dfrac{E_1(t)N_2(t)}{N(t)}- c_{12}^B\dfrac{E_{2}(t)N_1(t)}{N(t)} \\
\Dot{I}_{a,1}(t) &= (1-r)\sigma_eE_1(t)-\gamma_aI_{a,1}(t) + c_{12}^B\frac{{I}_{a,2}(t)N_1(t)}{N(t)}-c_{21}^B\dfrac{{I}_{a,1}(t)N_2(t)}{N(t)} \\
\Dot{I}_{a,2}(t) &= (1-r)\sigma_eE_2(t)-\gamma_aI_{a,2}(t) +c_{21}^B\dfrac{{I}_{a,1}(t)N_2(t)}{N(t)}- c_{12}^B\dfrac{{I}_{a,2}(t)N_1(t)}{N(t)} \\
\Dot{I}_{s,1}(t) &= r\sigma_eE_1(t)-\sigma_iI_{s,1}(t)   + c_{12}^B\dfrac{{I}_{s,2}(t)N_1(t)}{N(t)}-c_{21}^B\dfrac{{I}_{s,1}(t)N_2(t)}{N(t)} \\
\Dot{I}_{s,2}(t) &= r\sigma_eE_2(t)-\sigma_iI_{s,2}(t)   + c_{21}^B\dfrac{{I}_{s,1}(t)N_2(t)}{N(t)}-c_{12}^B\dfrac{{I}_{s,2}(t)N_1(t)}{N(t)} \\
\Dot{I}_{h,1}(t) &= q\sigma_iI_{s,1}(t)-(\gamma_h + \delta_h)I_{h,1}(t)  + c_{12}^B\dfrac{{I}_{h,2}(t)N_1(t)}{N(t)}-c_{21}^B\dfrac{{I}_{h,1}(t)N_2(t)}{N(t)} \\
\Dot{I}_{h,2}(t) &= q\sigma_iI_{s,2}(t)-(\gamma_h + \delta_h)I_{h,2}(t)  + c_{21}^B\dfrac{{I}_{h,1}(t)N_2(t)}{N(t)}-c_{12}^B\dfrac{{I}_{h,2}(t)N_1(t)}{N(t)} \\
\Dot{R_1}(t) &= \gamma_aI_{a,1}(t)+(1-q)\sigma_iI_{s,1}(t)+\gamma_hI_{h,1}(t) -\xi R_1(t) + c_{12}^B\dfrac{R_{2}(t)N_1(t)}{N(t)}-c_{21}^B\dfrac{R_1(t)N_2(t)}{N(t)}\\
\Dot{R_2}(t) &= \gamma_aI_{a,2}(t)+(1-q)\sigma_iI_{s,2}(t)+\gamma_hI_{h,2}(t) -\xi R_2(t) + c_{21}^B\dfrac{R_1(t)N_2(t)}{N(t)}-c_{12}^B\dfrac{R_{2}(t)N_1(t)}{N(t)}
    \end{split}\label{eq:n2model}
\end{align}

\subsection{Equations for the behavior-free version of the 2-group model \eqref{eq:n2model}}
\noindent Define $S(t)=S_1(t)+S_2(t)$, $E(t)=E_1(t)+E_2(t)$, $I_a(t)=I_{a,1}(t)+I_{a,2}(t)$, $I_s(t)=I_{s,1}(t)+I_{s,2}(t)$, $I_h(t)=I_{h,1}(t)+I_{h,2}(t)$, and $R(t)=R_1(t)+R_2(t).$ Substituting these expressions into the 2-group model \eqref{eq:n2model}, and setting the behavior-free parameters ($a_1$, $a_2$, $c^B_{12}$, and $c^B_{21}$) to zero, gives the following behavior-free version of the 2-group model \eqref{eq:n2model}:
\begin{align}
    \begin{split}
\Dot{S}(t) &= \xi R_1(t)-\theta_l\frac{S(t)}{N(t)}\biggl(\beta_aI_{a}(t)+\beta_iI_{s}(t)+\beta_hI_{h}(t)\biggr) \\
\Dot{E}(t) &=\theta_l\frac{S(t)}{N(t)}\biggl(\beta_aI_{a}(t)+\beta_iI_{s}(t)+\beta_hI_{h}(t)\biggr) -\sigma_eE(t)\\
\Dot{I}_{a}(t) &= (1-r)\sigma_eE(t)-\gamma_aI_{a}(t)\\
\Dot{I}_{s}(t) &= r\sigma_eE(t)-\sigma_iI_{s}(t)\\
\Dot{I}_{h}(t) &= q\sigma_iI_{s}(t)-(\gamma_h + \delta_h)I_{h}(t)\\
\Dot{R}(t) &= \gamma_aI_{a}(t)+(1-q)\sigma_iI_{s}(t)+\gamma_hI_{h}(t) -\xi R(t)\\
    \end{split}\label{eq:bfmodel}
\end{align}

\subsection*{B $\enspace$ Proof of Theorem \ref{thm:LASG1DFE}}\label{sec:appendixB}
\stepcounter{section}
\renewcommand{\thesection}{B}
\renewcommand{\thefigure}{\thesection.\arabic{figure}}
\noindent 
\begin{proof}
The proof of Theorem \ref{thm:LASG1DFE} is based on linearizing the 2-group behavior model \eqref{eq:n2model} around the G1DFE and taking advantage of the properties of Metzler matrices as described in \cite{godsend}. It is convenient to reorder the equations of the 2-group model \eqref{eq:n2model} such that the equations corresponding to the non-disease (i.e., susceptible and recovered) compartments ($S_1,S_2,R_1$ and $R_2$) are first, followed by the equations corresponding to the eight infected compartments ($E_1, E_2, I_{a,1}, I_{a,2},\cdots,I_{h,2}$). The Jacobian of the reordered system, evaluated at the G1DFE, is given by the following block upper triangular matrix:
    \begin{align*}
    J\lvert_{G1DFE}=
        \begin{bmatrix}
            A_{11} & A_{12} \\
            \mathbf{0}_{8\times 4} & A_{22}
        \end{bmatrix},
    \end{align*}
\noindent where,
    
\begin{centering}
\begin{align}
        A_{11} =
\begin{bmatrix}
0 & c_{12}^B - c_{21}^B & \xi & -c_{21}^B \\
0 & c_{21}^B - c_{12}^B & 0 & c_{21}^B + \xi \\
0 & 0 & -\xi & c_{12}^B \\
0 & 0 & 0 & -c_{12}^B - \xi
\end{bmatrix},
\end{align}
\end{centering}

\begin{align}
A_{22} = \begin{bmatrix}
-\sigma_e & c^B_{12} & \beta_a \theta_l & \beta_a \theta_l & \beta_i \theta_l & \beta_i \theta_l & \beta_h \theta_l & \beta_h \theta_l \\
0 & -c^B_{12} - \sigma_e & 0 & 0 & 0 & 0 & 0 & 0 \\
(1-r)\sigma_e  & 0 & -\gamma_a & c^B_{12} & 0 & 0 & 0 & 0 \\
0 & (1-r )\sigma_e  & 0 & -c^B_{12} - \gamma_a & 0 & 0 & 0 & 0 \\
r \sigma_e & 0 & 0 & 0 & -\sigma_i & c^B_{12} & 0 & 0 \\
0 & r \sigma_e & 0 & 0 & 0 & -c^B_{12} - \sigma_i & 0 & 0 \\
0 & 0 & 0 & 0 & q \sigma_i & 0 & -\delta_h - \gamma_h & c^B_{12} \\
0 & 0 & 0 & 0 & 0 & q \sigma_i & 0 & -c^B_{12} - \delta_h - \gamma_h \\
\end{bmatrix},
    \end{align}
\noindent and $\mathbf{0_{8\times 4}}$ is an eight by four zero matrix. The matrix $A_{12}$ is the $4\times 8$ matrix corresponding to the derivatives of the non-disease compartments ($S_1,S_2,R_1$ and $R_2$) with respect to the eight disease compartments ($E_1, E_2, I_{a,1}, I_{a,2},\cdots,I_{h,2}$) evaluated at G1DFE. Since the matrix $A_{12}$ is not relevant in the computation of the eigenvalues of $J\lvert_{G1DFE}$, it is not explicitly defined or given here. The eigenvalues of the block upper triangular matrix $J\lvert_{G1DFE}$ are the eigenvalues of its diagonal blocks ($A_{11}$ and $A_{22}$).  It can be seen that the eigenvalues of the matrix $A_{11}$ are: 
\[\sigma(A_{11})=\{0,c^B_{21}-c^B_{12},-\xi,-c^B_{12}-\xi\},\]
\noindent from which it follows that three or the four eigenvalues of the matrix $A_{11}$ are negative (with one eigenvalue being identically equal to zero) provided that the following inequality (termed {\it Condition 1}) is satisfied:
\begin{align}\label{cond:TC1_a}
    \text{Condition 1: }& c^B_{12}>c^B_{21}.
\end{align}
\noindent 
Unlike in the case of the matrix $A_{11}$, the eigenvalues of the matrix $A_{22}$ cannot be readily computed directly. The eigenvalues of the matrix $A_{22}$ will be obtained using the properties of Metzler-stable matrices (recall that a Metzler-stable matrix is a Metzler matrix with all its eigenvalues having negative real part \cite{godsend}). This will be achieved by using Proposition
 3.1 of \cite{godsend}, which is given below:
\begin{proposition}\textnormal{(Proposition 3.1 of \cite{godsend})}
    Let $M$ be a Metzler matrix, which is block decomposed:
    \[M=\begin{bmatrix}
        A & B \\
        C & D
    \end{bmatrix}\]
    \noindent where $A$ and $D$ are square matrices. Then $M$ is Metzler stables if and only if $A$ and $D-CA^{-1}B$ are Metzler stable.
\end{proposition}
\noindent To apply this method, $A_{22}$ is broken down into smaller (and more tractable) matrices.  Specifically, this is done by re-writing $A_{22}$ as:
\begin{align*}
    A_{22} = \begin{bmatrix}
        P & Q \\
        R  & S 
    \end{bmatrix}
\end{align*}
where, 
\begin{align}
    P=\begin{bmatrix}
-\sigma_e & c_{12}^B & \beta_a \theta_l & \beta_a \theta_l \\
0 & -c_{12}^B - \sigma_e & 0 & 0 \\
\sigma_e (1-r) & 0 & -\gamma_a & c_{12}^B \\
0 & \sigma_e (1-r) & 0 & -c_{12}^B - \gamma_a
\end{bmatrix}, 
Q=\begin{bmatrix}
\beta_i \theta_l & \beta_i \theta_l & \beta_h \theta_l & \beta_h \theta_l \\
0 & 0 & 0 & 0 \\
0 & 0 & 0 & 0 \\
0 & 0 & 0 & 0
\end{bmatrix}\\
R=\begin{bmatrix}
r\sigma_e & 0 & 0 & 0 \\
0 & r\sigma_e & 0 & 0 \\
0 & 0 & 0 & 0 \\
0 & 0 & 0 & 0 
\end{bmatrix}, S=\begin{bmatrix}
-\sigma_i & c^B_{12} & 0 & 0 \\
0 & -c^B_{12}-\sigma_i & 0 & 0 \\
q\sigma_i & 0 & -\delta_h-\gamma_h & c^B_{12} \\
0 & q\sigma_i  & 0 & -c^B_{12}-\delta_h-\gamma_h
\end{bmatrix}.
\end{align}
It should be recalled from Proposition 3.1 of \cite{godsend} that the matrix $A_{22}$ is Metzler-stable if and only if the matrices $P$ and $S-RP^{-1}Q$ are Metzler-stable. As a principal sub-matrix of a Metzler matrix, $P$ is clearly a Metzler matrix \cite{godsend}. To confirm the Metzler-stability of $P$, it is necessary to compute its eigenvalues, which are given by:
\[\sigma(P)=\{-c^B_{12}-\gamma_a, -c^B_{12}-\sigma_e, -\frac{\gamma_a+\sigma_e}{2}\pm \lambda_{P}\},\]
\noindent where, 
\[\lambda_{P} = \frac{1}{2}\sqrt{(\gamma_a-\sigma_e)^2+4(1-r)\beta_a\sigma_e\theta_l},\] 
from which it follows that all eigenvalues of $P$ are strictly negative (making $P$ Metzler-stable) if and only if the following condition (termed {\it Condition 2)} is satisfied:
\begin{align}\label{cond:TC2_a}
    \text{Condition 2: }& \frac{(1-r)\beta_a\theta_l}{\gamma_a} < 1.
\end{align}
Hence, if Condition 2 holds, the matrix $P$ is Metzler-stable. The next step is to compute the matrix $S-RP^{-1}Q$, as done below:
\begin{align}\label{eq:SRPQ}
    S-RP^{-1}Q = \begin{bsmallmatrix}
\frac{\beta_i \gamma_a r \theta_l}{\gamma_a -(1-r) \beta_a \theta_l} - \sigma_i & c_{12}^B + \frac{\beta_i \gamma_a r \theta_l}{\gamma_a -(1-r) \beta_a \theta_l} & \frac{\beta_h \gamma_a r \theta_l}{\gamma_a - (1-r)\beta_a \theta_l} & \frac{\beta_h \gamma_a r \theta_l}{\gamma_a - (1-r)\beta_a \theta_l} \\
0 & -c_{12}^B - \sigma_i & 0 & 0 \\
q \sigma_i & 0 & -\delta_h - \gamma_h & c_{12}^B \\
0 & q \sigma_i & 0 & -c_{12}^B - \delta_h - \gamma_h
\end{bsmallmatrix}.
\end{align}
It follows that, if Condition 2 (given by \eqref{cond:TC2_a}) holds, then $S-RP^{-1}Q$ is a Metzler matrix, but its eigenvalues are not easily computed. Here, too, Proposition 3.1 of \cite{godsend} can be applied to the Metzler matrix $S-RP^{-1}Q$ to enable the computation of its eigenvalues based (on computing the eigenvalues for a simpler two by two matrix), as follows:
\begin{align*}
    S-RP^{-1}Q = \begin{bmatrix}
        W & X \\
        Y  & Z 
    \end{bmatrix},
\end{align*}
where, 
\begin{align}\label{eq:W}
    W = \begin{bmatrix}
\frac{\beta_i \gamma_a r \theta_l}{\gamma_a -(1-r) \beta_a \theta_l} - \sigma_i & c_{12}^B + \frac{\beta_i \gamma_a r \theta_l}{\gamma_a -(1-r) \beta_a \theta_l} \\
0 & -c_{12}^B - \sigma_i
\end{bmatrix}, X = \begin{bmatrix}
\frac{\beta_h \gamma_a r \theta_l}{\gamma_a -(1-r) \beta_a \theta_l} & \frac{\beta_h \gamma_a r \theta_l}{\gamma_a -(1-r) \beta_a \theta_l} \\
0 & 0
\end{bmatrix},\\
Y = \begin{bmatrix}
    q\sigma_i & 0 \\
    0 & q\sigma_i
\end{bmatrix} \text{, and } Z = \begin{bmatrix}
    -\delta_h-\gamma_h & c^B_{12} \\
    0 & -c^B_{12}-\delta_h-\gamma_h
\end{bmatrix},
\end{align}
from which it follows that the matrix $S-RP^{-1}Q$ will be Metzler-stable if and only if the matrices $W$ and $Z-YW^{-1}X$ are Metzler-stable. The matrix $W$ is a principal sub-matrix of a Metzler matrix.  Hence, it is a Metzler matrix and its eigenvalues are given by:
\[\sigma(W)=\biggl\{\frac{\beta_i \gamma_a r \theta_l}{\gamma_a -(1-r) \beta_a \theta_l} - \sigma_i, - c_{12}^B - \sigma_i\biggr\},\]
which are strictly negative if the following condition (termed {\it Condition 3}) holds:
\begin{align}\label{cond:TC3_a}
    \text{Condition 3: }& \frac{r\beta_i\theta_l}{\sigma_i} + \frac{(1-r) \beta_a \theta_l}{\gamma_a} < 1
\end{align} 
So, if  Condition 3 (\eqref{cond:TC3_a}) holds, then the matrix $W$ is Metzler-stable. Here, too, the  matrix $Z-YW^{-1}X$ needs to be computed, as below:
\begin{align}
    Z-YW^{-1}X = \begin{bmatrix}
    \frac{\beta_h \gamma_a r q  \sigma_i \theta_l }{- \beta_i \gamma_a r \theta_l + \sigma_i(\gamma_a  -(1-r) \beta_a\theta_l)} - \gamma_h - \delta_h & c_{12}^B + \frac{ \beta_h \gamma_a r q \sigma_i \theta_l}{- \beta_i \gamma_a r \theta_l + \sigma_i(\gamma_a  -(1-r) \beta_a\theta_l)} \\
    0 & - c_{12}^B - \delta_h - \gamma_h
    \end{bmatrix}.
\end{align}
To prove that $Z-YW^{-1}X$ is a  Metzler matrix, the following inequality must hold:
\[c_{12}^B + \frac{ \beta_h \gamma_a r q \sigma_i \theta_l}{- \beta_i \gamma_a r \theta_l + \sigma_i(\gamma_a  -(1-r) \beta_a\theta_l)} > 0\]
as $c_{12}^B, \beta_h \gamma_a r q \sigma_i \theta_l > 0$, which can be simplified to:
\[- \beta_i \gamma_a r \theta_l + \sigma_i(\gamma_a  -(1-r) \beta_a\theta_l) > 0,\]
which follows directly from Condition 3 above. Hence, $Z-YW^{-1}X$ is a Metzler matrix. To prove that this matrix is Metzler-stable, consider, first of all, its eigenvalues, which are given by:
\[\sigma(Z-YW^{-1}X)=\biggl\{\frac{\beta_h \gamma_a r q  \sigma_i \theta_l }{\gamma_a \sigma_i - \beta_a \sigma_i \theta_l - \beta_i \gamma_a r \theta_l + \beta_a r \sigma_i \theta_l} - \gamma_h - \delta_h, - c_{12}^B - \delta_h - \gamma_h\biggr\}.\]
These eigenvalues will be strictly negative if the following condition (termed {\it Condition 4}) holds:
\begin{align}\label{cond:TC4_a}
    \text{Condition 4: }& \theta_l\biggl(\frac{\beta_h  r q  }{\gamma_h + \delta_h} + \frac{(1-r)\beta_a}{\gamma_a} + \frac{r\beta_i }{\sigma_i}\biggr) = \RR_c < 1.
\end{align}
Hence, if Condition 4 holds, the matrix $Z-YW^{-1}X$ is Metzler-stable. This (in conjunction with the Metzler-stability of the matrix $W$) implies that the matrix $S-RP^{-1}Q$ is Metzler-stable, which implies that, if Conditions 1 through 4 hold, the matrix $A_{22}$ is Metzler-stable. Thus (since $A_{22}$ is Metzler-stable), all eigenvalues of the matrix $A_{22}$ have negative real parts. Furthermore, it should be noted that Condition 4 (given by \eqref{cond:TC4_a}) also contains Conditions 2 and 3 (given by \eqref{cond:TC2_a} and \eqref{cond:TC3_a}, respectively). Thus, if Conditions 1 and 4 hold, the matrix $A_{22}$ is Metzler-stable. Furthermore, the eigenvalues of the Jacobian matrix evaluated at G1DFE are given by the union of the eigenvalues of $A_{11}$ and $A_{22}$. That is, 
\[\sigma(J\lvert_{G1DFE})=\sigma(A_{11})\cup \sigma(A_{22}).\]
Thus, if Conditions 1 and 4 hold, the matrix $A_{22}$ is Metzler-stable (meaning its eigenvalues have negative real parts) and the matrix $A_{11}$ has three eigenvalues with negative real parts and one eigenvalue identically equal to zero, so 11 of the 12 eigenvalues of the Jacobian $J\lvert_{G1DFE}$ of the behavior model \eqref{eq:n2model} evaluated around the G1DFE have negative real parts, and the remaining eigenvalue is identically zero.  Hence, the G1DFE is marginally-stable \cite{strogatz} whenever Conditions 1 (i.e., $c^B_{12}>c^B_{21}$) and 4 (i.e., $\RR_c<1$) hold. 
\end{proof}

\subsection*{C $\enspace$ Additional Simulation Figures}\label{sec:appendixC}
\stepcounter{section}
\renewcommand{\thesection}{C}
\renewcommand{\thefigure}{\thesection.\arabic{figure}}

\begin{figure}[H]
\centering
\textbf{(a)} \quad\quad\quad\quad\quad\quad\quad\quad\quad\quad\quad \textbf{(b)} \quad\quad\quad\quad\quad\quad\quad\quad\quad\quad\quad \textbf{(c)}\\
  \includegraphics[width=.9\linewidth,scale=1]{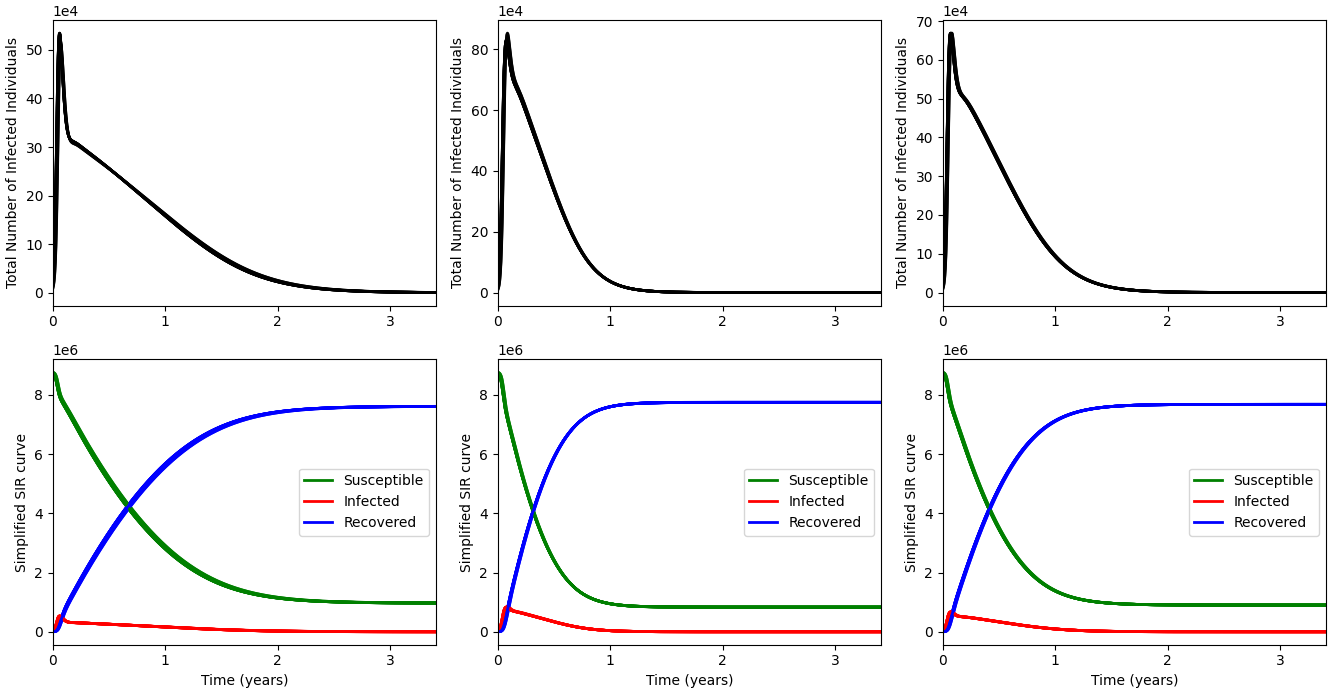}
  \caption{Simulations of the 2-group model \eqref{eq:n2model}, showing the profile of the total number of infected individuals (top) and the total number of individuals in the community, stratified by disease state (susceptible, infected, recovered), as a function of time for various initial conditions. Column (a): $c^B_{12}=0.5,c^B_{21}=0.25$ ($\Gamma=2>1$). Column (b): $c^B_{12}=0.25, c^B_{21}=0.5$ ($\Gamma=0.5<1$). Column (c): $c^B_{12}=c^B_{21}=0.5$ ($\Gamma=1$). In all panels all other parameter values used in these simulations are as given in Tables \ref{tab:fixedparams} and \ref{tab:fittedparams_full}, but with $\theta_l=1$, $\beta_a=1$, $\beta_i=0.5$, and $\xi=0$ (so that, $\RR_c=7.8>1$ and infection does not induce permanent immunity). These simulations show that when $\RR_c>1$ and infection does not induce permanent immunity, solutions of the 2-group model \eqref{eq:n2model} satisfy the Kermack-Mckendrick constraint $S(\infty)-S(0)>0$ \cite{gg_primer_gumel}, regardless of the value of $\Gamma$.}\label{fig:xi0}
\end{figure}

\begin{figure}[H]
\centering
\textbf{(a)}\hspace{80mm}\textbf{(b)}\\
  \includegraphics[width=85mm,scale=1]{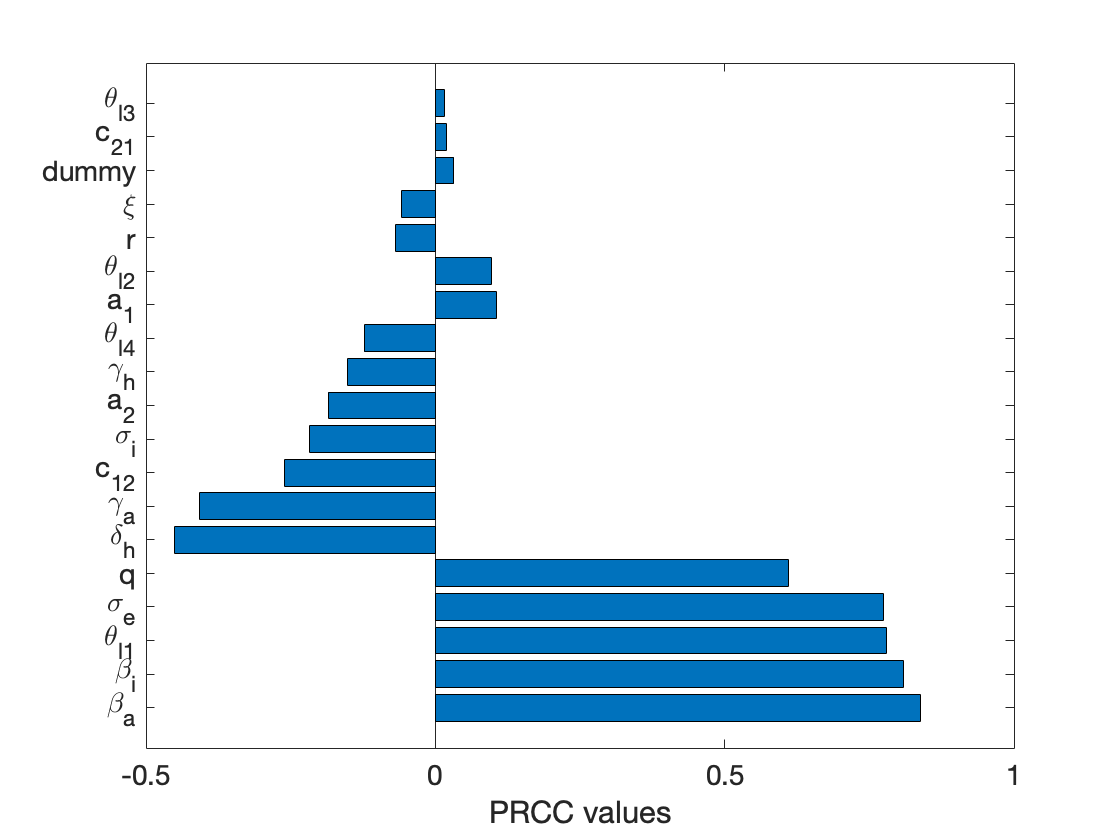}\includegraphics[width=85mm,scale=1]{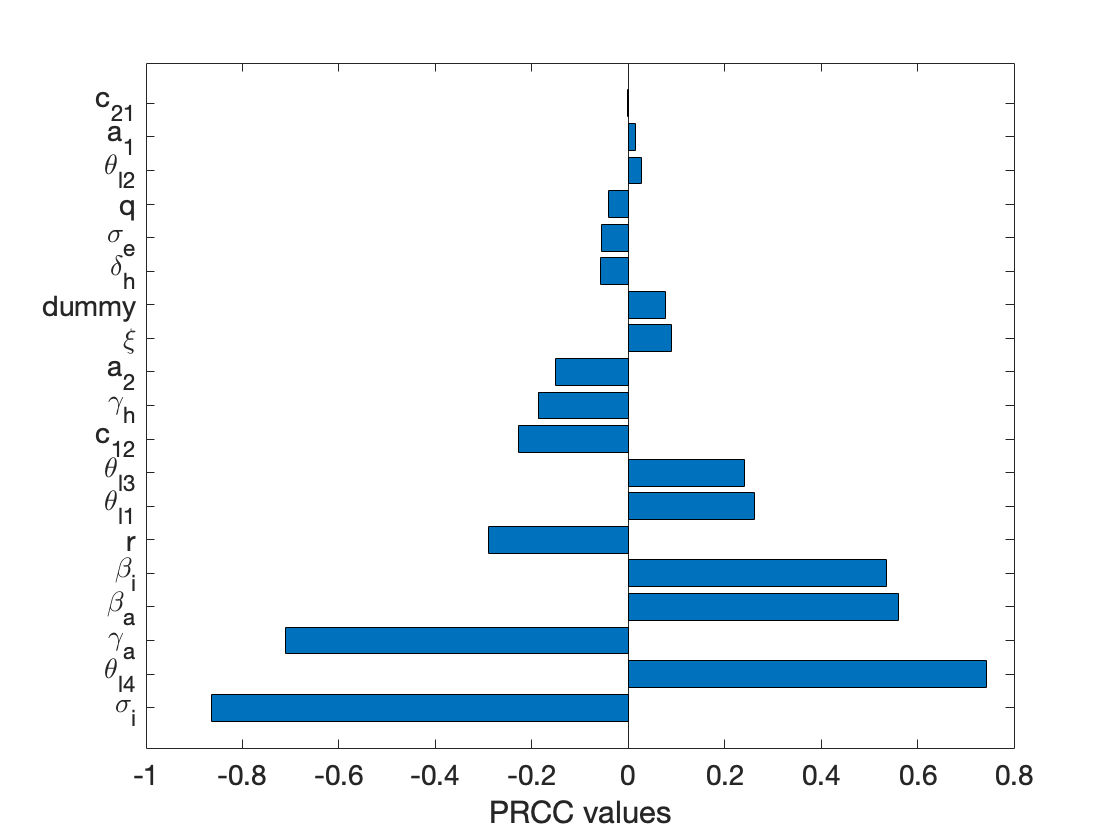}
  \caption{Partial Rank Correlation Coefficient (PRCC) values for the 18 parameters of the model \eqref{eq:n2model} with respect to peak daily hospitalizations at (a) April 2020 and (b) February 2021 using parameter intervals as given in Table \ref{tab:cube}.}\label{fig:prcc_hosp}
\end{figure}

\begin{figure}[H]
\centering
\textbf{(a)}\hspace{80mm}\textbf{(b)}\\
  \includegraphics[width=90mm,scale=1]{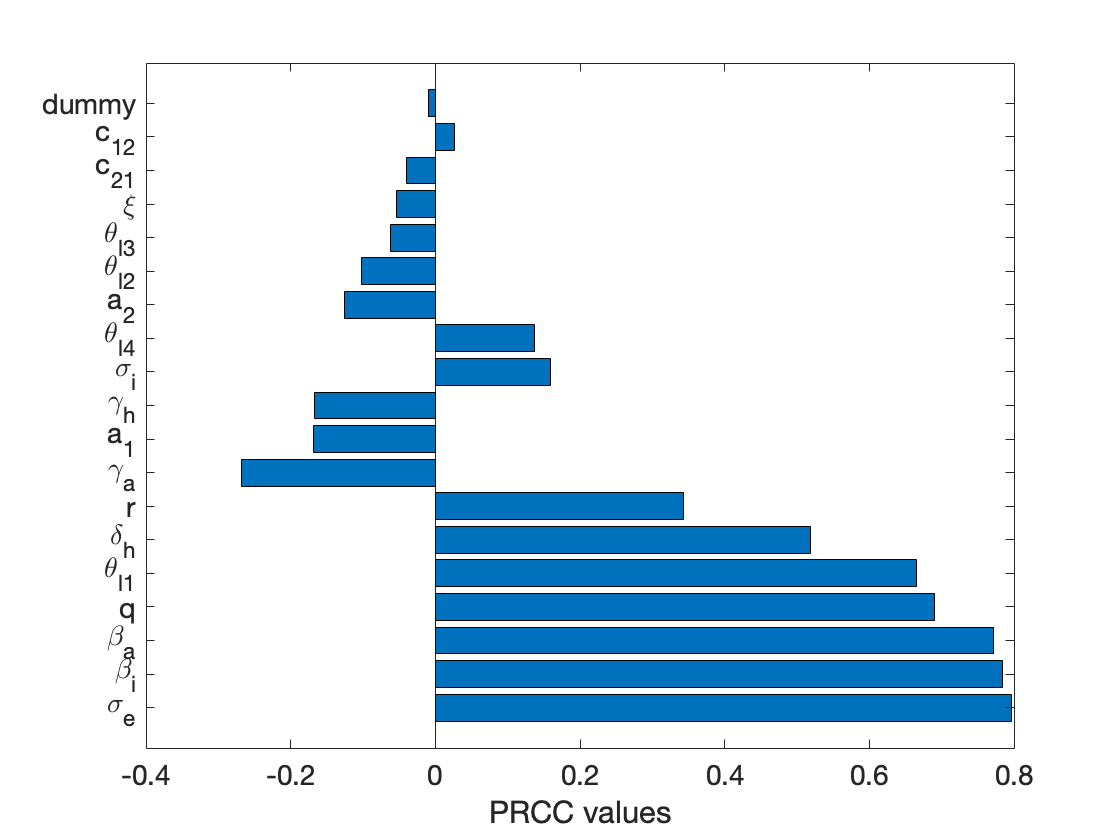}\includegraphics[width=90mm,scale=1]{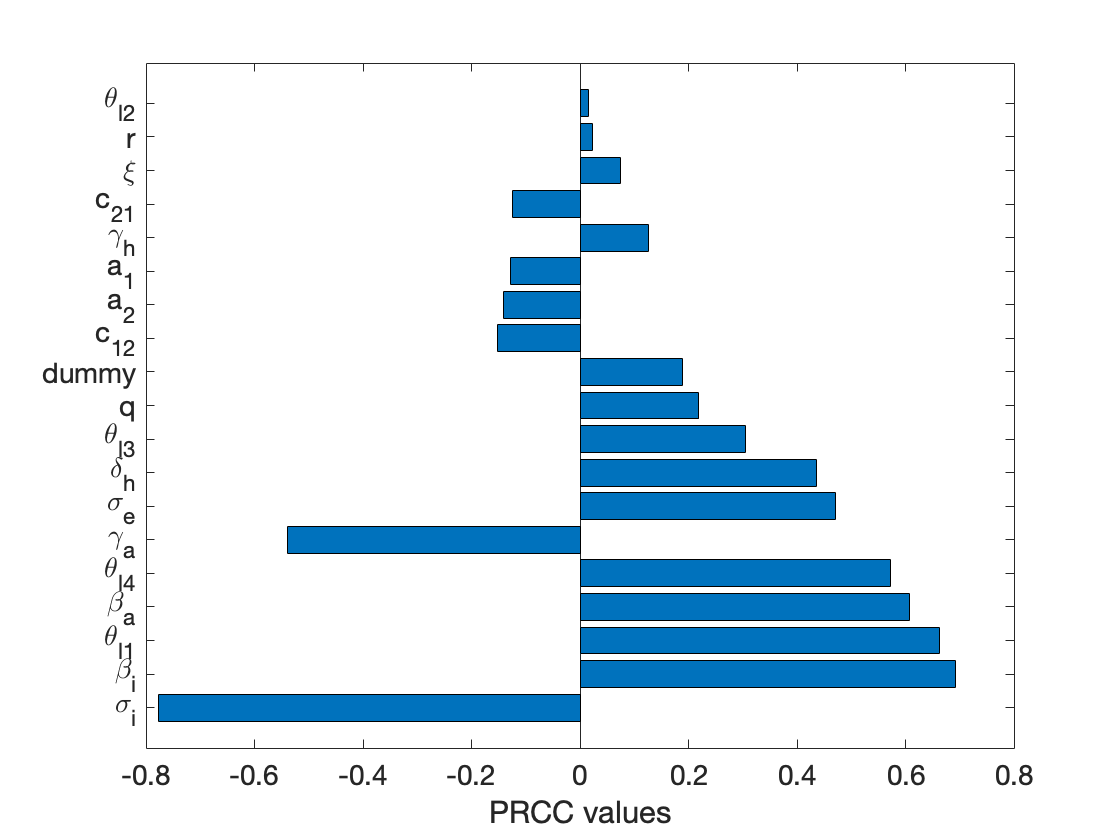}
  \caption{Partial Rank Correlation Coefficient (PRCC) values for the 18 parameters of the model \eqref{eq:n2model} with respect to cumulative mortality at (a) April 2020 and (b) February 2021 using parameter intervals as given in Table \ref{tab:cube}.}\label{fig:prcc_deaths}
\end{figure} 

\begin{figure}[H]
    \centering
    \includegraphics[width=0.65\linewidth]{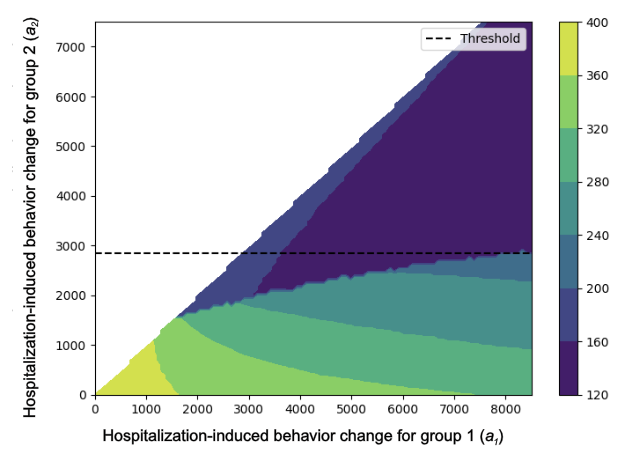}
    \caption{Effect of the hospitalization-motivated behavioral modifiers, $a_1$ and $a_2$, on the length (in days) of the second wave of the SARS-CoV-2 pandemic in New York City. Simulations of the model \eqref{eq:n2model} with the baseline parameter values in Table \ref{tab:fixedparams} and \ref{tab:fittedparams_full}, and various values of $a_1$ and $a_2$. The end of the second wave was determined by the first day of the simulation after the start of the second wave wherein the total number of infected individuals (including exposed, asymptomatically infectious, symptomatically infectious, and hospitalized individuals) fell below 200.}
    \label{fig:ai_w2len}
\end{figure}

\begin{figure}
    \centering
    \includegraphics[width=0.7\linewidth]{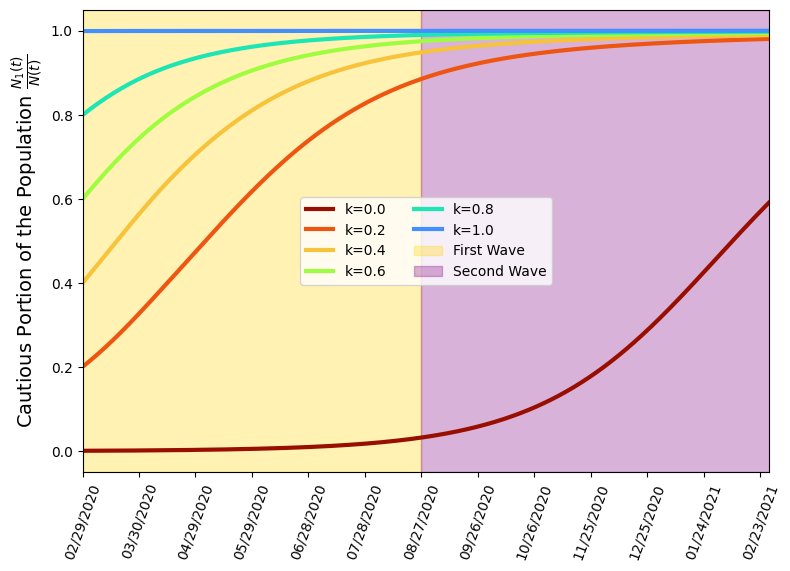}
    \caption{Effect of initial sizes of the two behavioral groups on the relative size of Group 1 ($\frac{N_1(t)}{N(t)}$)  during the first two waves of the SARS-CoV-2 pandemic in New York City. Simulations of the 2-group behavior model \eqref{eq:n2model} using the parameter values in Tables \ref{tab:fixedparams} and \ref{tab:fittedparams_full} and various initial sizes of the behavioral groups ($S_1(0) = kN(0)$ and $S_2(0) = (1 - k)N(0)$, for $k = 0, 0.2, 0.4, ..., 1)$.}
    \label{fig:n1_k}
\end{figure}

\begin{figure}[H]
    \centering
    \includegraphics[width=0.5\linewidth]{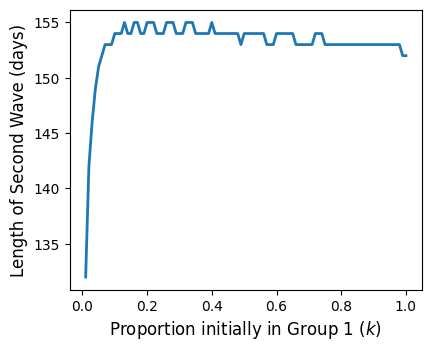}
    \caption{Effect of initial sizes of the behavioral groups on the length (in days) of the second wave of the SARS-CoV-2 pandemic in New York City. Simulations of the model \eqref{eq:n2model} with the baseline parameter values in Table \ref{tab:fixedparams} and \ref{tab:fittedparams_full} and various initial sizes of behavioral groups ($S_1(0)=kN(0)$ and $S_2(0)=(1-k)N(0)$ for $0\leq k\leq 1$). The end of the second wave was determined by the first day of the simulation after the start of the second wave wherein the total number of infected individuals (including exposed, asymptomatically infectious, symptomatically infectious, and hospitalized individuals) fell below 200.}
    \label{fig:ic_impact_w2l}
\end{figure}
\newpage
\bibliographystyle{unsrt}
\bibliography{references.bib}

\end{document}